\documentclass[a4paper,reqno]{amsart}
\pdfoutput=1
\usepackage[T1]{fontenc}
\usepackage{lmodern}

\usepackage{amsmath,amsfonts,amssymb,amsthm,stmaryrd,bbm,graphicx,mathtools,enumerate}
\usepackage{mathrsfs}
\usepackage[utf8]{inputenc}  

\usepackage[english]{babel}
\usepackage[tracking,spacing,kerning,babel]{microtype}
\usepackage{wasysym} 
\usepackage{color}
\usepackage{xcolor}
\usepackage{framed} 
\usepackage{caption} 
\usepackage{comment}
\usepackage{tikz}
\usetikzlibrary{arrows,shapes,positioning,quotes}
\usetikzlibrary{calc,decorations.markings}

\usepackage{wrapfig} 

\usepackage[normalem]{ulem}
\usepackage{soul}

\usepackage[final]{hyperref}   
\hypersetup{
    linktoc=page,
    linkcolor=red,          
    citecolor=blue,        
    filecolor=blue,      
    urlcolor=cyan,
   colorlinks=true           
}

\usepackage{lipsum} 

\usepackage[toc,page]{appendix}

\DeclareRobustCommand{\SkipTocEntry}[5]{}

\usepackage{multicol}

\newtheoremstyle{mine}
{\baselineskip}
{\baselineskip}
{\itshape}
{
}
{\bfseries}
{.}
{.5em}
{#1 #2\ifx#3\relax\else~(#3)\fi}

\theoremstyle{mine}

\newtheorem{theorem}{Theorem}
\numberwithin{theorem}{section}
\newtheorem{corollary}[theorem]{Corollary}
\newtheorem{proposition}[theorem]{Proposition}
\newtheorem{lemma}[theorem]{Lemma}

\newtheorem{definition}[theorem]{Definition}

\newtheorem{OP}{Open Problem} 
\numberwithin{equation}{section}

\theoremstyle{remark}
\newtheorem{remark}{Remark}

\theoremstyle{remark}
\newtheorem{question}{Question}

\colorlet{shadecolor}{blue!10}

\def\rm{\reversemarginpar}

\renewcommand{\epsilon}{\varepsilon}

\newcommand{\R}{\mathbb{R}}
\newcommand{\D}{\mathbb{D}}

\newcommand{\C}{\mathbb{C}}
\newcommand{\Z}{\mathbb{Z}}
\newcommand{\N}{\mathbb{N}}
\renewcommand{\S}{\mathbb{S}}

\newcommand{\1}{\mathbf 1}
\renewcommand{\d}{\mathbf d}

\def\T{\mathbb{T}}

\def\calC{\mathcal{C}}

\def\calE{\mathcal{E}}

\def\calL{\mathcal{L}}

\def\calX{\mathcal{X}}

\newcommand{\caF}{{\mathcal F}}

\newcommand{\caX}{{\mathcal X}}
\newcommand{\caY}{{\mathcal Y}}

\def\sle{\mathrm{SLE}}
\def\cle{\mathrm{CLE}}
\def\diam{\mathop{\mathrm{diam}}}
\def\im{\mathop{\mathrm{im}}}

\def\Var#1{\mathrm{Var}\bigl[ #1\bigr]}
\def\cov{\mathrm{Cov}}

\def\dist{\mathrm{dist}}

\def\iso{\xrightarrow[]{\sim}}

\newcommand{\haus}{\mathrm{Haus}}

\def\P{\mathbb{P}} 
\def\E{\mathbb{E}} 
\def\md{\mid}

\def \eps {\epsilon}

\def\Bb#1#2{{\def\md{\bigm| }#1\bigl[#2\bigr]}}
\def\BB#1#2{{\def\md{\Bigm| }#1\Bigl[#2\Bigr]}}

\def\Pb{\Bb\P}
\def\Eb{\Bb\E}

\def\EB{\BB\E}

\def\FK#1#2#3{{\def\md{\bigm| } \P_{#1}^{\,#2}  \bigl[  #3 \bigr]}}

\def \p {{\partial}}

\def\<#1{\langle #1\rangle}

\definecolor{darkgreen}{rgb}{0,0.6,0.05}

\def\Wick#1{\,\colon\!\! \, #1 \, \!\colon}

\def\bi{\begin{itemize}}  
\def\ei{\end{itemize}}
\def\bnum{\begin{enumerate}} 
\def\enum{\end{enumerate}}
\def\ni{\noindent}
\def\bf{\bfseries}

\def\free{\mathrm{free}}

\def\dimM{\overline{\dim}_M}

\def\nes{\mathsf{n}}
\def\Expl{\mathsf{Expl}}

\title[The magnetization field can be reconstructed from the nested--$\cle_3$]
{The critical Ising magnetization field can be reconstructed from its $+/-$ interfaces}
\author{Paul Cahen}
\author{Christophe Garban}
\author{Avelio Sep\'{u}lveda}
\date{}

\address
{Universit\'e Claude Bernard Lyon 1, CNRS UMR 5208, Institut Camille Jordan, 69622 Villeurbanne, France}
\email{cahenpaul25@gmail.com}

\address
{Universit\'e Claude Bernard Lyon 1, CNRS UMR 5208, Institut Camille Jordan, 69622 Villeurbanne, France, and Courant Institute (NYU), New York, USA}
\email{christophe.garban@gmail.com}

\address{Universidad de Chile,  Centro de Modelamiento Matemático (AFB170001), UMI-CNRS 2807, Beauchef 851, Santiago, Chile.}
\email{lsepulveda@dim.uchile.cl}

\begin{document}

\begin{abstract}

The critical Ising model admits several natural scaling limits. For an Ising model with $+$ boundary conditions on a lattice approximation $D^a$ of a planar domain $D$, the rescaled spin field converges, as $a\to0$, to the critical Ising magnetization field (IMF), a rough random distribution on $D$, while the spin interfaces converge to a nested $\cle_3$ and the FK--Ising representation converges to a colored $\cle_{16/3}$.
	
	In this work, we prove that the IMF can be directly reconstructed from the nested $\cle_3$. This complements the previously known construction of the IMF from the colored $\cle_{16/3}$ \cite{Camia_2015}. Our result also implies that if one constructs a colored $\cle_{16/3}$ from the nested $\cle_3$ through CLE percolation \cite{miller2017cle}, the resulting construction recovers the same magnetization field.
	
	A main difficulty in this reconstruction comes from the fact that the renormalization exponent of the IMF, $15/8$, is smaller than the Hausdorff dimension of the $\cle_3$ carpet, $187/96$. Consequently, the field cannot be recovered by a direct Minkowski-content construction from the nested $\cle_3$ as is the case for $\cle_{16/3}$.
	
	We conjecture that the converse measurability also holds, namely that the nested $\cle_3$ can itself be reconstructed from the IMF.

%
%
%

\end{abstract}

\maketitle

\begin{figure}[!htp]
	\centering  
	\begin{tikzpicture}[
		node distance = 1cm and 2cm,
		box/.style = {
			rectangle, draw, rounded corners,
			minimum width=4cm,
			minimum height=1.5cm,
			align=center
		},
		base arrow/.style = {
			very thick,
			-stealth,
			shorten >=6pt,
			shorten <=6pt
		},
		arrow/.style = {base arrow},
		BIDIRarrow/.style = {
			very thick,
			<->,
			shorten >=6pt,
			shorten <=6pt
		},
		DASHEDarrow/.style = {base arrow, dashed},
		REDarrow/.style = {base arrow, draw=red}
		]
		
		\node (a1) [box] {nested--$\cle_3$};
		\node (a2) [box, right=3.8cm of a1] {Magnetization Field $\Phi$};
		\node (blank) [below=3cm of a1] {};
		\node (a3) [box, right=2cm of blank] {colored--$\cle_{16/3}$};
		
		\draw[REDarrow] 
		([yshift=0.3cm]a1.east) -- ([yshift=0.3cm]a2.west) 
		node[midway, above, text=red] {THIS PAPER}; 
		
		\draw[DASHEDarrow] 
		([yshift=-0.3cm]a2.west) -- ([yshift=-0.3cm]a1.east) 
		node[midway, below] {Open Problem \ref{c.main}};
		
		\draw[arrow] 
		(a3) to[
		"CLE percolation\\(\cite{miller2017cle} and \cite{kohler2024fuzzy})"
		{align=center, left}
		] (a1);
		
		\draw[arrow] 
		(a3) to["\cite{Camia_2015}"{align=center, right}] (a2);
		
	\end{tikzpicture}
	
	\caption{Solid arrows represent proven measurability results. Our main result is the top arrow in red. We also show in this paper that the lower path (i.e. going through the colored-$\cle_{16/3}$ box) can be followed as well; however, this result is obtained as a corollary of our measurability theorem. Let us remark that when going from nested-CLE$_3$ to the colored CLE$_{16/3}$, additional randomness is needed.
}
	\label{f.1}
\end{figure}
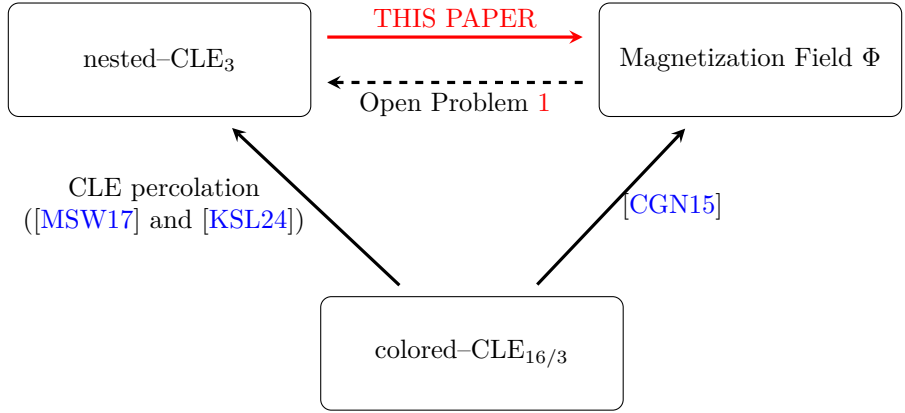

\tableofcontents

\section{Introduction}

\subsection{Context.}
In this paper we are interested in the interplay between different ways of taking the scaling limit of a critical Ising model in a domain $D\subset \C$ with a sufficiently regular boundary. If $D$ is such a domain, then any scaling limit of the critical Ising model in $D$ starts with an Ising model defined on a lattice approximation $D^a$ of $D$ with small mesh size $a$: for example $D^a:=a\Z^2 \cap D$ or $D^a:=a\T \cap D$ where $\T$ stands for the triangular lattice. An Ising spin configuration $\sigma=\sigma_D^a$ is a random element of $\{\pm 1\}^{D^a}$ sampled according to the Gibbs measure
\begin{align*}\label{}
\FK{D^a, \beta}{}{\sigma} \propto \exp\big( \beta \sum_{i\sim j \in D^a}  \sigma_i \sigma_j \big)
\end{align*}
where $\beta$ is the inverse temperature. A well-known symmetry breaking arises when $\beta$ exceeds  a critical inverse temperature $\beta_c$ which depends on the underlying microscopic lattice used (for example $\beta_c= \frac 1 2 \log(1+\sqrt{2})$ for $\Z^2$). In relation with this symmetry breaking, it is of great interest to prescribe different types of {\em boundary conditions} around the boundary $\p D$. The most prevalent ones are the $\{+,-,\free\}$ boundary conditions.  In this paper, we will work with  $+$ boundary conditions. For precise definitions and for a modern account of the planar Ising model, we refer to \cite{duminil2017lectures,friedli2017statistical}.  
\smallskip

For each $a>0$, we may sample a critical Ising model $\sigma_D^a \sim \P_{D^a,\beta_c}$. A natural question is to ask what happens to $\sigma_D^a$ as $a\searrow 0$? Does it have an interesting scaling limit? The answer to this question is surprisingly rich and depends on the information we choose to keep when we pass to the limit. The information we keep is encoded in the topology used to pass to the limit. Here are some celebrated examples of scaling limit results for the planar critical Ising model. 
\bnum

\item {\em $\cle_3^\nes$ and scaling limit of the contour process.} It is easier to describe this convergence in the case of the microscopic triangular lattice $a\T$. Sample an Ising configuration $\sigma_D^a$ on the graph $D^a = D \cap a \T$,  with $+$ boundary conditions. On the hexagonal lattice, dual of $a\T$, the set of edges which separate a $+$ spin from a $-$ spin forms a set of closed loops in $D^a$, called the \emph{(discrete and nested) $\cle_3$}; we will denote it by $\cle_3^{\nes,a}$. As $a\to 0$, the set of closed loops is known to converge to a non-trivial random countable family of fractal curves, the $\cle_3^\nes$ process, i.e. the nested--$\cle_3$. 
See \cite{benoist2019scaling} as well as Sections \ref{ss.CLEs} and \ref{ss.CLE3} below for the appropriate choice of topology and the precise statement. If we keep track only of the {\em first layer} loops separating the $+$ spins connected to the boundary from the first layer of $-$ spins, 
we end up with a subset of $\cle_3^\nes$, usually called the \emph{simple} $\cle_3$ in the literature, which we shall denote by $\cle_3^1$.

\item {\em $\cle_{16/3}$ and scaling limit of the FK--Ising percolation.} It is well known that many properties of  the Ising model are naturally encoded by a dependent percolation model called the FK--Ising percolation {(corresponding to FK percolation with cluster weight $q=2$)}. We refer here also to \cite{duminil2017lectures, friedli2017statistical} for background. The main and most useful connection between these two models is the identity 
\begin{align*}\label{}
\<{\sigma_x \sigma_y}_{G,\beta} = \FK{p(\beta)}{}{x \longleftrightarrow y},
\end{align*}
{where $p(\beta)=1-e^{-2 \beta}$.}
If one considers the FK--Ising model on $D^a = a \Z^2 \cap D$, it is known that the large connected clusters have a scaling limit as $a\to 0$ (for a suitably chosen topology, for example the one described in Section \ref{ss.CLEs} or the so-called Schramm-Smirnov topology as introduced in \cite{Schramm_2011}). The limit of this percolation is called  a $\cle_{16/3}$ process and if one keeps track of the (independent, balanced) coin flip assigned to each cluster, this is called the {\em colored}-$\cle_{16/3}$ process.

\item\label{it:magn} {\em The critical magnetization field.}  It is tempting to collect the information given by all Ising spins $\{\sigma_x^a\}_{x\in D^a}$ within one random Schwartz distribution (suitably rescaled). This is the viewpoint considered in \cite{Camia_2015,camia2016planar} where the following is proved: the random Schwartz distribution $\Phi^a:= a^{15/8} \sum_{x\in a \Z^2\cap D} \sigma_x \delta_x$ converges in $H^{-3}$ to a random distribution on $D$ called the critical magnetization field $\Phi$. Furthermore, in \cite{Camia_2015}, this non-Gaussian field is proved to be measurable w.r.t. item (2) above, namely the colored $\cle_{16/3}$ process.   See Subsection \ref{ss.IMF} for more background on this field. 

\enum

\bnum
\setcounter{enumi}{3}

\item {\em $n$-point correlation functions.} There are several of these: the magnetization, the energy, mixed versions of these etc.  For example the $n$-point correlation function of the magnetization field, say in the full plane $\C$, corresponds to the function $z_1,\ldots, z_n \in \C^n \mapsto \<{\sigma_{z_1} \ldots \sigma_{z_n} }_{\C, \beta_c}$ defined as the limit when the mesh $a \to 0$ of  $a^{- n/8}\<{\sigma_{z_1^a} \ldots \sigma_{z_n^a}}_{a\Z^2, \beta_c}$ where $z_k^a \in a \Z^2$ and $|z_k^a-z_k| \leq a$ for $k\in\{1,\ldots,n\}$. Notice the renormalization factor $a^{-n/8}$ which ensures that the limiting function is non-trivial.

We refer for example to \cite{chelkak2015conformal, chelkak2023universality} for convergence and universality results on these $n$-point functions. 
 
\item {\em The critical Energy field.} As mentioned in the previous item, several $n$-point correlation functions of interest have been proved to converge for the critical Ising model. This goes back to the success of {\em conformal field theory} which in some sense computed and classified such correlation functions. Next to the magnetization $n$-point correlation functions are the $n$-point correlation functions of the {\em Energy field} which have a very nice explicit limit in $2d$ as proved in
\cite{boutillier2010critical,boutillier2011critical,giuliani2012scaling,hongler2013energy,chelkak2021correlations,izyurov2024energy}. 
As in item \ref{it:magn} above, it is a natural question to collect all the energy variables indexed by the edges of $D^a$ into one (suitably rescaled) random Schwartz distribution $\calE^a$. The work \cite{garban2025energy} suggests that in the limit $a\to 0$ , such a field $\calE$, viewed as a random distribution, does not exist. 

\item {\em The height function of the double current model.} In \cite{lopez2026ising} a new representation of the discrete Ising model is introduced, analogous in spirit to the Edwards-Sokal coupling, but based on the double random current model. Combining this representation with the convergence of the associated height function to the GFF \cite{DCLQ1,DCLQ2} yields a joint scaling limit coupling the GFF with the critical magnetization field. In this coupling, two-valued sets associated with the GFF can be used to construct the critical magnetization field together with an associated XOR-Ising model. This coupling is furthermore compatible with the bosonisation relation identifying the associated XOR--Ising fields with the sine or cosine of the GFF \cite{JSW,AlS}. 
\enum

By combining the seminal works \cite{chelkak2012universality, chelkak2014convergence, chelkak2015conformal, benoist2019scaling,  chelkak2023universality}, it is known that none of these scaling limits depends on the underlying lattice used ($a\Z^2$ or $a\T$ or any other ``reasonable'' choice of microscopic lattice). 

To date, the interplay between these different scaling limits is only partially understood. It is briefly summarized in Figure \ref{f.1} at the beginning of this paper (see also Section \ref{se.Preliminaries} for more details). The main goal of this paper is to strengthen our understanding of the link between some of these notions of scaling limits, answering a question in \cite{chelkak2018planar}.

\subsection{Main result.}

As already stated in Figure \ref{f.1}, our main result is that we can construct the critical Ising magnetization field (IMF) as a measurable (and explicit) function of the nested $\cle_3$.
\begin{theorem}\label{th.main}
Let $D$ be a Jordan domain. For the critical Ising model on
$D\cap a\T$ equipped with $+$ boundary conditions, the joint coupling $(\cle_3^{\nes,a}, \Phi^a)$ converges in law as the mesh $a\to 0$ to the continuum coupling $(\cle_3^{\nes}, \Phi)$ in the topology of $\caX_D\times H^{-2}(D)$\footnote{See Sections \ref{ss.CLEs} and \ref{ss.Sobolev} for the definitions of the topologies of $\caX_D$ and $H^{-2}(D)$.} where the magnetization field $\Phi$ is a measurable function of the nested $\cle_3$. We will denote this property by 
\begin{align*}
\Phi = \Phi(\cle_3^\nes).
\end{align*}
Here and below, the superscript $\nes$ indicates that we are considering the nested collection of loops.
\end{theorem}

We refer to Section \ref{se.Preliminaries} and more specifically Subsections \ref{ss.CLEs} and \ref{ss.Sobolev} for a formal definition of the spaces where the $\cle_3^\nes$ and the magnetization field live, and the topologies on these spaces used for the convergence. The theorem above holds without any assumption on the regularity of the domain $D$\footnote{Besides the fact that it is Jordan, an assumption that \cite{benoist2019scaling} used to prove their main result on which our theorem relies.}. However, in the case of a piecewise smooth boundary, or even a boundary whose roughness is controlled, one may show the convergence to $\Phi$ in a slightly stronger topology. See Section \ref{s.RoughBoundary} and Remarks \ref{rk.reason-Omega}, \ref{rk.reason-Omega-2} for more details. 


\medskip

Theorem \ref{th.main} is not the first result in the literature which discusses the link between two scaling limits (i.e. two topologies) for one probabilistic object.
We emphasise in the list below other such fundamental examples in probability theory for which several scaling limits exist under different natural choices of topologies.
\bnum
\item {\em Critical percolation and 2d white noise.} If one considers critical site percolation with monochromatic boundary on a triangular lattice $a\T\cap D$ of mesh $a>0$ inside a Jordan domain $D$, then one may either view it as a discrete white noise (by viewing it as a random point $\mathrm{WN}^a$ in a Sobolev space $H^{-1-\eps}(\R^2)$) or instead as a collection of interfaces between open and closed sites (by viewing it as a discrete nested--$\cle_{6}$ process).  In this case, it is now well known (see \cite{benjamini1999noise,Schramm_2011,garban2014noise}) that 
\begin{align*}\label{}
(\mathrm{WN}^a, \mathrm{nested}-\cle_6^a) \overset{(law)}\longrightarrow (\mathrm{WN},\cle_6)
\end{align*}
but this time as opposed to our main result for Ising (Theorem \ref{th.main}) $\mathrm{WN}$ and $\cle_6$ are now independent of each other (even though they are obviously deterministic functions of each other in the discrete). This is reminiscent of the notion of {\em noise sensitivity} \cite{benjamini1999noise,garban2011oded,garban2014noise}.

\item {\em (nested)-$\cle_4$ and Gaussian free field}. 
If one now considers the discrete GFF on $D^a:=D \cap a\T$, or rather its cable extension (as  introduced in \cite{lupu2016loop,lupu2018convergence}), then we naturally end up with a coupling of a distribution field $\Gamma^a$ (the discrete cable GFF) and a discrete nested--$\cle_4^a$ equipped with independent labels. It is known from \cite{lupu2018convergence,aru2019bounded,QW,ALS2,ALS4} that
\begin{align*}\label{}
(\Gamma^a, \text{nested, labelled--}\cle_4^a) \overset{(law)}\longrightarrow (\Gamma,\text{nested,labelled--}\cle_4)
\end{align*}
and furthermore both components of the limiting coupling are measurable functions of each other! I.e.
\begin{align*}\label{}
& \Gamma   = \Gamma(\text{nested, labelled--}\cle_4) \\
& \text{ and } \\
& \text{nested, labelled--}\cle_4  = \text{nested, labelled--}\cle_4(\Gamma)
\end{align*}
We believe the same structure holds for $(\Phi, \text{nested--}\cle_3)$, the limiting Ising coupling (without any labels on the $\text{nested--}\cle_3$).  Our main Theorem provides the first measurability identity for Ising while the second, i.e. $ \text{nested--}\cle_3
\overset{?}= \text{nested--}\cle_3(\Phi)$ is still open, see Open Problem \ref{c.main} below.

\item {\em Weakly asymmetric simple exclusion process and space-time white noise.} The celebrated KPZ equation is written as the following potentially ill-defined SPDE in space/time $\p_t h = \p_{xx} h \, + \Wick{(\p_x h)^2}  + \xi$ (where $\xi$ is a space-time white noise on $\R\times \R$). Before the breakthrough works 
\cite{hairer2013solving, gubinelli2017kpz} it was not clear whether $h$---which can be defined by the indirect Cole-Hopf transform---would be measurable or not with respect to $\xi$.  It is now a consequence of these works.

\item {\em KPZ fixed point and space-time white noise.}

In the breakthrough works \cite{matetski2021kpz,dauvergne2022directed}, the so-called {\em KPZ fixed point} was constructed. Another natural stochastic process happens to live at the same $[1:2:3]$ scale, namely the underlying space-time white noise which is driving the KPZ fixed point (at least when the discrete mesh is still present). One may then ask about the joint scaling limit and the situation seems to be similar to what happens with critical $2d$ percolation, i.e. {\em noise sensitivity} and asymptotic independence between the two limits. See the recent works \cite{ganguly2024stability, himwich2024directed,ahlberg2026noise}. 

\item {\em Planar maps, Brownian surface and LQG.} The description of the limit of a uniformly chosen planar map can be done using different representations. The limit of the Schaeffer bijection was first found in \cite{MM} and defined as the Brownian map, then as a consequence of the works \cite{Mier,LG} it was shown that this can be seen as a random metric space which is a function of the continuum Schaeffer bijection.
A rather different bijection and topology (the peanosphere topology) were introduced in \cite{Sheffield} to study FK-decorated planar maps and keep track of both the percolation clusters and the counting measure on the maps. It was then shown in  \cite{DMS,MSBrownianMap} that the peanosphere limit of uniformly chosen maps (i.e. with FK parameter equal to 1) also determines its metric structure.

\enum

\subsection{Strategy of proof.}\label{ss.Strategy Proof}

Our main goal is to construct the magnetization field from the set of interfaces $\Gamma\sim \cle_3^\nes$, in other words, for a given test function $f$, to understand the conditional expectation
\begin{align*}\label{}
\Eb{\<{\Phi, f} \md \Gamma}.
\end{align*}

Since this is too complicated to define right away, it is natural to first focus on the conditional expectation w.r.t. the first layer $\Gamma^1 \sim \cle_3^1$
\begin{align*}\label{}
\Eb{\<{\Phi, f} \md \Gamma^1}
\end{align*}
for which we have a natural candidate: 
\begin{equation}\label{eq.attempt}
H^1_f(\Gamma^1):= -\int_{D\setminus \Gamma^1}   f(z) \,  \,  CR_{D \setminus \Gamma^1}(z)^{-1/8} dz\,, 
\end{equation}
where $z \mapsto CR_{D \setminus\Gamma^1}(z)$ is the conformal radius of the open set $D\setminus\Gamma^1$ seen from the point $z$ (this conformal radius only depends on the connected component of $D\setminus \Gamma^1$ which contains $z$) (N.B. The superscript $1$ in the notation $H_f^1$ is present to highlight that we want a conditional expectation w.r.t. the \emph{first} layer.)
\smallskip

\par 
Unfortunately, this candidate has no chance to be the right one. Indeed, assume this expression is accurate. Then for any $f:D\to \R$ which is positive, we would have the inequality $\E[\<{\Phi,f}]>0$. However a.s. $H^1_f(\Gamma^1)<0$, leading of course to a contradiction. The problem comes from the fact that the carpet of the $\cle_3$ does carry some mass of the field which is not taken into account in the ``attempted definition'' \eqref{eq.attempt}. This mass can be seen to be infinite from the dimension inequality \begin{align*}
	{187}/{96}=\dim(\{z\in D: z\text{ not surrounded by any }\ell\in\Gamma^1\})>15/{8},
\end{align*} where $15/8=2-1/8$ is the renormalization exponent of the magnetization field.

We then expect the correct conditional expectation to be some  suitably renormalized 
\begin{align*}\label{}
H^1_f(\Gamma^1):= ``  -\int_{D}   f(z) \,  \,  CR_{D \setminus \Gamma^1}(z)^{-1/8} dz\,  \ + \ \infty. " 
\end{align*}
The need for such an infinite renormalization can be seen as a consequence of the infinitely many small loops discovered in the exploration which together carry an infinite mass.

\begin{remark}
Though it is not clear how to define it directly, it will be a corollary of our work that we can give an explicit description of $\Eb{\<{\Phi,f}\md\Gamma^1}$, see Theorem \ref{th.IMFwrtFirst-layer}. 
\end{remark}

The need for such a renormalization is where a smooth curve $\eta$ enters the scene: it allows us to remove many divergences in our above attempt~\eqref{eq.attempt} while still keeping non-trivial information on $\<{\Phi, f}$. Indeed, instead of taking \emph{all} the loops of the first layer in $\Gamma^1$ we will restrict to those intersecting $\eta$. There are far fewer such loops (see Figure \ref{f.simulation} for a simulation of this) which makes it plausible that the small loops don't contribute too much to the conditional expectation. Thus, it is possible to define a candidate $H_{f,\eta}$ for the conditional expectation of $\<{\Phi,f}$ w.r.t. this exploration. See \eqref{eq.H f eta} and Proposition \ref{pr.Def H f eta} stating that $H_{f,\eta}$ is indeed well defined.

\begin{figure}[h!]
\begin{center}
\includegraphics[width=0.9\textwidth]{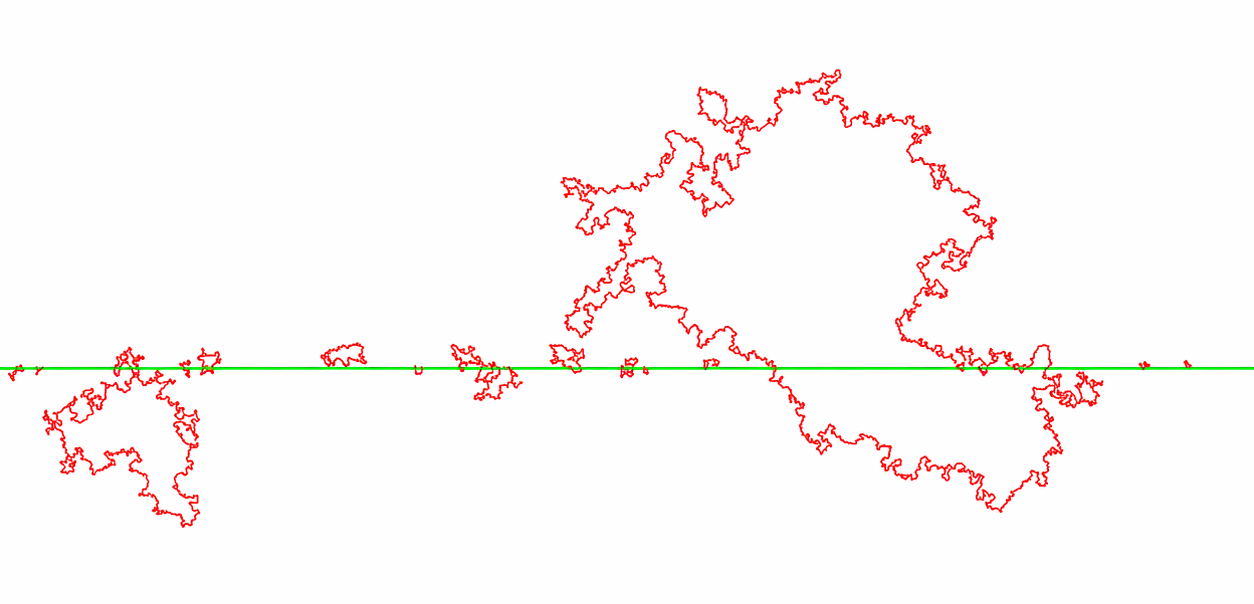}
\caption{Simulation of the loops of a discrete $\cle_3^{1,a}$ which intersect a vertical line $\eta$.}\label{f.simulation}
\end{center}
\end{figure}

In fact, it is somehow simpler to explore all the interfaces intersecting the line instead of just the first layer. Then the proof of Theorem \ref{th.main} may be summarized as follows:
\bnum
\item First, if $\Gamma(\eta)$ denotes all the loops of the nested-$\cle_3$ which intersect $\eta$ (in Figure \ref{f.simulation}, only the loops of the first layer are pictured), we introduce an explicit candidate for the conditional expectation $\Eb{\<{\Phi, f} \md \Gamma(\eta)}$, which is measurable w.r.t. $\Gamma$. 
By coupling Ising models on small mesh grids so that they converge to the continuum nested $\Gamma\sim \cle_3^\nes$,  we show that, as the mesh $a$ goes to zero, the discrete analog of the above conditional expectation converges in probability to this candidate.  One technical difficulty is controlling the contribution coming from the possible large oscillations of the magnetization field $\Phi$ close to $\eta$. This is the first main part of the proof handled in Subsections \ref{ss.Skorokhod} to \ref{ss.Proof PhiCondGamma(Eta)}.   
\item Finally, the second step of the proof is to show that  the conditional expectation $\Eb{\<{\Phi, f} \md \Gamma(\eta)}$  -- once well defined --  must be close to $\<{\Phi, f}$ (in $L^2$ norm, say) at least when the mesoscopic grid $\eta$ is small enough. This boils down to showing that the unexplored spin field in $D \setminus \Gamma(\eta)$ contains little macroscopic information.  This crucial step was not mentioned so far and involves in a key way the GHS inequality. This is handled in Subsection \ref{ss.InfoHoles}. 
\enum


\bigskip
\ni
\textbf{A natural but unsuccessful approach.} Another tempting strategy of proof would be to use CLE percolation: indeed, it is known from \cite{miller2017cle, kohler2024fuzzy} that the interfaces of the $+/-$ clusters and the interfaces of the FK--Ising clusters converge jointly in law in the Edwards-Sokal coupling, to a continuous coupling between the $\cle_3$ and the $\cle_{16/3}$. The same holds if the clusters of the FK--Ising percolation remember their Ising-sign, an object we call a colored $\cle_{16/3}$. Furthermore, the original proof of the convergence of the Ising Magnetization Field (IMF) in \cite{Camia_2015} actually shows that it is a measurable function of these colored continuous FK--Ising clusters.
\par In order to prove the measurability of the field in the $\cle_3$, one may try to show that in this coupling $\big(\Gamma,\omega^c,\Phi\big)$ (respectively the $+/-$ interfaces, the FK--Ising clusters with their sign, and the magnetization field) the last term does not depend on the additional randomness used to construct $\omega^c$ conditionally on $\Gamma$. Unfortunately, \textbf{there is a deep gap there}: see Section \ref{se. limits of literature} for a discussion on this approach. In particular Open Question \ref{op la loi du cle(16/3) conditionnellement au cle(3) vient du discret} would close such a gap. 
\par However, let us emphasise that our proof implies the following partial answer: our theorem implies that this path indeed works, \emph{a posteriori}: see Corollary \ref{co. the road through cle percolation gives a measurable cle_3}.

\subsection{Main open question raised by our work.}
We conclude this introduction by stating the following open problem about the measurability. 
\begin{OP}\label{c.main}
Our main result in this work is that $\Phi = \Phi(\mathrm{nested}-\cle_3)$. 
Prove the other direction, namely, that the nested--$\cle_3$ is a measurable function of the magnetization field $\Phi$. 
Note that this corresponds to the hard direction in the analogous situation of mutual measurability between nested, labelled--$\cle_4 \longleftrightarrow $ GFF (see the references \cite{dubedat2009sle,schramm2013contour,aru2019bounded}). 
\end{OP}

\subsection{Organization of the paper.}

Section \ref{se.Preliminaries} introduces the relevant objects, tools and theorems that we need for our paper. Section \ref{se.IMFwrtCurve} is devoted to the definition of $H_{f,\eta}$ and the proof that it is indeed the correct conditional expectation. 
Section \ref{se.Proof-Main-thm} concludes the proof of the main Theorem. It deals in particular with the amount of information left in the small unexplored regions.
Section \ref{s.RoughBoundary} strengthens our main Theorem, under an assumption on the boundary of $D$ by showing the convergence of the discrete fields in a stronger topology.
Section \ref{se.MainThm-consequences} proves a Markov Property for the IMF, and describes the conditional expectation $\Eb{\<{\Phi,f}\md\Gamma^1}$.
Section \ref{s.NearC} discusses the possible consequences of this work for the singularity of the near-critical limit (in temperature) and Section \ref{se. limits of literature} explains why our main Theorem does not follow easily from the theory of CLE percolation.


\medskip
\ni
{\bf Acknowledgments.} 
We wish to thank Roland Bauerschmidt, Gefei Cai, Jean-Christophe Mourrat and Jason Miller for very useful discussions.

The three authors acknowledge support from the ERC grant VORTEX 101043450. 
C.G. also acknowledges support from the Institut Universitaire de France (IUF) and 
A.S. also acknowledges support from Centro de Modelamiento Matem\'{a}tico Basal Funds FB210005 from ANID-Chile and by Fondecyt Grant 1240884.  

\medskip
\ni 
{\bf AI use.}
The use of AI tools were very limited in this work. It was not used for the proof of the main theorem (which was announced in June 2025). However it was used during the writing of the paper in the following ways: 1) finding a proof of Lemma \ref{l.Domains-to-be-explored-are-Jordan} (in particular it allowed us to find the reference \cite{kuratowski1968topology2}). 2)  Handling the  technical details in the proof of Theorem \ref{th.MarkovIMF} and finally 3) Proofreading and removing minor mistakes.

\section{Preliminaries}\label{se.Preliminaries}

\subsection{Conformal Loop Ensembles.}\label{ss.CLEs}
Conformal Loop Ensembles (abbreviated CLE) are a one-parameter family of random sets $\cle_\kappa$, $\kappa\in(8/3,8]$, of non-(self)-crossing loops, first introduced by Sheffield in \cite{sheffield2009exploration} and studied in more depth in \cite{sheffield2012conformal}. They describe collections of loops (conjecturally) arising as the scaling limits of the interfaces of many models in statistical physics.
\par The \emph{simple} $\cle_\kappa$ measure is defined on any proper planar simply connected domain and consists of a collection of non-nested loops; its law is conformally invariant, i.e. invariant under any conformal map. From the simple $\cle_\kappa$, which we denote in the rest of the article $\cle_\kappa^1$, one can define iteratively what is known to be the \emph{nested--}$\cle_\kappa$ by applying the following procedure: first sample a simple $\cle_\kappa^1$, then inside each of the connected component which are surrounded by a loop sample an independent $\cle_\kappa^1$ to get a set of loops $\cle_\kappa^2$ (that consists of those loops of the first $\cle_\kappa^1$, and the loops of the $\cle_\kappa^1$ inside). To define $\cle_\kappa^3\supset\cle_\kappa^2$, repeat the operation in each of the components surrounded by a loop in the second layer $\cle_\kappa^2\setminus\cle_\kappa^1$, and by induction one can define $\cle_\kappa^k$ for any $k$; the full collection of nested loops $\bigcup_k\cle_\kappa^k$ is by definition the nested--$\cle_\kappa$, which we denote $\cle_\kappa^\nes$ in the remainder of the article. (The superscripts $1,k,\nes$ are here to stress that we are respectively considering the first layer of the loops, the first $k$ layers, and the full nested collection of loops). 

\par Since we will need to deal with the $\cle$ often, it is important to carefully define the measure space where they live, and thus its topology.

\subsubsection*{The space $\caX$ of loop sets.}
For loops $L_1,L_2:\S^1\to\overline D$, put 
\begin{equation*}
    \d_\calL(L_1,L_2)=\inf_{\varphi}\|L_1-L_2\circ\varphi\|_{\infty},
\end{equation*}
where $\varphi$ runs over the homeomorphism from $\S^1$ to itself.
\par We define $\widehat{\calL_D}$ to be the quotient of the set of loops by the pseudo-distance $\d_\calL$, and let $\calL_D$ be the closure in $\widehat{\calL_D}$ of the set of simple loops (i.e. loops admitting a representative which is simple). We will call the elements of $\calL_D$ \emph{non-self-crossing loops}. 

\par For any loop $\ell\in\calL_D$, we define its image $\im(\ell)$ to be $L(\S^1)$, the image of $\S^1$ by any representative $L:\S^1\to \overline D$ of $\ell$. The mapping $\ell\mapsto\im(\ell)$ is continuous from the set of non-self-crossing loops to the space $K(D)$ of compact subsets of $\overline D$ equipped with the Hausdorff topology (see Subsection \ref{ss.Hausdorff Topology} for the definition of the Hausdorff topology).  The loops whose image consists of a single point are named \emph{trivial loops.} They are contained in $\calL_D$.

\medskip
Let $\calX_D$ be the set of (at most) countable collections $\Gamma=(\ell_i)_{i\in I}$ (i.e. sets with possible repetitions) of loops in $\calL_D$ that do not contain any trivial loops and are \emph{locally finite}, i.e. for all $\varepsilon>0$, the set $\{i\in I, \, \mathrm{diam}(\ell_i)>\varepsilon\}$ is finite. Elements of $\caX_D$ are called loop sets, and we equip the space $\calX_D$ with the distance $\d_\caX$ defined by
\begin{equation*}
    \d_\calX(\Gamma,\Gamma')=\inf_{\pi:J\iso J'}\Big(\sup_{i\in J}2\d_\calL(\ell_{i},\ell'_{\pi(i)}) \vee \sup_{i\not\in J}\diam(\ell_i)\vee \sup_{i'\not\in J'}\diam(\ell'_{i'})\Big)
\end{equation*}
where $\Gamma=(\ell_i)_{i\in I}$ and $\Gamma'=(\ell'_{i'})_{i'\in I'}$ are two loop sets, and the infimum is taken over all bijections $\pi$ between finite subsets $J\subset I$ and $J'\subset I'$.

\begin{proposition}[Lemma 5 in \cite{benoist2019scaling}]\label{pr. measurable properties of the loop set topology}
    The space $\calX_D$ equipped with the topology from $\d_\caX$ is Polish. Moreover, the event that all the loops are distinct is measurable for the Borel $\sigma$-algebra of $\caX$.

\end{proposition}

\begin{remark}
    We have considered collections of loops (rather than sets of loops) only to ensure that the space $\calX_D$ is complete for $\d_\caX$, hence Polish. However, thanks to proposition above, whether a collection is actually a set is measurable, and since all the random collections that we will consider will be a.s. simple, we can safely consider that loop sets are indeed sets (whence the name). 
\end{remark}

\begin{remark}\label{rk.LoopSets as ClosedSets}
The map $\calX_D\rightarrow K(D)$ defined by $\Gamma\mapsto\overline{\bigcup_{\ell\in\Gamma}\im(\ell)}$ is measurable for the Borel $\sigma$--algebras. Therefore, we may see loop sets as closed sets for the Hausdorff topology.
\end{remark}

\subsection{Convergence to the nested--$\cle_3$}\label{ss.CLE3} We will rely on the convergence to the $\cle_3^\nes$ of the interfaces of a critical Ising model proven in \cite{benoist2019scaling}.  We chose to use the lattice $\T$ instead of $\Z^2$ since the exploration procedure is easier to define. 



\begin{theorem}[Convergence to $\cle_3^\nes$, Theorem 1 in \cite{benoist2019scaling}]\label{th.CLE3}
For $a>0$, let $\sigma^a$ be an Ising configuration on $D\cap a\T$ at critical temperature with + boundary condition, and let $\Gamma^a=\Gamma(\sigma^a)$ be the set of $+/-$ interfaces in $\sigma^a$. Then, in the topological space $\caX_D$ defined in Subsection \ref{ss.CLEs}, we have the convergence in law $\Gamma^a\xrightarrow[a\to0]{(d)}\Gamma$, where $\Gamma\sim\cle_3^\nes$.
\end{theorem}

It may happen that $D\cap a\T$ is not a connected graph; in such a case, we choose its largest connected component in term of number of vertices. Theorem 1 in \cite{benoist2019scaling} is stated over $\Z^2$, however their result goes through on the triangular lattice.

\subsection{Convergence of the one--point function.}\label{ss.1pointFunction} A key result to complement the convergence of the interfaces will be the convergence of the spin correlation functions of the Ising model, a famous
and difficult theorem proven in \cite{chelkak2015conformal}.
For simplicity, we will only state below the case of the one-point function which is all we need in our paper.
Let us also emphasise that given the recent breakthrough results about the convergence of the one-point function for the $q\in\{3,4\}$ critical Potts model  (see \cite{duminilcopin2026gaussianfreefieldconvergence, chen2026critical} and the manuscript in preparation \cite{alpha1}), it might be possible to extend our work to the $q\in\{3,4\}$ case, at least to some extent.
\par Before stating the theorem, recall that if $z$ is an interior point of a proper open subset $D$ of $\C$, which is simply connected, one can define the \emph{conformal radius of $D$ seen from $z$}, which we denote $CR(z,D)$, and which is defined as $CR(z,D)=\vert\phi'(z)\vert^{-1}$ where $\phi:D\iso\D$ is any conformal bijection sending $z$ to $0$. Recall that by Koebe's one-quarter theorem $\dist(z,\p D)\le CR(z,D)\le4\dist(z,\p D)$. We will only use that $\dist(z,\p D)\asymp CR(z,D)$ uniformly on $z,D$.

\begin{theorem}[Convergence of the one-point function of Ising \cite{chelkak2015conformal}]\label{th.1pointFunction}
There is a constant $\calC_\T>0$, which depends only on the lattice used, such that for any sequence of pointed discrete domains $(D^a,z^a)_{a>0}$ (where $z^a$ is a vertex of $D^a$) converging in the Carathéodory sense\footnote{The Carathéodory topology is recalled in Appendix \ref{ap. Topologies and properties of (nested) CLE}, see Subsection \ref{ss.Caratheodory Topology}.} to a pointed simply connected domain $(D,z)$, one has \begin{equation*}
\<{\sigma_{z^a}^a}_{D^a}^+\underset{a\to0}{\sim} {\calC_\T}a^{1/8}CR(z,D)^{-1/8}
\end{equation*}
where $\sigma^a$ follows the law of an Ising model on $D^a$ at inverse critical temperature $\beta_c(\T)$ with + boundary conditions.
\par Furthermore, in the case that $D^a=D\cap a\T$, the convergence is uniform on each compact $K\subset D$.
\end{theorem}

\begin{remark}\label{rk.CaratheodoryForCHI}
In \cite{chelkak2015conformal} the convergence result is expressed with the slightly stronger assumption that $\partial D^a\to\partial D$ in the Hausdorff sense. However, in Subsection 2.6. of their paper, the authors explain that their result holds with the weaker assumption of Carathéodory convergence as well.
We will need this generality in the proof of Proposition \ref{pr.NicePart}.
\end{remark}

We will actually need to use here and there uniform estimates on the 1-point and 2-point functions, that can easily be established as consequences of the previous theorem. These bounds are not new: slightly different ones appear in \cite{Camia_2015} (see Appendix B there), which were used for similar reasons.

\begin{proposition}\label{pr.BoundTwoPoint}
There exists a constant $C>0$ such that, for any finite domain $\Lambda\subset\T$, denoting by $\<{\cdot}_\Lambda$ the expectation of the critical Ising model with $+$ boundary condition on $\Lambda$, 
\bnum
\item if $x\in\Lambda$, $\<{\sigma_x}_\Lambda\le C \dist(x,\p \Lambda)^{-1/8}$;
\item if $x,y\in\Lambda$ are at distance at least $\|x-y\|$ from the boundary of $\Lambda$, $\<{\sigma_x\sigma_y}_\Lambda\le C\|x-y\|^{-1/4}$.
\enum
The second point holds as well for the (unique) infinite-volume measure (i.e. $\Lambda=\T$), without assumption on $x,y$.
\end{proposition}

\begin{proof}
For the first point, if $n:=\lfloor\dist(x,\p \Lambda)\rfloor\asymp\dist(x,\p \Lambda)$, then $\<{\sigma_x}_\Lambda\le \<{\sigma_0}_{\Lambda_n}$, by the classical monotonicity in the domain of the $+$ boundary Ising measures, where $\Lambda_n$ is the box of size $n$ centered at $0$. Furthermore, $\<{\sigma_0}^+_{\Lambda_n}\asymp n^{-1/8}$ by Theorem \ref{th.1pointFunction}.
\par For the second point, let $n:=\lfloor\|x-y\|/3\rfloor\asymp\|x-y\|$. Then we have
\begin{align*}
\<{\sigma_x\sigma_y}_\Lambda=\P_{FK,\Lambda}^{\mathrm{w}}(x\leftrightarrow y)\le\P_{FK,\Lambda}^{\mathrm{w}}\big(\{x\leftrightarrow x+\p\Lambda_n\},\{y\leftrightarrow y+\p\Lambda_n\}\big).
\end{align*}
Furthermore, since the FK--Ising model is stochastically increasing in the boundary conditions, one has $\P_{FK,\Lambda}^{\mathrm{w}}\big(\{x\leftrightarrow x+\p\Lambda_n\}\md\{y\leftrightarrow y+\p\Lambda_n\}\big)\le\P_{FK,\Lambda_n+x}^{\mathrm{w}}\big(\{x\leftrightarrow \p\Lambda_n\}\big)$ so that 
\begin{align*}
\P_{FK,\Lambda}^{\mathrm{w}}\Big(\{x\leftrightarrow x+\p\Lambda_n\}, \{y\leftrightarrow y+\p\Lambda_n\}\Big) \le \P_{FK,\Lambda_n}^{\mathrm{w}}\big(\{0\leftrightarrow \p\Lambda_n\}\big)^2=\<{\sigma_0}_{\Lambda_n}^2.
\end{align*}
We get the desired conclusion using $\<{\sigma_0}^+_{\Lambda_n}\asymp n^{-1/8}$ once again. Finally, the inequality in the infinite volume case holds by taking the limit $\Lambda\to\T$ in point 2).
\end{proof}

\subsection{Sobolev spaces of negative index $H^{-\alpha}$}\label{ss.Sobolev}

Let $\Omega$ be a Jordan domain in $\R^2$. Recall that $H_0^1(\Omega)$ is the Hilbert space obtained as the closure of $C_0^\infty(\Omega)$ under the norm induced by the scalar product 
\begin{equation}
\<{f,g}_{H_0^1(\Omega)}:=\<{f,g}_\nabla= \int_\Omega \nabla f(x) \cdot \nabla g(x) dx.
\end{equation}

We may define $H^{-1}(\Omega)$ to be the dual of the Hilbert space $H^{1}_0(\Omega)$ so that for each $T\in H^{-1}(\Omega)$
\begin{align*}\label{}
\|T\|_{H^{-1}(\Omega)} = \sup_{g \in C_0^\infty(\Omega), \|g\|_{H_0^1(\Omega)} \leq 1} \<{T,g}_{L^2(\Omega)}\,.
\end{align*}
The Hilbert space $H^{-1}(\Omega)$ can also be seen as the closure of $C^\infty(\Omega)$ under the above norm (and associated scalar product).

\medskip One very convenient orthogonal basis of $H_0^1(\Omega)$ is obtained by considering an orthogonal basis of eigenfunctions $\{e_i\}_{i=1}^\infty$ of the Laplacian $\Delta$ on the domain $\Omega$ with Dirichlet boundary conditions and normalised so that $\| e_i\|_{L^2} = \int_\Omega e_i(x)^2  dx = 1$.  Let $\lambda_i >0$ be their corresponding eigenvalues $(-\Delta) e_i = \lambda_i e_i$. 
Using this natural $L^2$ orthonormal basis, one may check that for any $f\in C_0^\infty(\Omega)$, 
\begin{align*}\label{}
\|f\|_{H_0^1}^2 = \sum_{i=1}^\infty \lambda_i  \<{f,e_i}_{L^2(\Omega)}^2\,.
\end{align*}
(This follows readily from the integration by parts identity $\int_\Omega\nabla f(x)\cdot\nabla f(x)dx=\int  f(x) (-\Delta) f(x)  dx$, in other words, $\|f\|^2_{H_0^1} = \<{f, (-\Delta) f}_{L^2(\Omega)}$. We make extensive use here of the Dirichlet boundary conditions rooted in the definition of $H_0^1(\Omega)$). 

For any $\alpha>0$ and $f\in C_0^\infty(\Omega)$, we may also define

\begin{align}\label{e.H0A1}
\|f\|_{H_0^{\alpha}}^2 & = \sum_{i=1}^\infty (\lambda_i)^\alpha  \<{f,e_i}_{L^2(\Omega)}^2  \\
\|f\|_{H^{-\alpha}}^2 & = \sum_{i=1}^\infty \frac 1 {(\lambda_i)^\alpha}  \<{f,e_i}_{L^2(\Omega)}^2  =  \Big( \sup_{g \in C_0^\infty(\Omega), \|g\|_{H_0^\alpha(\Omega)} \leq 1}  \<{f,g}_{L^2(\Omega)} \Big)^2 
\label{e.H0A2}
\end{align}
as well as the associated Hilbert spaces $H_0^{\alpha}(\Omega)$ and $H^{-\alpha}(\Omega)$.  
One disadvantage of this spectral viewpoint is that even though the behaviour of the sequence $\{ \lambda_i\}_{i\geq 1}$ is rather well understood by Weyl's formula, the control on the eigenfunctions $\{ e_i\}_{i\geq 1}$ as a function of the domain $\Omega$ is more difficult.  
Especially in the case of $H^{-1}(\Omega)$, we may overcome this spectral difficulty by using the following useful expression of the $H^{-1}(\Omega)$ norm: for any function $f\in L^2(\Omega)$, 

\begin{align}\label{eq.Green-Norm H-1}
\|f\|_{H^{-1}}^2 = \iint_{\Omega\times\Omega}  f(x)f(y) G_\Omega(x,y)dx dy\,,
\end{align}
where $G_\Omega$ is the Green's function with Dirichlet boundary conditions of the domain $\Omega$ normalized to be the inverse of $-\Delta$, so that $G_\Omega(x,y)\sim-\frac1{2\pi}\log(\|x-y\|)$. (Recall that $L^2(\Omega)$ is embedded continuously in $H^{-1}(\Omega)$ so that $\|f\|_{H^{-1}}$ is a priori well-defined).

\begin{remark}\label{}
See also the recent work \cite[Section 2.2]{lupu2025relation} for an extension of this integral kernel type expression to other index Sobolev norms (at least in the setting where $\Omega=\C$). Note also that if this extension was also available in smooth domains $D$, this would provide a soft proof of the optimal Sobolev regularity result $H^{-1/8-}(\Omega)$ of \cite{Furlan_2017}.
\end{remark}

The following Proposition provides us with a useful sufficient criterion for tightness in $H^{-\alpha}(\Omega)$. 
\begin{proposition}[A sufficient condition for tightness]\label{pr.SuffTightness}
Let $\Omega$ be a Jordan domain and $0 < \alpha' < \alpha$. Let $\{f_n\}_{n\geq 1}$ be a sequence of random distributions in  $H^{-\alpha'}(\Omega)$ such that the sequence $n \mapsto\Eb{ \| f_n \|_{H^{-\alpha'}(\Omega)}}$ is bounded. Then the family $(f_n)$ is tight in $H^{-\alpha}(\Omega)$.
\end{proposition}

\begin{proof} Let $A$ be a uniform bound for $\Eb{ \| f_n \|_{H^{-\alpha'}(\Omega)}}$, and let $\eps>0$. By Markov's inequality, \begin{align*}
		\P\big( \| f_n \|_{H^{-\alpha'}(\Omega)} > A/\eps\big) \le \eps.
	\end{align*}
By Rellich Theorem, $H^{-\alpha'}(\Omega)$ is compactly embedded in $H^{-\alpha}(\Omega)$, i.e. the set $K_\eps:=\{g\in H^{-\alpha}\md \|g\|_{H^{-\alpha'}}\le A/\eps\}$ is a compact subset of $H^{-\alpha}(\Omega)$. Thus, for some compact $K_\eps$, the family of numbers $\big(\P(\{f_n\not\in K_\eps\})\big)_n$ is bounded by $\eps$. In other words, the sequence $(f_n)$ is tight in $H^{-\alpha}(\Omega)$, as desired. 
\end{proof}

Finally, the following lemma will be useful:

\begin{lemma}\label{le.MeasurableCoeff}
Let $\Omega$ be a Jordan domain and $\alpha>0$ and let $\phi$ be a random distribution in  $H^{-\alpha}(\Omega)$. Let $(e_i)$ be any orthonormal basis of the space $H^\alpha_0(\Omega)$, and let $a_i=\<{\phi,e_i}$. Then $\phi$ is an explicit function of the $a_i$, which is measurable in the product $\sigma$--algebra of the $a_i$.
\end{lemma}

\begin{proof}
	Indeed, if $(e_i^*)$ is the dual basis of $(e_i)$, then $\phi$ is the $H^{-\alpha}(\Omega)$--limit of the random variables $\sum_{i\le N} a_ie_i^*$ as $N\to\infty$.
\end{proof}

\subsection{Critical magnetization field}\label{ss.IMF}

The Ising magnetization field (IMF)  is a random distribution introduced in \cite{Camia_2015}. It has been proven to converge as the mesh-size vanishes to a random distribution living in the Sobolev space $H^{-1/8-\eps}$ \cite{Camia_2015,Furlan_2017}. Our paper is concerned with analyzing the interplay between this continuous magnetization field and the $\cle_3$, the other main scaling limit of the critical Ising model, introduced above. More precisely, as explained in the Introduction, we aim to show that the IMF can be reconstructed in the limit from the set of interfaces. In particular, it provides a new proof of its existence (see as well \cite{lopez2026ising} for another work constructing the IMF from another perspective).
\par To this end, let us introduce formally the discrete IMF.

\begin{definition}\label{d.Phi}
Let $D$ be a Jordan domain, and for $a>0$, $D^a$ the (largest connected component of the) graph $D\cap a\T$.
We define the Ising Magnetization Field (IMF) to be the random distribution in $H^{-1}(D)$ defined by \begin{equation}
\Phi^a = \Phi^a(\sigma^a):= a^{-1/8} \sum_{z\in D^a} \sigma^a_z 1_{C_a(z)},
\end{equation} where for each vertex $z\in D^a$, $C_a(z)$ stands for the regular hexagon of sidelength $a$ centered around $z$. 
\end{definition}

\begin{remark}
We chose to sum indicators of hexagons instead of a more standard $\delta_x$ mass to ensure, as in \cite{Camia_2015}, that the discrete fields $\Phi^a$ belong to $H^{-1}(D)$.
\end{remark}

In the rest of the article, for a spin configuration $\sigma^a$ on $D^a$, with a slight abuse of notation, we will extend $\sigma^a$ to the filling of $D^a$, i.e. $\overline{D}^a:=\bigcup_{z^a\in D^a} C_a(z^a)$ so that it is constant on the interior of each hexagon $C_a(z^a)$ for each lattice point $z^a\in D^a$; the way we define $\sigma^a$ on the boundary of the hexagons is irrelevant and can be chosen arbitrarily. 
\par A very important fact to study the magnetization field is to show its tightness in an appropriate Sobolev space. This is the purpose of the following proposition.

\begin{proposition}[Tightness of the IMF]\label{pr.PhiTight}
For any Jordan domain $D$, the sequence of random distributions $\Phi^a$ is tight in $H^{-2}(D)$. 
\end{proposition}

To prove this proposition, it is convenient to consider first the case of the \emph{centered} IMF, that is the random distribution $\widehat{\Phi}^a=\Phi^a-\E[\Phi^a]=\big(z\mapsto a^{-1/8}\cdot(\sigma^a_z-\<{\sigma^a_z})\big)$.

\begin{proposition}[Tightness of the centered IMF]\label{pr.Phi-Centered-Tight}
	For any Jordan domain $D$, the sequence of random distributions $\widehat{\Phi}^a$ is tight in $H^{-2}(D)$. 
\end{proposition}

The tightness of the IMF will be a corollary of this last proposition together with the estimates on the one-point function derived in Proposition \ref{pr.BoundTwoPoint} and a technical but elementary result postponed in Appendix \ref{ap.Green-bounds-proof-finiteness}. 

\subsubsection*{GHS inequality and a consequence} One can prove Proposition \ref{pr.Phi-Centered-Tight} with the help of the following GHS inequality first proved in \cite{griffiths1970concavity}. The GHS inequality has many important corollaries and will be a key tool once again in Section \ref{se.IMFwrtCurve}.
Let us recall it here.

\begin{proposition}[GHS inequality \cite{griffiths1970concavity}]\label{pr.GHS}
Consider a ferromagnetic Ising model on a finite graph $G=(V,E)$ with nonnegative magnetic field. Then, for any $u,v,w\in V$ vertices \begin{equation}
\<{\sigma_u\sigma_v\sigma_w}-\big(\<{\sigma_u\sigma_v}\<{\sigma_w}+\<{\sigma_u}\<{\sigma_v\sigma_w}+\<{\sigma_u\sigma_w}\<{\sigma_v}\big)+2\<{\sigma_u}\<{\sigma_v}\<{\sigma_w}\le 0.
\end{equation}
\end{proposition}

This inequality has the following useful corollary, whose proof is included below for the sake of completeness.

\begin{lemma}[Covariances are increasing in the domain] \label{l.covGHS}
Let $G=(V,E)$ be a finite graph, with boundary $\partial G\subset V$, and $H$ be a subgraph of $G$ which is induced by a subset of vertices, with $\p H$ being the union of the vertices of $H$ adjacent to $G\setminus H$ and $\big(\partial G\cap V[H]\big)\cup$. For any $x,y\in V[H]$, the following inequality holds: $\cov_{H}(\sigma_x,\sigma_y)\le \cov_{G}(\sigma_x,\sigma_y)$ where covariances are taken with respect to Ising measure with + boundary condition (respectively on $\p H$ and $\p G$) and same nonnegative coupling constants.
\end{lemma}

We will actually use this lemma in the case where $G$ is infinite; it follows from the finite case by taking limits since $\sigma_x,\sigma_y$ are cylinder functions.

\begin{proof}[Proof of lemma \ref{l.covGHS}]
    GHS inequality says precisely that the partial derivatives of the covariance functions w.r.t. the external magnetization field $h$ are nonpositive, on the range $h\ge0$. As restricting to a subgraph amounts to imposing an infinite magnetization field outside of the subgraph, the covariances must decrease.
\end{proof}

\begin{proof}[Proof of the tightness of the centered IMF, Proposition \ref{pr.Phi-Centered-Tight}]
By Proposition \ref{pr.SuffTightness}, it is sufficient to show that 
\begin{align*}\label{}
\limsup_{a\to 0} \E\left[\| \widehat{\Phi}^a \|_{H^{-1}(D)}^2\right] < \infty.
\end{align*}
Now we have for each $a>0$ and using~\eqref{eq.Green-Norm H-1}, 
\begin{align*}\label{}
 \E\left[\| \widehat{\Phi}^a \|_{H^{-1}(D)}^2\right]
 & =  a^{-1/4} \iint_{D\times D} \cov_{D^a}(\sigma^a_u, \sigma^a_v) G_D(u,v)du dv.
\end{align*} 

\par Since covariances are increasing in the domain, we can bound $\cov_{D^a}(\sigma^a_u, \sigma^a_v)$ by $\cov_{a\T}(\sigma^a_u,\sigma_v^a)=\<{\sigma^a_u\sigma^a_v}_{a\T}$ where the subscript $a\T$ means that we are considering expectation w.r.t. the full-plane critical Ising model on $a\T$ (well defined since the phase transition is continuous). However, the full-plane correlation function can be upper bounded from Proposition \ref{pr.BoundTwoPoint}, so that one only needs to show that 
\begin{align*}
\iint_{D\times D} \|u-v\|^{-1/4} G_D(u,v)du dv<+\infty.
\end{align*}
This is straightforward: we can write $G_D(u,v)=O\big(-\log(\|u-v\|)\big)$ and $\log(\|u-v\|)\|u-v\|^{-1/4}$ is integrable on $D$.
\end{proof}


\begin{proof}[Proof of the tightness of the IMF, Proposition \ref{pr.PhiTight}]
By the previous proposition, what remains to be done is to show that the family of functions $z\mapsto\<{\sigma^a_z}$ is tight in the Sobolev space $H^{-2}(D)$; to this end, by Proposition \ref{pr.SuffTightness} again, it is enough to show that \begin{align}\label{eq.Norm H-1 --1-pt}
\|a^{-1/8}\<{\sigma_\cdot^a}\|_{H^{-1}(D)}^2= a^{-1/4}\iint_{D\times D}G_D(u,v)\<{\sigma_u^a}\<{\sigma_v^a}dudv
\end{align} is bounded uniformly on $a>0$. Now, $\<{\sigma_u^a}\le \dist(u,\p D)^{-1/8}a^{1/8}$ by Proposition \ref{pr.BoundTwoPoint}, and similarly for $v$. Hence, the uniform boundedness of the norm \eqref{eq.Norm H-1 --1-pt} boils down to the finiteness of
\begin{align}\label{e.I_D(D)}
I_D(D):=\iint_{D\times D}G_D(u,v)\dist(u,\p D)^{-1/8}\dist(v,\p D)^{-1/8}dudv \,.
\end{align}
We give a proof of the finiteness of this integral in Appendix \ref{ap.Green-bounds-proof-finiteness}, see Lemma \ref{l.int-finite-Green-Dirichlet} and the proof following it.
\end{proof}

\subsubsection*{Remarks on the space $H^{-2}(D)$ where our fields live}

Note that under the setting of Definition \ref{d.Phi}, the magnetization field cannot be tested against test functions that do not vanish on the boundary, since they are not elements of $H^1_0(D)$. This is not a defect of our work since for domains such that $\int_D\dist(z,\p D)^{-1/8}=+\infty$ the random variables $m^a=\<{\Phi^a,\1}$ can be seen not to be tight (using the tools developed above) and so there can be no hope of proving a tightness result such as Proposition \ref{pr.PhiTight} in a framework that allows testing against the function $\1$.

In Section \ref{s.RoughBoundary} we explain how to fix this in the case of a domain whose boundary is not too rough in a precise sense, see Theorem \ref{th.main-smooth-bdry} there.



\section[Magnetization field conditioned on loops]{Magnetization field conditioned on the loops intersecting a curve.}\label{se.IMFwrtCurve}


In this section, given a test function $f\in C_0^\infty(D)$, 
we will study the conditional expectation \begin{equation}\label{eq.CondExpectation}
\Eb{\<{\Phi,f}\md\Gamma(\eta)},
\end{equation} where $\Gamma(\eta)$ is the set of loops $\ell\in\Gamma$ that intersect a given smooth curve $\eta$\footnote{We will consider more generally $\eta$ made of a finite union of smooth curves.}. 
This is a key step for the proof of our main theorem since we ought to explicitly describe the random variable $\<{\Phi,f}$ as a function of $\Gamma$. 
\par Our goal will be to define in the continuum a candidate $H_{f,\eta}$ for this conditional expectation \eqref{eq.CondExpectation} as a measurable function of $\Gamma(\eta)$ and then to show that the analogous conditional expectation in the discrete, which we will denote $H^a_{f,\eta}$ in the following (see Proposition \ref{pr.MarkovForPhi eta}), converges to it. This convergence will have as a corollary that our candidate is indeed the correct conditional expectation. This will be the content of Theorem \ref{th. cond. expect. of magn wrt Gamma(eta)}.
 
\medskip The first step will be to explain how to explore in a Markovian way (in the discrete) the loops intersecting $\eta$. We will define the objects of interest in the continuum by analogy. We will then proceed with the proof of the convergence.
\par To prepare our work, let us define a \emph{smooth curve in} $D$ as the image of an injective $C^1$ path $\eta:[0,1]\to \overline{D}$ such that $\eta((0,1))\subset D$ but $\eta(0),\eta(1)\in\p D$. A \emph{smooth grid in} $D$ is simply a finite union of smooth curves, having finitely many intersection points.
In particular, smooth curves and smooth grids are closed subsets of $\overline D$. It will be important in Section \ref{ss.Set-up-mainThm} to consider smooth grids rather than smooth curves.
\par Let us fix as well the notations that we will use: $D$ is a Jordan domain, and for $a>0$, $D^a$ is the (largest connected component of the) graph $D\cap a\T$. Furthermore, $\eta=\bigcup_{i=1}^k\eta_i$ is a smooth grid in $D$ where $(\eta_i)_{1\le i\le k}$ are smooth curves.


\subsection{Setup and discrete exploration procedure.}\label{ss.ExplorationProcedure}

In this Subsection, we will define an exploration procedure of some spins of a $\{\pm\}$-configuration $\sigma$ on $a\T$ which is identically + outside of $D^a$. The exploration corresponds to disclosing the interfaces of the $+/-$ clusters of $\sigma$ that intersect the grid $\eta$, as explained in Subsection \ref{ss.Strategy Proof}.

\medskip For each $a$ small enough and $i\in\{1,\ldots,k\}$, we choose a self-avoiding path $\eta_i^a$ in $a\T$ such that:
\bi
\item[(i)] For any $i$, $\d_{\mathrm{Haus}}(\eta_i^a,\eta_i)\le 2a$, i.e. $\eta_i^a$ follows $\eta_i$ at distance at most $2a$.
\item[(ii)] The beginning and end of each of the paths $\eta_i^a$ are both outside of $D^a$. 
\ei
We will write $\eta^a=\bigcup_i\eta_i^a$, viewed as a subset of $D^a$, and we will say that $(\eta^a)_{a>0}$ is a \emph{discretization} of $\eta$.

\begin{figure}[!htp]
\begin{center}
\includegraphics[width=\textwidth]{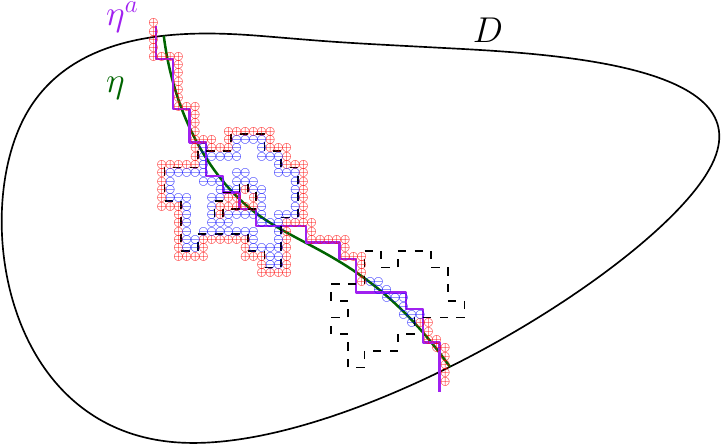} 
\end{center}
\caption{Illustration of the exploration procedure leading to $\Gamma^a(\eta)$ in the case of a square lattice. As opposed to the cleaner case of the triangular lattice (harder to draw), such an exploration would require to prescribe some conventions for the $\Z^2$ lattice. However, as proven in \cite{benoist2019scaling}, any given such convention converges in the limit to the same $\cle_3$. The convention used in the above picture is not explained and is not relevant in the limit.}\label{f.exploration-nested}
\end{figure}


\bigskip\ni
\textbf{Exploration procedure.} Let us assume first that $\eta=\eta_1$ is a smooth curve in $D$. Write the self-avoiding path $\eta^a$ as $(x_0,\ldots, x_N)$ where $\sigma^a(x_0)=\sigma^a(x_N)=+$, and write $\Gamma^a=\Gamma^a(\sigma^a)$ for the collection of loops on the dual lattice separating + and $-$ clusters. 
\bi 
\item[a)] We reveal all the spins $\sigma^a(x_0),\ldots, \sigma^a(x_N)$.
\ei
\ni Now, let $I=\{i_1,\ldots, i_M\}$ be the set of indices $i\in\{0,\ldots, N-1\}$ such that $\sigma^a(x_i)\ne\sigma^a(x_{i+1})$.

\bi
\item[b)] For $i\in I$, there is a loop $\ell\in \Gamma(\sigma^a)$ passing between $x_i$ and $x_{i+1}$; we explore it, that is we explore the spins as we follow the $+/-$ interface, starting in one direction from the part of the interface between $x_i$ and $x_{i+1}$. 
\ei 
\ni Of course, a loop $\ell$ that will be explored in this step may be explored many times (actually, it will be explored the number of times it intersects $\eta^a$, so an even number of times). This is not a problem: it is just a waste of time. During the exploration of a loop, we reveal precisely the spins which are on either side of the loop. See Figure \ref{f.exploration-nested} for a drawing of the spins discovered during the exploration procedure.
\par In case $\eta$ is a smooth grid, i.e. a finite union $\bigcup_i\eta_i$, just do all the explorations one by one in any order. (The set of vertices revealed doesn't depend on the order of exploration.)

\medskip Let us give names for the data collected during the exploration.
\bnum[{\bf (i)}]
\item We write $\Gamma^a(\eta)$ for the set of loops discovered during the exploration, i.e. the loops intersecting $\eta$: \begin{align}\label{e.GammaEtaA}
\Gamma^{a}(\eta):=  \big\{ \ell \in \Gamma^{a}, \ell \cap \eta^a \neq \emptyset \big\}\,.
\end{align}
\enum 
The set $\Gamma^{a}(\eta)$ contains the same information as what was revealed by the Markovian exploration procedure. The counterpart in the continuum is \begin{align}\label{e.GammaEta}
\Gamma(\eta):=  \left\{ \ell \in \Gamma \md \ell \cap \eta \neq \emptyset \right\}\,,
\end{align}
seen as a random element of the set of loops $\caX_D$; the map $\Gamma\mapsto\Gamma(\eta)$ is indeed measurable, see Lemma \ref{l. intersection with closed set is measurable}. We will prove in Proposition \ref{pr.JointConv} that $\Gamma^a(\eta)$ converges jointly with $\Gamma^a$ to the continuous coupling $\big(\Gamma(\eta),\Gamma\big)$.

\bnum[{\bf (ii)}]
    \item We define $\Expl^{a}(\eta)=\Expl^{a}(\Gamma^a(\eta))$ to be the set of sites in $a\T$ which are visited by the exploration procedure, which we view as a closed set in the Hausdorff topology.   
\enum
    
\ni
In the continuum, we define $\Expl(\eta)$ by \begin{equation}
\Expl(\eta)=\eta\cup\bigcup_{\ell\in\Gamma(\eta)}\im(\ell).
\end{equation}
Lemma \ref{l. intersection with closed set is measurable} shows that this is indeed a closed set, measurable in $\Gamma$. Notice that in the discrete, the set of sites $\Expl^a(\eta)$ together with their explored spin contains precisely the same information as the set of loops $\Gamma^a(\eta)$. We will often deal with the complementary of $\Expl(\eta)$, which we write $D^\eta=D\setminus\Expl(\eta)$.

\medskip Finally, an important piece of data is the number of loops surrounding a given $z$ since its parity will give the sign of the boundary condition of its connected components of the spins not yet explored: 
\bnum[{\bf (iii)}]
\item for $z\in D^a\setminus\Expl^a(\eta)$, let $D_z^a$ be its connected component in $D^a\setminus\Expl^a(\eta)$ and let $\nes^a(z)=\nes_{\Gamma^a(\eta)}(z)$ be the number of loops of $\Gamma^a(\eta)$ which surround $z$.
\enum
In the continuum, for $z\not\in\Expl(\eta)$, let $D_z$ be its connected component in $D^\eta=D\setminus\Expl(\eta)$ and let $\nes(z)=\nes_{\Gamma(\eta)}(z)$ be the number of loops of $\Gamma(\eta)$ surrounding $z$. It is almost surely finite for any such $z$ since $\Gamma$, and hence $\Gamma(\eta)$, is locally finite (in the sense that it contains a finite number of loops of diameter larger than $\eps$ for any $\eps>0$). 

\medskip An immediate consequence of this exploration procedure is the following Proposition:

\begin{proposition}[Markov Property for the Discrete Field]\label{pr.MarkovForPhi eta}
Let $f:\overline D\to \R$ be a continuous function and let $H_{f,\eta}^a:=\Eb{\<{\Phi^a,f}\md\Gamma^a(\eta)}$. Almost surely, we have \begin{equation}\label{eq.H f eta discret}
H_{f,\eta}^a := a^{-1/8}\sum_{z\in D^a}\bar{f}(z)\<{\sigma^a(z)}^+_{D_z^a}(-1)^{\nes^a(z)}
\end{equation}
where $\bar{f}$ is the convolution of $f$ with the indicator of a hexagon of size $a$ centred at $0$.
\rm 
\end{proposition}

(For $z\in \Expl^a(\eta)$, the term $\<{\sigma^a(z)}^+_{D_z^a}(-1)^{\nes^a(z)}$ is simply defined as $\sigma^a(z)$ since the spin at $z$ has been discovered by the exploration.)

\subsection{Dimension of the fractal set $\Expl(\eta)$.}\label{ss.dim-Expl(eta)}
It is a key technical point in what follows to analyze the geometry of the set $\Expl(\eta)$ and more specifically to upper-bound its Minkowski dimension. For this, we use exact values of some arm exponents that were derived in \cite{wu2018alternating}.

\begin{proposition}\label{pr.dim-Expl(eta)}
With the notations above, for any compact set $K\subseteq D$ the upper Minkowski dimension of the random set $\Expl(\eta)\cap K$ is almost surely less than or equal to $\frac{11} 8$.
\end{proposition}

\begin{proof}
Let us consider the case of one smooth curve $\eta$, the case of a smooth grid being very similar. For each $\eps>0$, we want to upper bound in expectation the number of $\eps$-squares which cover the set $\Expl(\eta)$. Consider the $\eps \Z^2$ tilings of such $\eps$-squares. Group the squares in terms of their dyadic distances from the curve $\eta$ (we should also take into account the distance to $\p D$ but we shall not do it here for simplicity). 
\bi
\item We have $O(\eps^{-1})$ such squares at distance $\eps$ from $\eta$
\item Then for any $1 \leq k \leq O(\log_2 (\frac 1 \eps))$ we have $2^{k}  \eps^{-1}$ such squares at distance between $2^{k-1} \eps$ and $2^k \eps$ from $\eta$. 
\ei
For any $1 \leq k \leq O(\log_2 (\frac 1 \eps))$, the probability for a square at distance $2^k \eps$ from $\eta$ to intersect $\Expl(\eta)$ is bounded from above by a constant times the plane alternating 2-arm probability for spin-Ising as computed by  Hao Wu in \cite{wu2018alternating}. The reason one can compare the finite domain 2-arm exponent with the plane 2-arm exponent is thanks to the mixing argument from \cite{duminil2011connection} (or rather the updated version available on H. Duminil-Copin's website).  Using RSW for FK-Ising,  it implies that the probability to have a spin two-arm event from $\eps$ to $2^k \eps$ inside a finite domain is up to constant the same as having a two-arm event in the full plane. The latter one is computed in \cite{wu2018alternating} and is equal to $\frac 5 8$.
(N.B. A similar -- but more subtle -- mixing argument will be needed in the proof of Lemma \ref{l. curve not tangent cle_3}). 
As such, the probability of a two-arm event from macroscopic distance $\eps$ to distance $2^k \eps$ is bounded above by 
\begin{align*}\label{}
O(1) \, \left( \frac \eps {2^k \eps}\right)^{5/8}\,.
\end{align*}
If $N_\eps$ denotes the number of $\eps$ squares in $\eps \Z^2$ which intersect $\Expl(\eta)$, we obtain the following first moment bound 
\begin{align*}\label{}
\Eb{N_\eps}  & \leq O(1) \sum_{k=0}^{C \log_2(\eps^{-1})} 2^{k} \eps^{-1} (2^k)^{-5/8} \\
& \leq O(1)  \eps^{-1 }\sum_{k=0}^{C \log_2(\eps^{-1})}  (2^{k})^{\frac {3} 8 } \\
& = O(1) \eps^{-1}  \eps^{-\frac {3} 8 } = O(1) \eps^{-\frac {11} 8}.
\end{align*}
This bound, coupled with a standard Borel-Cantelli argument, ends the proof.
\end{proof}

\subsection{An explicit candidate for the conditional expectation of the continuous observable.}\label{ss.candidate}
Recall that our goal in this section is to understand the conditional expectation $\E[\<{\Phi,f}\md\Gamma(\eta)]$. In the same way as in Subsection \ref{ss.Strategy Proof}, we introduce the following expression, a natural candidate for this conditional expectation by Proposition \ref{pr.MarkovForPhi eta}:
\begin{equation}\label{eq.H f eta}
H_{f,\eta}=H_{f,\eta}(\Gamma):= \calC_\T\int_{D}   f(z) \, (-1)^{\nes(z)} \,  CR_{D^\eta}(z)^{-1/8} dz\,,
\end{equation}
where we write $CR_{D^\eta}(z)$ for the conformal radius of $z\in D^\eta=D\setminus\Expl(\eta)$ in its connected component $D_z$ of $D^\eta$, which is a.s. simply connected. (Note that, since $\Expl(\eta)$ is a.s. of Lebesgue measure zero, we may view the integral above over $D^\eta$ instead.)

\begin{proposition}[Definition of $H_{f,\eta}$]\label{pr.Def H f eta}
If $\Gamma\sim\cle_3^\nes$, the integrand in \eqref{eq.H f eta} is a.s. integrable. The functional $H_{f,\eta}$ is thus well defined and measurable in $\Gamma$. 
\end{proposition}

\begin{proof}
By Proposition \ref{pr.dim-Expl(eta)} $\Expl(\eta)=\p D^\eta\setminus\p D$ is of Minkowski dimension less than $15/8$. Lemma \ref{l.dist^-1/8 int} (or rather Remark \ref{rk.bdry-smooth on a subset} following it, since $f$ is zero on a neighbourhood of $\p D$) implies that $f(z)\dist(z,\p D^\eta)^{-1/8}$ is integrable on $D^\eta$, and thus the same is true with the conformal radius.

Furthermore, the functional $\Gamma\mapsto H_{f,\eta}$ is clearly measurable. 
\end{proof}

\subsection{Skorokhod coupling.}\label{ss.Skorokhod}

For convenience we now couple discrete configurations and limiting objects on the same probability space. As often, one could proceed without Skorokhod but this will make the proof easier to explain and set up. So, thanks to Theorem \ref{th.CLE3}, let us assume that all the Ising spin configurations $\sigma^a$ are coupled together such that a.s. $\Gamma^a(\sigma^a)\xrightarrow[a\to0]{}\Gamma$ in the topological space of loop sets $\caX_D$ for some random $\Gamma\in\caX_D$, with $\Gamma\sim\cle_3^\nes$. 
\par One advantage of this coupling is that the joint convergence statements we are going to prove are a consequence of the convergence in probability of the observables we are interested in.

\subsection{Joint convergence of the data collected during the exploration.}\label{ss.JointConv}
In this subsection, we state and prove a stabilisation result concerning the sets introduced above. The main argument is to show that a.s. there are no loops in $\Gamma$ tangent to $\eta$, and the rest concerns topological properties of the space $\caX_D$. All the lemmas needed are written down in Appendix \ref{ap. Topologies and properties of (nested) CLE}.

\begin{proposition}\label{pr.JointConv}
Under the Skorokhod assumption $\Gamma^a\to\Gamma$, we have almost surely $\Gamma^a(\eta)\to\Gamma(\eta)$. As a consequence, almost surely the following three points hold:
\bnum 
\item $\Expl^a(\eta^a)\to\Expl(\eta)$ in $K(D)$;
\item for any $z\not\in\Expl(\eta)$, the pointed domains $(D^a_z,z)$ converge to $(D_z,z)$ in the Carathéodory sense;
\item for any $z\not\in\Expl(\eta)$, the indices $\nes^a(z)$ stabilize to $\nes(z)$. 
\enum 
\end{proposition}

\begin{proof}
Almost surely, $\Gamma\sim\cle_3^\nes$ is not tangent to the deterministic smooth grid $\eta$. The proof of this fact is similar to the proof of Proposition \ref{pr.dim-Expl(eta)} that upper-bounds the dimension of $\Expl(\eta)$. Since it is a bit more technical, we give it in Appendix \ref{ap. Topologies and properties of (nested) CLE}, see Lemma \ref{l. curve not tangent cle_3}. 

Now, since the points of continuity of the map $\Gamma\mapsto\Gamma(\eta)$ are precisely those $\Gamma$ such that $\eta$ is not tangent to any of the loops in $\Gamma$, we have that $\Gamma\mapsto\Gamma(\eta)$ is a.s. continuous in $\Gamma\sim\cle_3^\nes$, and thus that $\Gamma^a(\eta)\to\Gamma(\eta)$ a.s. 

The rest is purely deterministic and follows from technical lemmas about the space of loop sets $\caX_D$: the first point is a consequence of Lemma \ref{l. intersection with closed set is measurable}, the second one follows from Lemma \ref{l.OpenDomainsConvergeCaratheodory}, and the third one from Lemma \ref{l.NumberOfLoopsSurroundingForall z at once}.
\end{proof}

\subsection{Joint convergence of $H^a_{f,\eta}$ with the set of interfaces $\Gamma^a$}\label{ss. imf cond to Gamma(eta)}
Recall the notations that we used up to now:
\bi 
\item $D$ be a Jordan domain, and for $a>0$, $D^a$ is the (largest connected component of the) graph $D\cap a\T$. 
\ei
Each $D^a$ comes with a critical Ising configuration $\sigma^a$ defined on it, with + boundary conditions.
\bi\item $\eta$ is a smooth grid in $D$.
\ei We have chosen a discretization $(\eta^a)_{a>0}$ of $\eta$ as in Subsection \ref{ss.ExplorationProcedure}.
\bi\item $f$ is an element of $C^\infty_0(D)$. \ei

Now, we can state the main theorem of this section.

\begin{theorem}\label{th. cond. expect. of magn wrt Gamma(eta)}
With the notations from above, the pair $(\Gamma^a,H^a_{f,\eta})$ jointly converges in law as $a\to0$ to $(\Gamma,H_{f,\eta})$.
In particular, if $(\Gamma,\Phi)$ is any sublimit of the coupling $(\Gamma^a,\Phi^a)$, it must satisfy
    \begin{equation}\label{eq. cond. expect. of magn wrt Gamma(eta)}
     \Eb{\<{\Phi,f} \md \Gamma(\eta)} = H_{f,\eta}\qquad\text{almost surely.}
\end{equation}
\end{theorem}

The almost sure equality $\Eb{\<{\Phi,f} \md \Gamma(\eta)} = H_{f,\eta}$ is a formal consequence of the joint convergence since, by Proposition \ref{pr.Def H f eta}, $H_{f,\eta}$ is measurable w.r.t. $\Gamma$; see Lemma \ref{l. conditional expectation of the limit}. Since we coupled everything with Skorokhod's Theorem, and since $H_{f,\eta}$ will be shown to be measurable w.r.t. $\Gamma$, it is enough to show that, in our coupling, the following convergence in probability holds: 

\begin{proposition}\label{pr.Conv H f eta a vers H f eta}
With the notations from above, and under the Skorokhod coupling $\Gamma^a\to\Gamma$,
\begin{equation}\label{eq.Conv H f eta a vers H f eta}
H^a_{f,\eta}\xrightarrow[a\to0]{(\P)}H_{f,\eta}.
\end{equation} 
\end{proposition}

The above convergence in probability~\eqref{eq.Conv H f eta a vers H f eta}  is the core of Theorem \ref{th. cond. expect. of magn wrt Gamma(eta)}; the proof will take us the rest of this Section.

\subsection{Splitting the conditional expectation into two parts}\label{ss.TwoCut-offs}

We introduce an $\eps$-cut-off which is used to control the convergence of the discrete one-point function $a^{-1/8} \<{\sigma_z}_{D^a \setminus \Expl^a(\eta)}^+$ towards $\pm \calC_\T CR(z, D \setminus \Gamma(\eta))^{-1/8}$. Indeed, to ensure this convergence, we shall rely on \cite{chelkak2015conformal} which requires us to be at sufficiently long distance from the conditioned explored domain.

Since we coupled everything using the Skorokhod embedding theorem, it is convenient to only define the geometric cut-off in the continuum.

\begin{definition}\label{d. thickened explored set}
For any $\eps>0$ we define $\Expl(\eta)^\eps$ to be the closed set defined as the union of the squares of side-length $\eps$ from the grid $\eps\Z^2$ which are at distance less than $\eps$ from $\Expl(\eta)$.
\end{definition} 

Now, we can define the $\eps$-cut-off approximations of $H_{f, \eta}^a$ and $H_{f,\eta}$.
 
\begin{definition}\label{d.H-f-eta-Regularized}
For any $\eps>0$, let 
\begin{align}\label{e.H eps1 eps2}
H_{f, \eta}^\eps(\Gamma):= \calC_\T\int_{D} \1_{z\not\in\Expl(\eta)^\eps}  f(z) \, (-1)^{\nes_{\Gamma(\eta)}(z)}  \,  CR_{D^\eta}(z)^{-1/8} dz.
\end{align}

Furthermore, for any $a>0$, we define in the discrete
\begin{equation}\label{eq.H a eps1 eps2}
H^{a,\eps}_{f,\eta} := a^{-1/8}\sum_{z\in D^a\setminus\Expl(\eta)^\eps}\bar{f}(z)\<{\sigma^a(z)}^+_{D_z^a}(-1)^{\nes^a(z)}.
\end{equation}
\end{definition}

\ni By Proposition \ref{pr.Def H f eta}, $H^\eps_{f,\eta}$ is well defined and measurable in $\Gamma$ for any $\eps>0$. Furthermore, $H_{f,\eta}^\eps\xrightarrow[\eps\to0]{a.s.} H_{f,\eta}$.

Our strategy is to prove that, in the decomposition
\begin{align}\label{eq.dec-H a f eta}
H^a_{f,\eta}=H^{a,\eps}_{f,\eta}+ \big(H^a_{f,\eta}-H^{a,\eps}_{f,\eta}\big),
\end{align}
the first part converges to $H^\eps_{f,\eta}$, and the second one is uniformly in $a$ bounded by a term that goes to $0$ when $\eps\to0$. 
Since it is technically easier, let us first deal with the regularized part.

\subsection{Convergence of the regularization $H_{f,\eta}^{a,\eps}$.}

We can use Theorem \ref{th.1pointFunction} to prove that the nicer part $H^{a,\eps}_{f,\eta}$ in \eqref{eq.dec-H a f eta} converges to its continuum counterpart:

\begin{proposition}\label{pr.NicePart}
For any $\eps>0$, we have $H_{f,\eta}^{a,\eps}\xrightarrow[a\to0]{a.s.} H_{f,\eta}^\eps$.
\end{proposition}

We will make use of all the convergence results of Proposition \ref{pr.JointConv} in the following proof. 

\begin{proof}
Almost surely, we have the following three a.e. convergences on $D$:
\begin{align*}
\mathbf{1}\big(\overline{D^a}\setminus\Expl(\eta)^\eps\big)(z) & \xrightarrow[a\to0]{}\mathbf{1}\left(D\setminus\Expl(\eta)^\eps\right)(z)  \\
 a^{-1/8}\<{\sigma^a_z}_{D^a_z}^+ & \xrightarrow[a\to0]{} \calC_\T CR(z,D_z)^{-1/8} \\
 f(z)(-1)^{\nes^a(z)} & \xrightarrow[a\to0]{} f(z)(-1)^{\nes(z)}
\end{align*}
\indent For the first one, it is a.s. true for any point $z$ outside $\partial\big(\Expl(\eta)^\eps\big)$, which is a set of Lebesgue measure 0 (indeed, it is piecewise affine).
\par For the second one (for which we interpret $\<{\sigma^a_z}_{D^a_z}^+$ as a constant function on each hexagon centered on an $a$-lattice point), almost surely for any $z$ outside $\Expl(\eta)$ the pointed filling of the discrete domain $(D^a_z,z)$ converges in the Carathéodory sense to the connected component $D_z$ of $z$ in $D^\eta$ so that we can apply Theorem \ref{th.1pointFunction} to obtain the convergence of the one-point function. Furthermore, $\Expl(\eta)$ is almost surely a set of Lebesgue measure 0 since it is the union of $\eta$ and at most countably many loops of $\Gamma(\eta)$ (these loops are of zero area since e.g. they are of upper Minkowski dimension $11/8$).
\par Finally, the third one is true as well for the same set of $z$ (those $z$ not in $\Expl(\eta)$).

\medskip
It can be observed that the integral of the product of the three functions on the left is $H^{a,\eps}_{f,\eta}$ (it is a rewriting of equation \eqref{eq.H a eps1 eps2}) and the integral of the product of the three on the right gives $H_{f,\eta}^\eps$ by definition.
\par Furthermore, since $f$ vanishes on a neighbourhood of the boundary of $D$, the product of these limiting functions is uniformly bounded on $\overline D$. The same is true for the analogous product of the discretized functions, since by Proposition \ref{pr.BoundTwoPoint}, $a^{-1/8}\<{\sigma^a_z}_{D^a_z}^+=O\big(\dist(z,D_z^a)^{-1/8}\big)=O\big(\dist(z,D_z)^{-1/8}\big)$ uniformly in $a$, which is, as above, bounded on the support of the test function $f$. Whence the dominated convergence theorem applies and implies the desired convergence.
\end{proof}



\subsection{Large deviation estimates on the magnetization with GHS inequality.}
Now, we want to handle the remaining part, $H^a_{f,\eta}-H^{a,\eps}_{f,\eta}$. To achieve this, it is important for us to have a control on the tail of the magnetization field. More precisely, we make use of the following concentration result:

\begin{proposition}[Concentration of the Discrete Magnetization]\label{pr.ConcentrationDiscreteMag}
There exists an absolute constant $A>0$ such that, for any continuous function $g:\overline{D}\to\mathbb R$, for all $0<a<\eps$, and for any (deterministic) square $C$ of side-length $\eps$ contained in $D$, one has

\begin{equation}\label{eq.ConcentrationMag}
    \P\Big( \big| \<{ \Phi^{a},g\vert_{ C} } \big|\ge \eps^{15/8}\|g\|_{\infty}x\Big)\le Ae^{-x}.
\end{equation}
This inequality holds (with the same constants) as well if we replace $\<{\Phi^{a},g\vert_{ C}}$ by its conditional expectation $\Eb{\<{ \Phi^{a},g\vert_{ C} }\md\caF}$ for any $\sigma$-field $\caF$.
\end{proposition}

The reason why this Proposition is useful is that it shows that the Ising magnetization field against our test function in a small box is concentrated and very rarely takes a large value, which is necessary to upper bound the remaining part of the cut-off procedure. We will actually only use it in the case $g=f$ to control the tail of \begin{equation*}
\Eb{\<{ \Phi^{a},f\vert_{ C} }\md \Gamma^a(\eta)},
\end{equation*}
that is, when $\caF$ is the $\sigma$-field of the r.v. $\Gamma^a(\eta)$. We decided to display inequality \eqref{eq.ConcentrationMag} with $\caF$ the trivial $\sigma$--algebra for clarity reasons.

\medskip To prove Proposition \ref{pr.ConcentrationDiscreteMag}, we will use the following proposition from \cite{Camia_2015} which also relied on GHS inequality.

\begin{proposition}[Proposition 3.5 in \cite{Camia_2015}]
For any $t>0$, one has \begin{equation}
\limsup_{a\to0}\E[e^{tm_+^a}]<+\infty
\end{equation}
where $m^a_+=\<{\Phi^a,1}=a^{15/8}\sum_{z\in a\Z^2\cap[0,1]^2}\sigma_z$ is the (renormalized) spontaneous magnetization of an Ising model $(\sigma_z)_{z\in a\Z^2\cap[0,1]^2}$ with + boundary condition.
\end{proposition}

Since the Ising measure has the FKG lattice property, the + boundary condition maximizes the expectation $\E[e^{tm_+^a}]$ so the Proposition holds uniformly in the b.c. This is crucial in what follows.

\begin{proof}[Proof of Proposition \ref{pr.ConcentrationDiscreteMag}] Assume first that $\caF$ is the trivial $\sigma$-field, so that we aim to prove the inequality \eqref{eq.ConcentrationMag}. Let $A>0$ such that $\E[e^{m_+^a}]<A/2$ for all $a>0$, with the notations of the previous proposition; $A$ will be the constant appearing in the desired inequality~\eqref{eq.ConcentrationMag}. We can easily deduce that \begin{equation}\label{eq.InegalIntermConcentrationMag}
\E\left[ \exp\left(\frac{\<{ \Phi^{a},g\vert_{ C}} }{\eps^{2-1/8}\|g\|_{\infty}}\right) \right]<A/2.
\end{equation}
Indeed, if one develops the exponential, by Griffiths' inequality all correlation functions are nonnegative and thus the sum is maximal when $g$ is constant equal to its supremum $\|g\|_{\infty}$, in which case it is equal to $\Eb{\exp\big(m^{a/\eps}\big)}$ with \emph{random} boundary conditions (given by the spins on the boundary of the square $C$) and where $m^{a/\eps}$ is the spontaneous magnetization with {those} random boundary conditions.
 Whence, by the discussion above before the proof, the inequality~\eqref{eq.InegalIntermConcentrationMag} must be true.

\medskip By Markov's inequality one gets \begin{align*}
\P\big(\<{\Phi^{a},g\vert_{ C}}\ge \eps^{15/8}\|g\|_{\infty}x\big)\le (A/2)e^{-x}. 
\end{align*}
Furthermore, one can show exactly in the same way the complementary inequality, that is $\P\big(\<{\Phi^{a},g\vert_{ C}}\le -\eps^{15/8}\|g\|_{\infty}x\big)\le (A/2)e^{-x}$, hence, by the union bound, one gets~\eqref{eq.ConcentrationMag}. 
\par To treat the case of general $\caF$, it is enough to show the analog of inequality \eqref{eq.InegalIntermConcentrationMag} with the random variable $\<{ \Phi^{a},g\vert_{ C}}$ replaced by its conditional expectation w.r.t. $\caF$. This is a general fact: indeed, by the conditional Jensen's inequality, a.s. $e^{\E[X\md\caF]}\le \Eb{e^X\md\caF}$ for any real random variable $X$, so that $\E\big[e^{\E[X\md\caF]}\big]\le\E[e^X]$.
\end{proof}

\subsection{Handling the remaining part $H_{f,\eta}^a-H^{a,\eps}_{f,\eta}$.}

Now that we have Propositions \ref{pr.dim-Expl(eta)} and \ref{pr.ConcentrationDiscreteMag} in hand, we are able to show that in the enlargement $\Expl^{\eps}(\eta)$ of $\Expl(\eta)$, the discrete magnetization field carries mass that vanishes when $\eps\to0$.

\begin{proposition}[Control of the remaining part]\label{pr.RemainingPartSmall}
For any $\delta>0$ there exists $\eps_0>0$ such that for any $0<a<\eps<\eps_0$, we have \begin{equation}
\P\left(\big| H^a_{f,\eta}-H_{f,\eta}^{a,\eps}\big| > \delta \right) < \delta.
\end{equation}
\end{proposition}

Before turning to the details of the proof, let us explain how we proceed. Recall that, for $\eps>0$, the set $\Expl(\eta)^{\eps}$ is defined as the union of the squares $C$ from the grid $\eps\Z^2$ which are at distance less than $\eps$ from $\eta$. Hence, one can write
\begin{equation}\label{eq.DiffAsCondExpectation}
H^a_{f,\eta}-H_{f,\eta}^{a,\eps} = \sum_{C}\Eb{\<{\Phi^a,f\vert_{C}} \md \Gamma^a(\eta)},
\end{equation}
where the sum is over such squares.

Now, to upper-bound the absolute value of this sum, we have two ingredients: Proposition \ref{pr.dim-Expl(eta)}, which tells us that the number of such squares is bounded from above with high probability, and Proposition \ref{pr.ConcentrationDiscreteMag} that says that on deterministic squares $C'$, the value $\Eb{\<{\Phi^a,f\vert_{C'}} \md \Gamma^a(\eta)}$ cannot be too large. However, we must be careful since our squares $C$ are not deterministic and we cannot use Proposition \ref{pr.ConcentrationDiscreteMag} directly.

\begin{proof}[Proof of Proposition \ref{pr.RemainingPartSmall}]
Let $\delta>0$. For $\eps>0$, let $N_\eps$ be the number of $\eps\times\eps$-squares $C$ in the sum above, which intersect the support of $f$. By equation \eqref{eq.DiffAsCondExpectation}, for the inequality $\big\lvert H^a_{f,\eta}-H_{f,\eta}^{a,\eps}\big\rvert>\delta$ to hold, 
\bi 
\item either $N_\eps>\eps^{-3/2}$,
\item or $\big\lvert\Eb{\<{\Phi^a,f\vert_{C}} \md \Gamma^a(\eta)}\big\rvert>\delta\eps^{3/2}$ for some square $C$. 
\ei
(The exponent $3/2$ has been chosen so that both events will be shown to have small probability, however any number lying in $(11/8,15/8)$ would work.)

By Proposition \ref{pr.dim-Expl(eta)}, a.s. $N_\eps=o(\eps^{-3/2})$. Hence, $\P(N_\eps>\eps^{-3/2})<\delta/2$ for any $\eps$ small enough and the first event has probability bounded by $\delta/2$ for $\eps<\eps_1$.

If the second event happens, then \emph{a fortiori} there must be an $\eps\times\eps$-square $C'$ among \emph{all} the $N_\eps'$ squares of the grid $\eps\Z^2$ at distance at most $\eps$ from the support of $f$ for which the inequality $\big\lvert\Eb{\<{\Phi^a,f\vert_{C'}} \md \Gamma^a(\eta)}\big\rvert>\delta\eps^{3/2}$ holds. ($N_\eps'$ is deterministic and of course of order $\eps^{-2}$.) By Proposition \ref{pr.ConcentrationDiscreteMag} applied to each of the squares $C'$\footnote{for $\eps$ small enough, any of the squares $C'$ is contained inside $D$ so that Proposition \ref{pr.ConcentrationDiscreteMag} applies} and the union bound, we can bound the probability of the second event:

\begin{align*}
\P\left( \left\{\Big| \Eb{\<{ \Phi^{a},f_{| C} }\md \Gamma^a(\eta)}\Big| > \delta\eps^{3/2} \text{ for some } C\right\}\right) \\ 
\le \P\left( \left\{\Big| \Eb{\<{ \Phi^{a},f_{| C'} }\md \Gamma^a(\eta)}\Big| > \delta\eps^{3/2} \text{ for some } C'\right\}\right) 
 \\  \le N'_\eps\times\bigg(A\exp\left(-\eps^{-15/8}\delta\eps^{3/2}/\|f\|_\infty\right)\bigg)
 \\ = O\big(\eps^{-2}\big)\exp\big(-\eps^{-3/8}\delta/\|f\|_\infty\big) < \delta/2&
\end{align*}
where the last inequality holds for $\eps$ small enough. 

It is straightforward to conclude: if $\eps>0$ is small enough,
\begin{align*}\P\left(\big\{| H^a_{f,\eta}-H_{f,\eta}^{a,\eps}| > \delta\big\}\right)<\delta,
\end{align*}by the union bound. This concludes the proof.
\end{proof}

\subsection{The end of the proof.}\label{ss.Proof PhiCondGamma(Eta)}
We are finally able to give the proof of Proposition \ref{pr.Conv H f eta a vers H f eta} on the convergence of $H_{f,\eta}^a$ to $H_{f,\eta}$. This proof will complete the proof of Theorem \ref{th. cond. expect. of magn wrt Gamma(eta)}, the main result of this section. 

\begin{proof}[Proof of Proposition \ref{pr.Conv H f eta a vers H f eta}]
Let $\delta>0$ be arbitrary. For $\eps$ sufficiently small, we have $\P\big(\lvert H_{f,\eta}^{\eps}-H_{f,\eta}\rvert>\delta\big)<\delta$ and $\P\big(|H_{f,\eta}^{a,\eps}-H^a_{f,\eta}|>\delta\big)<\delta$, the first inequality by Proposition \ref{pr.Def H f eta} and the second by Proposition \ref{pr.RemainingPartSmall}. For $a$ small enough depending on $\eps$, 
\begin{align*}
\P\left(\big\lvert H_{f,\eta}^{a,\eps}-H_{f,\eta}^{\eps}\big\rvert>\delta\right)<\delta
\end{align*}
by Proposition \ref{pr.NicePart}. Therefore, by the union bound, $\P\big(|H^a_{f,\eta}-H_{f,\eta}|>3\delta\big)<3\delta$ for small enough $a$. This concludes the proof of Proposition \ref{pr.Conv H f eta a vers H f eta} and thus of Theorem \ref{th. cond. expect. of magn wrt Gamma(eta)}.
\end{proof}

\section{Proof of the main Theorem.}
\label{se.Proof-Main-thm}

\subsection{Set-up of the proof of the main Theorem}
\label{ss.Set-up-mainThm}

Let us now build the set-up for the proof of our main Theorem \ref{th.main}. Recall that, for $f\in H_0^2(D)$, we have described explicitly the conditional expectation $H_{f,\eta}$ of $\<{\Phi,f}$ with respect to the set $\Gamma(\eta)$ of those loops intersecting a given smooth curve $\eta$, and that we have shown that this conditional expectation is the limit in law of the analogous conditional expectation in the discrete, jointly with the set of $+/-$ interfaces - that was the content of Section \ref{se.IMFwrtCurve} and Theorem \ref{th. cond. expect. of magn wrt Gamma(eta)}.
\par Recall as well that our goal is to show that $\<{\Phi,f}$ is completely determined by $\Gamma$, so we would like to take $\eta$ so that $\Gamma(\eta)$ contains much information.

\bnum
\item We want to show that the difference $\<{\Phi,f} - H_{f,\eta}$ is small for some well-chosen $\eta$;
\enum
Theorem \ref{th. cond. expect. of magn wrt Gamma(eta)} shows that it is enough to argue only in the discrete. Indeed, thanks to this theorem, if one assumes that on a Skorokhod coupling and on some subsequence the random pair $(\<{\Phi^a,f},\Gamma^a)$ converges a.s. to $(\<{\Phi,f},\Gamma)$, one has that $H_{f,\eta}^a\to H_{f,\eta}$ (in probability) so that if we show some bound on $\<{\Phi^a,f}-H^a_{f,\eta}$ uniform in $a$, we will recover a bound on $\<{\Phi,f} - H_{f,\eta}$. This is actually the only place where we make use of Section \ref{se.IMFwrtCurve} and Theorem \ref{th. cond. expect. of magn wrt Gamma(eta)}.

\bnum
\setcounter{enumi}{1}
\item Instead of extracting the information from the loops touching a given curve, we will extract information on the magnetization field $\Phi$ from the loops intersecting the (smooth) grid $\eta(\rho)= \rho\Z^2\cap D$ for a mesh-parameter $\rho$ going to $0$ along the proof.
\enum
Here and in the following, $\rho\Z^2\cap D$ is meant to include the edges between the $\rho$--lattice points, i.e. $\eta(\rho)=\{(x,y)\in D\md x\in\rho\Z\text{\emph{ or }}y\in\rho\Z\}$ which is a finite union of straight lines. Applying Theorem \ref{th. cond. expect. of magn wrt Gamma(eta)} to $\eta(\rho)$ is the reason why we stated it with $\eta$ allowed to be a smooth grid. 
\par Furthermore, since for the proof we will only apply it to these special grids $\eta(\rho)$ for $\rho>0$, let us abbreviate the notation $\Gamma\big(\eta(\rho)\big),H_{f,\eta(\rho)}$ as $\Gamma_\rho,H_{f,\rho}$. This will be convenient for future statements where $\rho\to0$.

\bnum
\setcounter{enumi}{2}
\item In order to run the argument above, we will need to prove that when $\rho$ is tiny, $\<{\Phi^a,f}$ is in fact very close to $\Eb{\<{\Phi^a,f} \md \Gamma^a_\rho}$. 
\enum
The precise statement that we prove is Proposition \ref{pr.SmallHoles} below.


\subsection{Control of the information left in the holes.}\label{ss.InfoHoles}
In this subsection, we wish to prove the following estimate.

\begin{proposition}\label{pr.SmallHoles} 
Let $f\in H_0^2(D)$ as above, and for $\rho>0$, let $\eta(\rho)$ as above be the smooth grid $\rho\Z^2\cap D$. Then, denoting $\Gamma^a_\rho=\Gamma^a\big(\eta(\rho)\big)$, one has \begin{equation}\label{eq.SmallHolesProba}
\limsup_{\rho\to0}\limsup_{a\to0}\EB{\Var{ \<{\Phi^a, f} \ \big\vert \ \Gamma^a_\rho}}= 0.
\end{equation}
\end{proposition}

Let us stress that it is the key step that remains after Section \ref{se.IMFwrtCurve}. There will be no difficulty to prove Theorem \ref{th.main} once we have this proposition, see Subsection \ref{ss.ProofMainThm}. The proof we give builds on the monotonicity in the domain of the covariances of the Ising model, Lemma \ref{l.covGHS}, a consequence of GHS inequality. This fact was already helpful for proving the tightness of the field. Notice that another consequence of GHS inequality was crucial in Section \ref{se.IMFwrtCurve}.



\begin{proof}
It will be convenient to couple things using Skorokhod once again: let us assume that $\Gamma^a\to\Gamma$, and we will show the stronger statement that almost surely one has \begin{align}\label{eq.SmallHolesL2}
\limsup_{\rho\to0}\limsup_{a\to0}\Var{ \<{\Phi^a, f}  \md \Gamma^a_\rho}
= 0.
\end{align}
(This is indeed sufficient since the $\<{\Phi^a,f}$ are uniformly exponentially integrable.)

Let $(D_i^a)$ be the connected components of $D^a\setminus\Expl^a_\rho$. Since the various spin configurations $\sigma^a\vert_{D_i^a}$ are independent conditionally on $\Gamma^a_\rho$, the term above decomposes in the following way: \begin{equation}
        \Var{\<{\Phi^a,f} \md \Gamma^a_\rho} =  \sum_{i}\Var{\<{\Phi^a,f\vert_{D_i^a}}}.
    \end{equation}
Any term of this sum can be bounded by \begin{multline*}
         \Var{\<{\Phi^a,f\vert_{D_i^a}}}=\iint_{D_i^a\times D_i^a}f(z)f(z')\cov_{D_i^a}(\sigma_z^a,\sigma^a_{z'})a^{-1/4}d zd z' \\ 
         \underbrace{\le}_{\text{Lemma }\ref{l.covGHS}} \|f\|_\infty^2\iint_{D_i^a\times D_i^a}\cov_{a\T}(\sigma_z^a, \sigma_{z'}^a)a^{-1/4}d zd z' =  \|f\|_\infty^2\iint_{D_i^a\times D_i^a} \langle\sigma^a_z\sigma^a_{z'}\rangle_{a\T}a^{-1/4}d zd z' 
    \end{multline*}
where the integrated functions are meant to be constant on each $a \times a$ square centered on a lattice point of $D^a_i$. Now, by the classical bound of the two-point function on the full plane (see Proposition \ref{pr.BoundTwoPoint}) we deduce \begin{equation*}
\frac1{\|f\|_\infty^2} \mathrm{Var}\Big[\<{\Phi^a,f}\md\Gamma^a_\rho\Big]\le \sum_{i}\iint_{D_i^a\times D_i^a} \langle\sigma^a_z\sigma^a_{z'}\rangle_{a\T}a^{-1/4}d zd z'
        \lesssim\sum_{i}\iint_{D_i^a\times D_i^a}\|z-z'\|^{-1/4} dzdz'
    \end{equation*}
Finally, we can write \begin{equation*}
        \sum_{i}\iint_{D_i^a\times D_i^a}\|z-z'\|^{-1/4}d zd z' \le \iint_{D\times D}\|z-z'\|^{-1/4}1(|z-z'|<\sqrt2\rho)d zd z'
    \end{equation*} since all the components $D_i^a$ are of diameter smaller than $\sqrt2\rho$ (they must be inside a $\rho\times\rho$ square). However, this last term goes to zero as $\rho\to0$ since $\|z-z'\|^{-1/4}dzdz'$ is integrable on $D\times D$, which proves the convergence in equation \eqref{eq.SmallHolesL2}.
\end{proof}

\subsection{End of the proof.}\label{ss.ProofMainThm}
Now, we can gather Theorem \ref{th. cond. expect. of magn wrt Gamma(eta)} and Proposition \ref{pr.SmallHoles} to give the proof of our main theorem.

\begin{proof}[Proof of Theorem \ref{th.main}]
We will show that any sublimit $(\Gamma,\Phi)$ of the discrete coupling $(\Gamma^a,\Phi^a)$ satisfies $\Phi=\Phi(\Gamma)$ for an explicit measurable function $\Phi$; by tightness of the magnetization field (Proposition \ref{pr.PhiTight}), this implies the joint convergence and the measurability of the IMF in terms of the set of interfaces. In addition, it gives a new proof of the convergence of the magnetization field.
\par So let $(\Gamma,\Phi)$ be a sublimit, and assume by Skorokhod that it is the a.s. limit of the discrete set of interfaces - magnetization. Let $f$ be a function in $H_0^2(D)$. By Proposition \ref{pr.SmallHoles}, for any $\delta>0$ one can find $\eps_0>0$ such that for $\eps<\eps_0$ and small enough $a$, $\Pb{|H^a_{f,\eps} -\<{\Phi^a,f}|^2>\delta}<\delta$. Letting $a\to0$ shows that $\Pb{|H_{f,\eps}-\<{\Phi,f}|>\delta}<\delta$ as well, by Proposition \ref{pr.Conv H f eta a vers H f eta}. Therefore, the sequence $\big(H_{f,\eps}\big)_{\eps}$ converges in probability to $\<{\Phi,f}$, which is hence a measurable function of $\Gamma$.
\par Now, let $(e_i)$ be any orthonormal basis of the Hilbert space $H^{2}_0(D)$ made of smooth functions. Applying the above to $f=e_i$ for every $i$, Lemma \ref{le.MeasurableCoeff} shows that $\Phi$ is an explicit measurable function of $\Gamma$. This concludes the proof.
\end{proof}

\section{A stronger result under some regularity of $\partial D$.}
\label{s.RoughBoundary}

As mentioned in the Introduction and Subsection \ref{ss.IMF}, though we stated our main theorem without any assumption on the Jordan domain $D$, there is a slight improvement in the case where the boundary is not too rough, in the following precise sense. 

\begin{theorem}\label{th.main-smooth-bdry}
Let $D$ be a Jordan domain whose closure is contained in $\Omega$, a Jordan domain of the plane.
With the notations of Theorem \ref{th.main}, seeing the distributions $\Phi^a$ as elements of the Sobolev space $H^{-2}(\Omega)$\footnote{instead of elements of $H^{-2}(D)$ as in Definition \ref{d.Phi} and Theorem \ref{th.main}} \emph{and under the assumption that $\p D$ has Minkowski dimension strictly less than $15/8$}, the joint coupling $(\cle_3^{\nes,a}, \Phi^a)$ converges in law as the mesh $a\to 0$ to the continuum coupling $(\cle_3^{\nes}, \Phi)$ in the topology of $\caX_D\times H^{-2}(\Omega)$ where the magnetization field $\Phi$ is a measurable function of the nested $\cle_3$.
\end{theorem}

\begin{remark}\label{rk.reason-Omega}
The first reason why we want our fields to be elements of $H^{-2}(\Omega)$ is to be able to integrate against test functions which do not necessarily vanish on $\partial D$. For example, if we indeed see $\Phi$ as an element of $H^{-2}(\Omega)$, the magnetization $\<{\Phi^a,\1}$ is well-defined as $\<{\Phi^a,f}$ for any function $f\in H^{2}_0(\Omega)$ which is identically $1$ on $D$, as long as $\overline{D}\subset\Omega$. Such a test function was not allowed in the previous setting of Theorem \ref{th.main} when we instead considered the field $\Phi^a$ as an element of the somehow ``larger'' space $H^{-2}(D)$, whence Theorem \ref{th.main-smooth-bdry} is an improvement of Theorem \ref{th.main}. It is of no importance what $\Omega$ really is, as long as it contains $\overline{D}$, since our distributions have compact support $\overline D$.
\end{remark}

\begin{remark}\label{rk.reason-Omega-2}
Another reason for having a theorem in which the limiting object lives in a Sobolev space that does not depend on the domain is that one can cut our domains into smaller subdomains and ask for a Markov property for the IMF. For an example of such a statement, see Theorem \ref{th.MarkovIMF}.
\end{remark}

In order to prove Theorem \ref{th.main-smooth-bdry}, we have to prove first the tightness of the fields $\Phi^a$ in the topology of $H^{-2}(\Omega)$:

\begin{proposition}[Tightness of the IMF in $H^{-2}(\Omega)$]\label{pr.PhiTight-smooth-bdry}
With the notations of Theorem \ref{th.main-smooth-bdry} and under the assumption that $\p D$ has Minkowski dimension strictly less than $15/8$, the sequence of random distributions $\Phi^a$ is tight in $H^{-2}(\Omega)$. 
\end{proposition}

\begin{proof}
First, note that the analogous tightness in $H^{-2}(\Omega)$ with the centered field $\widehat{\Phi}^a$ holds true\footnote{recall we defined $\widehat{\Phi}^a=\Phi^a-\E[\Phi^a]=\big(z\mapsto a^{-1/8}\cdot(\sigma^a_z-\E[\sigma^a_z])\big)$ in Section \ref{ss.IMF}}, even without the hypothesis on $\p D$. Indeed, the proof of Proposition \ref{pr.Phi-Centered-Tight} on tightness of the centered field in the space $H^{-2}(D)$ could be written down with $\Omega$ instead of $D$, without any changes: it only relies on $G_D(u,v)=O\big(-\log(\|u-v\|)\big)$, which is true for the Green kernel $G_\Omega$ on $\Omega$ as well.

Hence, the only point which we have to modify is that, instead of proving the finiteness of $I_D(D)$ defined in \eqref{e.I_D(D)}, we need the finiteness of the following integral:
\begin{align}\label{e.I_Om}
I_\Omega(D):=\iint_{D\times D}G_\Omega(u,v)\dist(u,\p D)^{-1/8}\dist(v,\p D)^{-1/8}dudv \,.
\end{align}
This is precisely the content of Lemma \ref{l.int-finite-smoothBdry} proven in Appendix \ref{ap.Green-bounds-proof-finiteness}.
\end{proof}

Furthermore, the following modification of Proposition \ref{pr.dim-Expl(eta)} is needed:

\begin{proposition}\label{pr.dim-Expl(eta)-2}
Let $\Gamma\sim\cle_3^\nes$ in $D$ a Jordan domain, and let $\eta$ be a deterministic
smooth grid in $D$. Set $F_\eta=\partial D\cup\mathsf{Expl}(\eta)$.
Then, almost surely,
\begin{equation*}
 \overline{\dim}_{\mathrm M}(F_\eta)
  \leq \max\bigl\{11/8,\overline{\dim}_{\mathrm M}(\partial D)\bigr\}.
\end{equation*}
\end{proposition}

\begin{proof}
The proof is very similar to the one of Proposition \ref{pr.dim-Expl(eta)}. What differs is that, instead of grouping the $\eps$-squares in terms of their dyadic distances to $\eta$, we group them in term of their dyadic distances to $\eta\cup\p D$.
\par Let $s>\max\{11/8,\overline{\dim}_{\mathrm M}(\partial D)\}$; if $r$ is a small scale, there are $O(r^{-s})$ $r$-squares (i.e. unit squares of the grid $r\Z^2$) that intersect $\p D\cup\eta$. Now, we want to bound in expectation the number $N_\eps$ of $\eps$-squares which cover $F_\eta$. Group the squares in term of their distance to $\eta\cup\p D$.
\bi 
\item There are $O(\eps^{-s})$ squares at distance at most $4\eps$ from $\eta\cup\p D$
\item For any $3\le k\le \log_2(\eps^{-1})$ we have $O(2^{2k}(2^k\eps)^{-s})$ squares at distance between $2^{k-1}\eps$ and $2^k\eps$ from $\eta$. Indeed, each of these squares is contained in a $(2r)$-square that intersect $\eta\cup\p D$ where $r=2^k\eps$, there are $O(r^{-s})$ such $(2r)$-squares, and each of them contain $(2r/\eps)^2=4\times 2^{2k}$ $\eps$-squares.
\ei
Furthermore, an $\eps$-square at distance $\asymp r$ of $\eta\cup\p D$ has probability $O((\eps/r)^{5/8})$ of intersecting $F_\eta$ for the same reasons than the ones explained in the proof of Proposition \ref{pr.dim-Expl(eta)}. Hence, we can bound the expectation of $N_\eps$ by 
\begin{equation*}
\E[N_\eps]\lesssim \eps^{-s}+\sum_{k\ge 2}2^{2k}(2^k\eps)^{-s}(2^{-k})^{5/8}=\eps^{-s}\cdot\Big(1+\sum_{k\ge 2}2^{(2-s-5/8)k}\Big)=O(\eps^{-s}).
\end{equation*}
Therefore, by a Borel-Cantelli argument, the upper Minkowski dimension of $F_\eta$ is at most $s$, and letting $s\searrow \max\{11/8,\overline{\dim}_{\mathrm M}(\partial D)\}$ gives the result.
\end{proof}

\begin{proof}[Proof of \ref{th.main-smooth-bdry}]
Now that we have proven the tightness of the discrete IMF in $H^{-2}(\Omega)$, what remains is to check that Theorem \ref{th. cond. expect. of magn wrt Gamma(eta)} and Proposition \ref{pr.SmallHoles}, the core ingredients of our main Theorem \ref{th.main}, are still true if the test function $f$ is an element of $H^2_0(\Omega)$ instead of an element of $H^2_0(D)$.

Theorem \ref{th. cond. expect. of magn wrt Gamma(eta)} can be itself split into three main pieces: Proposition \ref{pr.Def H f eta} showing that the natural integral defining $H_{f,\eta}$ indeed converges, Proposition \ref{pr.NicePart} showing that the "nice part" $H_{f,\eta}^{a,\eps}$ of $H_{f,\eta}^a$ converges to its continuum analog $H_{f,\eta}^\eps$, and Proposition \ref{pr.RemainingPartSmall} justifying that the remaining part does not carry much mass.

Let us check that all the steps listed above still work through in the case of $f\in H^2_0(\Omega)$:

\bnum
\item To prove Proposition \ref{pr.SmallHoles}, the only assumption on $f\in H^2_0(D)$ we use is the fact that $f$ is bounded (which is a consequence of the continuity of $f$). This is still true for $f\in H^2_0(\Omega)$.
\item For Proposition \ref{pr.RemainingPartSmall}, using the same notations than th ones used in its proof, we used the fact that relevant squares $C'$ were those intersecting the support of $f$ to apply Proposition \ref{pr.ConcentrationDiscreteMag} (which itself did not use the fact that $f$ vanished on the boundary).
\enum
However, in the case that $\overline{\dim}_{\mathrm M}(\p D)<15/8$, the additionnal squares $C'$ that intersect the boundary of $D$ and cannot be handled by Proposition \ref{pr.ConcentrationDiscreteMag} bring a vanishing mass as $\eps\searrow0$.
\bi
\item[3)] For $H_{f,\eta}$ being well defined, and the convergence of the nice part, i.e. Propositions \ref{pr.Def H f eta} and \ref{pr.NicePart}, this is more subtle since we did use the fact that our test function $f$ vanished on a neighbourhood of $\p D$.
\ei
Under the hypothesis $\dimM(\p D)<15/8$ this can be fixed smoothly: the function $CR(z,D^\eta)^{-1/8}$ is integrable over $D$ by Proposition \ref{pr.dim-Expl(eta)-2} and Lemma \ref{l.dist^-1/8 int}. Hence the integral defining $H_{f,\eta}$ converges, and we may use the dominated convergence theorem in the proof of the convergence $H_{f,\eta}^{a,\eps}\to H_{f,\eta}^\eps$.
\end{proof}

\section{Some consequences of the measurability of the Field.}
\label{se.MainThm-consequences}

Now that we have finally proven our main result, the measurability in the set of interfaces of the Ising Model of the Magnetization Field, we push our analysis a little further in this Section in order to obtain a spatial Markov property for the IMF, and the explicit description of the conditional expectation of $\<{\Phi,f}$ w.r.t. $\Gamma^1$ the set of loops of the first layer. Let us begin with the Markov Property.

\subsection{The Markov Property for the IMF w.r.t. the loops in $\Gamma(\eta)$.}\label{ss.Markov for IMF}
We give a more satisfying version of Theorem \ref{th. cond. expect. of magn wrt Gamma(eta)}, the main result of Section \ref{se.IMFwrtCurve}. Let us recall that Theorem \ref{th. cond. expect. of magn wrt Gamma(eta)} gives an explicit description of the conditional expectation of the IMF w.r.t. the loops discovered by the exploration of a smooth grid $\eta$\footnote{or rather the IMF integrated against any test function, but it amounts to the same}, and allowed us to prove our main result, namely the measurability in the loops of the IMF.

Now that we have this latter result, we can actually come back to Theorem \ref{th. cond. expect. of magn wrt Gamma(eta)} and prove a slight improvement: the Markov Property for the magnetization field w.r.t. the "stopping set" $\Gamma(\eta)$. We are able to describe the conditional law of the IMF w.r.t. $\Gamma(\eta)$, and it is precisely what one could expect, an IMF in each of the subdomains cut out by the exploration $\Gamma(\eta)$.

Before we state it properly, let us recall and introduce some notations that we use: 
\begin{itemize}
\item $D$ is a Jordan domain, assumed to be of piecewise smooth boundary; $\Omega$ is a Jordan domain containing $\overline D$. Furthermore, $\eta$ is a smooth grid in $\overline D$.
\end{itemize}
It is irrelevant what $\Omega$ is, one may think of it as a large ball.
\bi
\item If $\Gamma$ follows the law $\cle_3^\nes$ in $D$, let $\Gamma(\eta)=\{\ell\in\Gamma\md\im(\ell)\cap\eta\neq\emptyset\}$. 
\end{itemize}
Define as well $\Expl(\eta)=\eta\cup\bigcup_{\ell\in\Gamma(\eta)}\im(\ell)$, and $D^\eta=D\setminus\Expl(\eta)$. We write $D_i$ for the connected components of $D^\eta$, and $\Gamma_i=\{\ell\in\Gamma\md\im(\ell)\subset D_i\}$ for the loops inside $D_i$. \emph{It is a classical fact that, conditionally on $\Gamma(\eta)$, the random variables $\Gamma_i$ are independent and each of them follows the law of a $\cle_3^\nes$ in its domain $D_i$.}
\bi 
\item Let $\nes: D^\eta\to\Z_+$ be the index of nestedness: for $z\not\in\Expl(\eta)$, $\nes(z)$ is the number of loops $\ell\in\Gamma(\eta)$ surrounding $z$. We write $\nes^\eta$ when we want to stress with respect to which set of loops the nesting index is taken with.
\ei
A.s. $\nes$ is a continuous function, hence constant on the connected components of $D^\eta$. Let $\nes_i$ be its value on $D_i$.

Let us now state the Markov property of $\Phi$ w.r.t. the stopping set $\Gamma(\eta)$:

\begin{theorem}[Markov property for the IMF]\label{th.MarkovIMF}
Under the notations above, let $(\Gamma,\Phi)$ be the coupling between a $\cle_3^\nes$ and an IMF given by Theorem \ref{th.main-smooth-bdry}. The law of $\Phi$ conditionally on $\Gamma(\eta)$ is the law of an independent IMF in each of the connected components of $D^\eta$, multiplied by the spatial function $(-1)^\nes$.

More precisely, for any test function $f\in H_0^2(\Omega)$, we have the almost sure equality 
\begin{align}\label{eq.Phi eta decomposedSmallDomains}
\<{\Phi,f}\stackrel{a.s.}{=} \sum_i (-1)^{\nes_i}\<{\Phi_i,f},
\end{align}
where $\Phi_i=\Phi(\Gamma_i)$ is the magnetization field in $D_i$ associated to the loop set $\Gamma_i$. (The right hand side is a convergent serie thanks to Point \ref{point.proof-serie-cv} below.)
\end{theorem}

Note that by Proposition \ref{pr.dim-Expl(eta)-2} almost surely the Minkowski dimension of the boundary of any of the domains $D_i$ is at most $11/8$, and furthermore, by Lemma \ref{l.Domains-to-be-explored-are-Jordan} the domains $D_i$ are a.s. all Jordan. Thus, Theorem \ref{th.main-smooth-bdry} applies and we can make sense of the $\Phi_i$ acting against a test function $f\in H^2_0(\Omega)$.

Though not particularly difficult, the proof is quite technical and requires a careful analysis. We decompose it in a few steps: 
\bnum
\item\label{point.eq is enough} First we explain why it is enough to prove equation \eqref{eq.Phi eta decomposedSmallDomains}.
\item\label{point.eq for special f} Then, we prove \eqref{eq.Phi eta decomposedSmallDomains} for test functions $f$ vanishing on a neighbourhood of $\eta$. (For these functions, the right hand side has finitely non-zero terms.)
\item\label{point.proof-serie-cv} We prove that, for general $f$, the sum in \eqref{eq.Phi eta decomposedSmallDomains} indeed converges in probability.
\item\label{point.eq for any f} Finally, we prove \eqref{eq.Phi eta decomposedSmallDomains} for any test function $f$.
\enum

\begin{proof}[Proof of \ref{point.eq is enough}]
Conditionnally on $\Gamma(\eta)$, the random loop sets $\Gamma_i$ are independent and follow the law of a $\cle_3^\nes$ in their respective domains, hence the $\Phi_i$ are independent and follow the law of IMFs in their respective domains. Now, assuming \eqref{eq.Phi eta decomposedSmallDomains} for a dense countable family $(f_n)$ of test functions, we get the equality of distributions \begin{align*}
\Phi \stackrel{a.s.}{=} \sum_i(-1)^{\nes_i}\Phi_i,
\end{align*}i.e. the Markov Property stated in the Theorem.
\end{proof}

\begin{proof}[Proof of \ref{point.eq for special f}]
Let $f\in C^\infty_0(\Omega)$ which vanishes on a neighbourhood of $\eta$. We prove \eqref{eq.Phi eta decomposedSmallDomains} for $f$.

Let $K:=\operatorname{supp}(f)\cap\overline D$. Since $K$ is at positive distance from $\eta$, only finitely many components $D_i$ meet $K$. Indeed, apart from the finitely many components of $D\setminus\eta$ meeting $K$, every additional component meeting $K$ is separated from $\eta$ by a loop of $\Gamma(\eta)$ whose diameter is at least $\dist(K,\eta)$; the conclusion follows from the local finiteness of the $\cle_3$. In what follows, all sums are restricted to these finitely many indices.

For $n\geq1$, let
$$
\gamma_n:=2^{-n}\Z^2\cap\overline D,
\qquad
H_{f,\eta,n}:=H_{f,\eta\cup\gamma_n},
\qquad
\mathcal G_n:=\sigma\big(\Gamma(\eta\cup\gamma_n)\big).
$$
For a component $D_i$, denote by $H^{D_i}_{f,n}(\Gamma_i)$ the functional constructed in the domain $D_i$ from the loops of $\Gamma_i$ intersecting the grid $\gamma_n$.

We first observe that all these functionals may be represented by their unregularized integrals. Indeed, Proposition~\ref{pr.dim-Expl(eta)-2}, applied to the smooth grid $\eta\cup\gamma_n$, gives
$$
\overline{\dim}_M\big(\partial D\cup\Expl(\eta\cup\gamma_n)\big)\leq\frac{11}{8}
\qquad\text{almost surely.}
$$
The same statement holds, conditionally on $\Gamma(\eta)$, for the explored set generated by $\gamma_n$ inside each of the finitely many domains $D_i$. Since $11/8<15/8$, a Whitney-decomposition estimate gives the integrability of the corresponding function $\dist(\,\cdot\,,F)^{-1/8}$. Thus dominated convergence in Definition~\ref{d.H-f-eta-Regularized} shows that $H_{f,\eta,n}$ and $H^{D_i}_{f,n}$ coincide with the integrals in~\eqref{eq.H f eta}.

The integral defining $H_{f,\eta,n}$ can now be split according to the components $D_i$. Inside any of the $D_i$, the components created by the further exploration of $\gamma_n$ are exactly those used to define $H^{D_i}_{f,n}(\Gamma_i)$, and the global nesting index $\nes^{\eta\cup\gamma_n}$ in $D$ with respect to the set of loops $\Gamma(\eta\cup\gamma_n)$, restricted to the component $D_i$, is the sum of the constant $\nes_i^{\eta}$ and the nesting index $\nes_{D_i}^{\gamma_n}$ internal to $D_i$.
Therefore
\begin{align}
H_{f,\eta,n}(\Gamma)
=\sum_i(-1)^{\nes_i}H^{D_i}_{f,n}(\Gamma_i)
\qquad\text{almost surely.}
\label{eq.H-Gamma(eta) decomposedSmallDomains}
\end{align}

The dyadic grids are increasing, hence $(\mathcal G_n)_{n\geq1}$ is an increasing sequence of $\sigma$--fields. Moreover,
$$
\bigvee_{n\geq1}\mathcal G_n=\sigma(\Gamma),
$$
because every non-degenerate loop intersects $\gamma_n$ for all sufficiently large $n$. By Theorem~\ref{th. cond. expect. of magn wrt Gamma(eta)},
$$
H_{f,\eta,n}=\Eb{\<{\Phi,f}\md\mathcal G_n}.
$$
Since $\Phi$ is measurable with respect to $\Gamma$, the $L^2$ martingale convergence theorem gives $H_{f,\eta,n}\xrightarrow[n\to\infty]{L^2}\<{\Phi,f}$. 

Conditionally on $\Gamma(\eta)$, the same argument applies in each of the finitely many domains $D_i$: the $\sigma$--fields generated by the loops of $\Gamma_i$ intersecting the dyadic grids increase to $\sigma(\Gamma_i)$, and $\Phi_i$ is measurable with respect to $\Gamma_i$. Hence
\begin{align*}
H^{D_i}_{f,n}(\Gamma_i)\xrightarrow[n\to\infty]{L^2}\<{\Phi_i,f}
\qquad\text{for every relevant }i.
\end{align*}
Taking the limit $n\to\infty$ in~\eqref{eq.H-Gamma(eta) decomposedSmallDomains}, using the finiteness of the sum, proves~\eqref{eq.Phi eta decomposedSmallDomains}.
\end{proof}

\begin{proof}[Proof of \ref{point.proof-serie-cv}]
We prove that for any $f\in H^2_0(\Omega)$, the sum $\sum_i\<{\Phi_i,f}(-1)^{n_i}$ converges in probability. More precisely, we show that, if $$m_i(f):=\E[\<{\Phi_i,f}\md\Gamma(\eta)]\qquad\text{and}\qquad X_i(f):=\<{\Phi_i,f}-m_i(f),$$
then 
\bnum[i)]
\item the serie $\sum_i (-1)^{n_i}m_i(f)$ is absolutely convergent, almost surely on $\Gamma(\eta)$;
\item the serie $\sum_i (-1)^{n_i}X_i(f)$ is convergent in $L^2$.
\enum
For point $i)$ above, let us remark that $m_i(f)=\calC_\T\int_{D_i}CR(z,D_i)^{-1/8}f(z)dz$ (this is e.g. a consequence of Theorem \ref{th. cond. expect. of magn wrt Gamma(eta)} applied to $D_i$ and the empty grid), so that 
\begin{align*}
\sum_i|(-1)^{n_i}m_i(f)|\le\calC_\T\|f\|_\infty \sum_i \int_{D_i}CR(z,D_i)^{-1/8}dz \\ \le  \calC_\T\|f\|_\infty\int_D\dist(z,D^\eta)^{-1/8}dz<+\infty
\end{align*}
by Proposition \ref{pr.dim-Expl(eta)-2} and Lemma \ref{l.dist^-1/8 int}. Hence the sum is absolutely convergent, and we define $M_\eta(f)=\sum_i(-1)^{n_i}m_i(f)$. Notice that a.s. \begin{align}\label{eq.bound-M_eta}
|M_\eta(f)|\le \int_D|f(z)| \dist(z,\p D^\eta)^{-1/8}dz,
\end{align}
and that $M_\eta$ is a.s. linear.

For point ii), let us first record the uniform second-moment estimate that will be used below. If $U$ is a Jordan domain to which Theorem~\ref{th.main-smooth-bdry} applies, if $\Phi_U$ denotes the IMF in $U$, and if $h$ is a bounded smooth function, then
\begin{align}
\mathrm{Var}\big(\<{\Phi_U,h}\big)
&\lesssim \iint_{U\times U}|h(z)h(w)|\,\|z-w\|^{-1/4}\,dz\,dw .
\label{eq.MarkovIMF-covariance-bound}
\end{align}
The implicit constant is independent of $U$. Indeed, the analogous estimate for the discrete field follows from Lemma~\ref{l.covGHS} and Proposition~\ref{pr.BoundTwoPoint}, exactly as in the proof of Proposition~\ref{pr.Phi-Centered-Tight}; the continuum estimate then follows by passing to the limit and applying Fatou's lemma, using the convergence of the means. 

In our case, since conditionnally on $\Gamma(\eta)$, the $X_i(f)$ are centered and independent, for any finite set of indices $J$ we have \begin{align*}
\Eb{\big(\sum_{i\in J}(-1)^{n_i}X_i(f)\big)^2\md\Gamma(\eta)} & = \sum_{i\in J}\Eb{X_i(f)^2\md\Gamma(\eta)}
\\ & \lesssim \sum_{i\in J}\iint_{D_i\times D_i}|f(z)f(w)|\|z-w\|^{-1/4}dzdw 
\\ & \le \iint_{D\times D}|f(z)f(w)|\|z-w\|^{-1/4}dzdw 
\end{align*}
where we use \eqref{eq.MarkovIMF-covariance-bound} for the second line.

Taking the expectation over $\Gamma(\eta)$ and letting $J$ increases to the full set of indices gives the convergence in $L^2$ of $\sum_i(-1)^{n_i}X_i(f)$. Let $X_\eta(f)$ be its limit; we have \begin{align}\label{eq.bound-X_eta}
\E[X_\eta(f)^2]\le\iint_{D\times D}|f(z)f(w)|\|z-w\|^{-1/4}dzdw
\end{align}
and notice that the functional $X_\eta$ is a.s. linear.

Finally, the sum $\sum_i(-1)^{n_i}\<{\Phi_i,f}$ converges to $S_\eta(f):=M_\eta(f)+X_\eta(f)$. 
\end{proof}

\begin{proof}[Proof of \ref{point.eq for any f}]
Let now $f\in C^2_0(\Omega)$ be arbitrary. We aim to show that $\<{\Phi,f}=S_\eta(f)$, an equality that holds for any $f$ vanishing on a neighbourhood of $\eta$.

Choose deterministic functions $g_\epsilon\in C^\infty_c(\Omega)$ such that $0\leq g_\epsilon\leq1$, $g_\epsilon=0$ on the $\epsilon$--neighbourhood of $\eta$, and $g_\epsilon=1$ on $D$ outside the $2\epsilon$--neighbourhood of $\eta$. Set
$$
h_\epsilon:=f(1-g_\epsilon),
\qquad
A_\epsilon:=\{z\in D:\dist(z,\eta)<2\epsilon\}.
$$
Then $\operatorname{supp}(h_\epsilon)\subset A_\epsilon$ and $\|h_\epsilon\|_\infty\leq\|f\|_\infty$.

We show that $S_\eta(fg_\eps)\xrightarrow[\eps\to0]{(\P)}S_\eta(f)$, and $\<{\Phi,fg_\eps}\xrightarrow[\eps\to0]{L^2}\<{\Phi,f}$. Since $fg_\epsilon$ vanishes in a neighbourhood of $\eta$, Point \ref{point.eq for special f} allows to conclude.

Since a smooth grid has zero Lebesgue measure, the integrability of $(z,w)\mapsto\|z-w\|^{-1/4}$ on $D\times D$ and of $z\mapsto\dist(z,\partial D)^{-1/8}$ on $D$ implies, by dominated convergence, that
\begin{align*}
\iint_{A_\epsilon\times A_\epsilon}\|z-w\|^{-1/4}\,dz\,dw&\xrightarrow[\epsilon\to0]{}0,
\\
\int_{A_\epsilon}\dist(z,\partial D)^{-1/8}\,dz&\xrightarrow[\epsilon\to0]{}0.
\end{align*}
Thanks to \eqref{eq.MarkovIMF-covariance-bound}, it follows that
\begin{align*}
\<{\Phi,h_\epsilon}\xrightarrow[\epsilon\to0]{L^2}0,\quad\text{hence}\quad \<{\Phi,fg_\epsilon}\xrightarrow[\epsilon\to0]{L^2}\<{\Phi,f}.
\end{align*}
Concerning $S_\eta$, we use the same two convergences highlighted above: \eqref{eq.bound-M_eta} and \eqref{eq.bound-X_eta} imply that
\begin{align*}
S_\eta(h_\eps)\xrightarrow[\epsilon\to0]{\P}0,\quad\text{hence}\quad S_\eta(fg_\eps)\xrightarrow[\epsilon\to0]{L^2}S_\eta(f).
\end{align*}
This concludes the proof of Point \ref{point.eq for any f}.
\end{proof}

\begin{proof}[Proof of Theorem \ref{th.MarkovIMF}]
It is the combination of Point \ref{point.eq for any f} and Point \ref{point.eq is enough} above.
\end{proof}

\subsection{The IMF conditioned on the first layers of the $\cle_3^\nes$.}\label{ss.IMFwrt-1st-layer}
We have described explicitly the conditional expectation $H_{f,\eta}$ of the IMF conditioned on the set of all loops intersecting a given curve. It is a natural question, raised in the Introduction, to do the same conditionally on all the loops of the first layer, or more generally the first $k$ layers, something a priori difficult but that our techniques allow us to handle. We will only treat the case of the first layer, since everything is analogous in the case $k\ge2$.
\par First, for $\eta$ a smooth grid and $f$ a test function, we need to define a functional $H^1_{f,\eta}$ which will be the candidate for the conditional expectation of $\<{\Phi,f}$ w.r.t. the set of loops $\Gamma^1(\eta)$ which are the loops of the first layer of $\Gamma$ which intersect the smooth grid $\eta$. Let us start by explaining what the exploration procedure attached to it is.

\subsubsection*{The exploration procedure restricted to the first layer.}
Let us use the notations of Subsection \ref{ss.ExplorationProcedure}. We will modify the exploration procedure described there to explore in a Markovian way the loop set $\Gamma^{1,a}(\eta)$
of the loops of the first layer of $\Gamma^a$ that touch the discretization $\eta^a$ of the smooth grid $\eta$. As in this previous exploration procedure, we restrict first to the case where $\eta$ is a smooth curve.

\begin{figure}[!htp]
\begin{center}
\includegraphics[width=\textwidth]{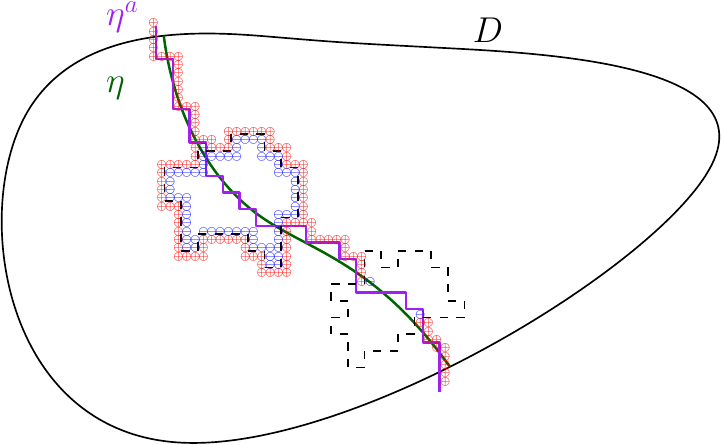}
\end{center}
\caption{The exploration procedure to find the first layer $\Gamma^{1,a}(\eta)$. Compare with Figure \ref{f.exploration-nested} which explored the entire $\Gamma^a(\eta)$. }
\label{f.exploration}
\end{figure}

Recall that $D^a$ is equipped with $+$ boundary conditions, i.e. all vertices in $a \T \setminus D^a$ are set to plus. In particular, by condition $(ii)$ of the discretization (the ends of $\eta$ are outside $D^a$), the beginning and end of the path $\eta^a$ are assigned a plus spin. Write the self-avoiding path $\eta^a$ as $(x_0,\ldots, x_N)$ where $\sigma^a(x_0)=\sigma^a(x_N)=+$ (if $\sigma^a$ is sampled according to the critical Ising model with $+$ boundary conditions in $D^a$). Explore the spins of $\sigma^a$ along the path $\eta^a$, i.e. $\sigma^a(x_0), \sigma^a(x_1),\ldots$. The first time this exploration encounters a $-$ spin at $x_{i+1}$, explore the entire loop in $\Gamma^a $ which goes through the dual edge between $x_i$ and $x_{i+1}$. This loop may intersect $\eta^a$ at many places. Reiterate the exploration starting from all points  $x_j$ explored so far in $\eta^a$ which carry a $+$ spin (in other words, this exploration does not explore inside the loop explored). See Figure \ref{f.exploration}.

\subsubsection*{The conditional expectation w.r.t. the set $\Gamma^1(\eta)$.}
Now that we have defined an exploration procedure, one can redo everything that has been done in Section \ref{se.IMFwrtCurve} for the exploration of the loops of all layers intersecting $\eta$: one gets a similar decomposition for $H^{1,a}_{f,\eta}:=\Eb{\<{\Phi^a,f}\md\Gamma^{1,a}(\eta)}$ as in \eqref{eq.H f eta discret}; in the same way, one introduces $H^1_{f,\eta}$ as an integral that converges by Proposition \ref{pr.dim-Expl(eta)-2}, and one shows that it is indeed the joint limit in law of $H^{1,a}_{f,\eta}$. Finally, one has the following result:

\begin{theorem}\label{th.cond.exp.wrt.Gamma1(eta)}
As in Theorem \ref{th. cond. expect. of magn wrt Gamma(eta)}, the pair $(\Gamma^a,H^{1,a}_{f,\eta})$ jointly converges in law as $a\to0$ to $(\Gamma,H^1_{f,\eta})$. In particular, if $(\Gamma,\Phi)$ is the limit of the coupling $(\Gamma^a,\Phi^a)$, it must satisfy
    \begin{equation}\label{eq. cond. expect. of magn wrt Gamma^1(eta)}
     \Eb{\<{\Phi,f} \md \Gamma^1(\eta)} = H^1_{f,\eta}\qquad\text{almost surely.}
\end{equation}
\end{theorem}

\begin{proof}
The proof is, mutatis mutandis, the same as the one for Theorem \ref{th. cond. expect. of magn wrt Gamma(eta)}. (Actually, everything is easier in our case since there are fewer loops.)
\end{proof}

\subsubsection*{Conditional expectation w.r.t. all the loops in the first layer.} Now, it is a mere formality to construct in an explicit way $H^1_{f,\eta}$ out of $\Gamma^1\sim\cle_3^1$, the conditional expectation of the IMF w.r.t. the first layer of the loops.

\begin{theorem}\label{th.IMFwrtFirst-layer}
Let $D$ be a Jordan domain, and let $(\Gamma,\Phi)$ be the coupling in Theorem \ref{th.main}. Let $f\in H_0^2(D)$. For $k>0$, let $\eta(k)=2^{-k}\Z^2\cap D$. Then the random variables $H^1_{f,\eta(k)}$ converge almost surely as $k\to\infty$ to $\Eb{\<{\Phi,f}\md\Gamma^1}$.
\end{theorem}

\begin{proof}
Indeed, if $H_{f}^1:=\Eb{\<{\Phi,f}\md\Gamma^1}$, then $\Eb{H_f^1\md\Gamma^1(\eta(k))}=H_{f,\eta(k)}^1$, and since \begin{align*}
\sigma\big(\Gamma^1\big)=\bigvee_{k>0}\sigma\big(\Gamma^1(\eta(k))\big),
\end{align*}
by the martingale convergence theorem we have the almost sure (and in $L^1$) convergence of the random variables $H^1_{f,\eta(k)}$ to $H^1_f$.
\end{proof}

\section{Near-critical limit(s) for nested--$\cle_3$}\label{s.NearC}

In this Section, we briefly discuss the consequences of our work for the near-critical scaling limits of the nested--$\cle_3$ induced by a near-critical Ising model on $a\Z^2$ under the following two natural massive limits:
\bnum
\item The near-critical limit in the magnetization direction, which we shall abbreviate by the {\em $h$-near-critical limit.}

It corresponds to scaling the strength of the magnetic field suitably as the mesh $a\to 0$ as follows: for any fixed $h\in \R$, set
$$
h_a:= h a^{15/8}\,.
$$
This limit is discussed for example in \cite{camia2016planar,papon2024interfacescalinglimitcritical}. At least when considered in a bounded window, these $h$-near-critical limits always lead to absolutely continuous limits with respect to the critical one. See subsection \ref{ss.NCh}.
\item The near-critical limit in the temperature direction, which we shall abbreviate by the {\em $\beta$-near-critical limit.} 

Here, instead of acting on the strength of the magnetic field, one tunes the inverse temperature suitably around the inverse critical temperature $\beta_c$ as the mesh $a\to 0$ in the following way: 
\begin{align}\label{e.lambda}
\beta_a:=\beta_c +  \lambda a
\end{align}
for some fixed $\lambda\in \R$. 
\enum

The second limit is more difficult to analyse and has been the focus of the recent works \cite{duminil2022planar,garban2025energy,park2025near}. It is shown in \cite{garban2025energy} that under the topology of FK-Ising percolation (i.e. nested--$\cle_{16/3}$), such a $\beta$-near-critical limit is \textbf{singular} with respect to the critical limit. This does not necessarily imply that the $\beta$-near-critical nested--$\cle_3$ is also singular. Yet we conjecture that this is the case:

\begin{OP}\label{OP.2}
As long as $\lambda\neq 0$ in~\eqref{e.lambda}, the $\beta$-near-critical nested--$\cle_3$ (as $a\to 0$ and under the loop topology defined in Section \ref{ss.CLEs}) is singular w.r.t. the critical nested--$\cle_3$. 
\end{OP}

We discuss in Subsection \ref{ss.NCb} below why we believe this conjecture should hold by combining our present main result with the recent work \cite{park2025near}. Before giving more details about the $h$ and $\beta$-near-critical limits, let us conclude this short discussion with the following puzzling question.
\begin{question}
For the $\beta$-near-critical limit in a finite domain $\Omega$, in the case of $+$ boundary conditions, is the near-critical limit for the first layer of the $\cle_3$, i.e. $\Gamma^1$ in the notations of this text, singular or absolutely continuous w.r.t. a critical $\cle_3^1$? (We do not even know what should be expected here)?
\end{question}

In the case of Dobrushin $\pm 1$ boundary conditions, it has been proved in a very recent work \cite{feng2026nearcriticalisinginterfacesthermal} that the near-critical limit of the $\pm1$ interface is absolutely continuous w.r.t. $\sle_3$.

\subsection{Absolute continuity of the $h$-near-critical limit of the nested--$\cle_3$.}\label{ss.NCh}

In the $h$-direction, the mere existence of the $h$-magnetization $m_h=\langle\Phi,h\rangle$ itself (with the fact that it has exponential moments, a result which goes back to \cite{Camia_2015}) allows one to show that the near-critical scaling limit of the Ising interfaces in the $h$ direction (resp. any near-critical scaling limit of the Ising model - the magnetization field itself was considered in \cite{camia2016planar}) is absolutely continuous w.r.t. the $\cle_3$ (resp. the corresponding critical scaling limit).

\par When the $h$-magnetization $m_h$ is furthermore measurable from the scaling limit $\Omega$, one has an explicit formula for the Radon-Nikodym derivative at $\omega$ (a realization of $\Omega$): \begin{equation}
\frac{d\Omega^h}{d\Omega}(\omega)\propto e^{m_h(\omega)}
\end{equation} where $m_h(\omega)$ is a (potentially complicated) measurable function of $\omega$.

\par See Lemma \ref{l.RN-derivative-L1} for a proof of this statement. See also the recent work \cite{papon2024interfacescalinglimitcritical} for the case of an interface $\sle_3$ and its $h$-massive analog which is also absolutely continuous w.r.t. $\sle_3$.


\subsection{Singularity(?) of the $\beta$-near-critical limit of the nested--$\cle_3$.}\label{ss.NCb}

We wish to briefly explain here why we believe Conjecture \ref{OP.2} should hold. Our reasoning can be decomposed in the following steps.
\bnum
\item First, the breakthrough work \cite{park2025near} tells us what the $\beta$-near-critical magnetization field should be. At critical inverse temperature $\beta_c$, it is known that the continuum magnetization field can be identified (at least in terms of correlation functions and then as a random field) as a suitable (Wick)-cosine of a Gaussian Free Field.
The case for the critical point was already done in \cite{JSW} based on the works of \cite{chelkak2015conformal}.

In the work \cite{park2025near}, they introduce the $k$-point functions of a field which should be thought of as the suitable (Wick)-cosine -- not of a GFF -- but instead of a Sine-Gordon field precisely at the Free-Fermion point (see \cite{bauerschmidt2023coleman,park2025near, bauerschmidt2025twisted}).

\item Let us assume that the near-critical magnetization field at finite small mesh $a$ indeed converges to this limiting field theory from \cite{park2025near} as $a\to 0$, which is already known in the case $\lambda>0$ \cite{park2019massivescalinglimitising}.

\item Let us furthermore assume that Conjecture \ref{OP.2} is wrong, i.e. that the $\beta$-near-critical $\cle_3$ is not singular w.r.t. the critical nested--$\cle_3$. That is, they can be coupled such that there are equal with positive probability.

\item We claim that our reconstruction operator -- from Theorem \ref{th.main} -- which outputs a field $\Psi$ given a nested--$\cle$ would still provide the near-critical magnetization field in the limit~\eqref{e.lambda}. This point is not obvious and would rely on the fact that in small windows (given by a small $\rho$ grid as in the proof of Theorem \ref{th.main}), the near-critical one-point magnetization function gets quantitatively very close to the critical one. (For example using the control and scaling relations from \cite{duminil2022planar}). 

\item At this stage, given the non-singularity of the near-critical nested--$\cle_3$ and our above extension of Theorem \ref{th.main}, we thus obtain a near-critical field $\Phi^\lambda:=\Phi^\lambda(\cle_3^{\nes,\lambda})$ where $\lambda$ is the near-critical parameter from~\eqref{e.lambda}. Furthermore, by the first item above, $\Phi^\lambda$ should be the Wick-cosine of a Sine-Gordon field.

\item Finally, to show a contradiction, we would need to argue that the Wick cosinus of a Sine-Gordon field is singular with respect to the Wick cosinus of the GFF. This singularity is known at the level of the fields: see \cite{gubinelli2024fbsde} (and \cite{garban2025energy} in the case of a hierarchical Sine-Gordon model w.r.t. a hierarchical GFF). Sadly, the field cannot be recover from the Wick cosinus, however the work \cite{aru2021reconstructing} managed to recover the underlying GFF from its Wick imaginary exponential. (The discrete situation of recovering a lattice GFF from its imaginary exponential was analysed in \cite{garban2023statistical}.) We shall therefore make our last key assumption in this analysis, namely that it is possible to mix the proof of measurability of \cite{aru2021reconstructing} together with the proof of singulary between the Sine-Gordon field and the GFF to conclude. Notice that the measurability result should be extended as well to the Sine-Gordon field, which also seems far from obvious especially at the Free Fermion point.

\item Assuming all assumptions above, the desired contradiction is reached. As such this would show that the near-critical $\cle_3$ should be singular w.r.t. the critical one (and the same for the (near-)critical magnetization field).

\enum

Of course, it seems quite non-trivial to make all these points rigorous, but we feel that it gives a rather robust reason for Conjecture \ref{OP.2}.

\section{Insights and limits of an approach via CLE percolation}
\label{se. limits of literature}

The purpose of this section is to discuss why another seemingly direct strategy to prove Theorem \ref{th.main} would be non-trivial to implement. It is based on the theory of CLE percolation as developed in \cite{miller2017cle}. In short, this other natural strategy goes as follows: 1) sample $\Gamma\sim \cle_3^\nes$ (as we do); 2) sample the continuum FK clusters given $\Gamma$ using, say, \cite{miller2017cle}; and 3) use the approach in \cite{Camia_2015}, i.e. the colored continuum FK clusters to reconstruct the magnetization field $\Phi$. This strategy sounds very legitimate and would provide (given these references) a much shorter proof of Theorem \ref{th.main}! Yet, as we explain below, in order to close the argument, it misses one key and rather subtle ingredient summarised in the Open Problem \ref{op la loi du cle(16/3) conditionnellement au cle(3) vient du discret} below. 

Even though we could not make work this direct strategy a priori, we conclude this section by pointing out that by using Theorem \ref{th.main} one can show that this natural strategy indeed leads to the correct reconstruction of $\Phi$ given $\Gamma$. This is the content of Corollary \ref{co. the road through cle percolation gives a measurable cle_3}. 

First, recall that one of the two proofs in \cite{Camia_2015} of the convergence of the Ising magnetization field relied on the assumption of a (geometric) scaling limit of the FK--Ising percolation which appeared in the work \cite{kemp/smirnov_ii}
\footnote{
To be more precise, the convergence in \cite{kemp/smirnov_ii} is stated for the loop topology rather than the quad topology. The convergence in fact  holds for both topologies. This was rigorously established in the case of critical percolation in the work \cite[Theorem 6.10]{holden2023convergence}. See also \cite[Section 2.3]{garban2013pivotal}.   Now, thanks to the known exponents for critical FK-Ising, together with the mixing property of events using RSW \cite{duminil2011connection}, the proof extends to the FK-Ising case as well.}.


\par The proof was  explicit as it constructed directly out of a quad configuration (decorated by signs) the magnetization field and actually showed a bit more than the mere convergence of the magnetization field: namely, the measurability of the magnetization field out of the scaling limit of the FK--Ising percolation. Below is a precise statement (we refer to section 2.2. of \cite{Camia_2015} for the definitions of the quad configuration, and the decorated quad configuration. We adopted similar notations).

\begin{theorem}[Theorem 2.5 in \cite{Camia_2015}]
\label{th. mesurabilité du champ de magnétisation par rapport au CLE 16/3 coloriée}
    Let $(\underline{\omega}^a,\Phi^a)$ be the coupling between $\underline{\omega}^a=(\omega^a_+,\omega^a_-)$, the decorated quad configuration associated to the Edwards-{Sokal} coupling and $\Phi^a$, the magnetization field. It converges in law to $(\omega_+,\omega_-,\Phi)$, a coupling between two quad configurations $(\omega_+,\omega_-)$ and a generalized function $\Phi$ for which $\Phi=\Phi(\omega_+,\omega_-)$ is $(\omega_+,\omega_-)$-measurable.
\end{theorem}

\medskip As such, it was already known that the IMF was measurable as a function of additional randomness. If one could show that this additional randomness is superfluous (as one naturally expects from the discrete case), one could prove Theorem \ref{th.main}. To see this, it is natural to study the continuous coupling $(\underline{\omega},\Gamma)$ arising from the discrete Edwards-Sokal coupling (where $\underline{\omega}$ is the decorated \emph{loop} configuration of an FK--Ising percolation, where each cluster is given an iid sign, and $\Gamma$ is the loop configuration of the spins of an Ising model). \cite{kohler2024fuzzy}, relying on \cite{miller2017cle}, proved that in the continuum one can reconstruct the loops of $\Gamma$ by merging the clusters of $\underline \omega$:

\begin{theorem}[Section 4 in \cite{kohler2024fuzzy}]\label{th. convergence ising géométrique + fk perco lehmkueller}
Let $(\underline{\omega}^a,\Gamma^a)$ be the coupling between $\underline{\omega}^a=(\omega^a_+,\omega^a_-)$, the decorated quad configuration associated to the Edwards-Sokal coupling and $\Gamma^a$, the set of interfaces of the Ising configuration. It converges in law to $(\omega_+,\omega_-,\Gamma)$, the coupling described in \cite{miller2017cle} between a colored CLE(16/3) and a CLE(3). In particular, $\Gamma$ is $\underline\omega$-measurable in the continuum.
\end{theorem}

\begin{remark}
    Actually, the proof presented in \cite{kohler2024fuzzy} shows a deeper statement: \emph{the law of a $\cle_3$ conditionally on a $\cle_{16/3}$ is the scaling limit of the law of the discrete $\cle_3$ conditionally on a discrete $\cle_{16/3}$ in the Edwards-Sokal coupling.}
\end{remark}

\medskip The combination of Theorem \ref{th. mesurabilité du champ de magnétisation par rapport au CLE 16/3 coloriée} and Theorem \ref{th. convergence ising géométrique + fk perco lehmkueller} implies formally the uniqueness of the coupling $(\Gamma,\Phi)$, see Lemma \ref{l. if two limits are measurable with respect to a third, then they are uniquely coupled}.

\medskip However, what precedes turns out to be not sufficient (as one may first think) to deduce the measurability of the magnetization field in terms of the set of interfaces: indeed, if we denote by $X^a,Y^a,Z^a$ the (discrete) spin-loop configuration, decorated FK--Ising loop configuration, magnetization, and $X,Y,Z$ the limit when $a\to0$, we aim to show that $Z$ is $X$-measurable using \begin{enumerate}
    \item $(X,Y,Z)$ is the limit (in law) of the triples $(X^a,Y^a,Z^a)$. Furthermore, ${\mathcal L}(X^a \,|\,Y^a)$ converges to ${\mathcal L}(X\,\md \, Y)$ (still true if taken inside the coupling).
    \item $Z^a$ is $Y^a$- and $X^a$-measurable for any $a>0$
    \item $Z$ is $Y$-measurable
\end{enumerate}
These three points are not enough to conclude. The following fourth assumption would however suffice:
\begin{enumerate}
    \setcounter{enumi}{3}
    \item the random measure ${\mathcal L}(Y^a \,\md \,X^a)$ converges to ${\mathcal L}(Y\,\md \, X)$
\end{enumerate}
Indeed, this fourth point implies that if $(X^a,Y^a,Y'^a)$ is such that, conditionally w.r.t. $X^a$, $Y^a,Y'^a$ are iid of law ${\mathcal L}(Y^a \,\md \,X^a)$, then the coupling converges to the analogous coupling $(X,Y,Y')$. Therefore, the extended coupling $(X^a,Y^a,Z^a,Y'^a,Z'^a)$ (for which $Z^a=Z'^a$ thanks to point 2 above) itself converges to $(X,Y,Z,Y',Z')$ with $Z=Z'$. Since $Z=Z(Y)$ (and $Z'=Z(Y')$), the random variables $Z$ and $Z'$ are conditionally independent w.r.t. $X$: their equality implies that $Z$ is actually a measurable function of $X$ alone.

A natural strategy to actually prove the measurability of the IMF {would} therefore be to prove a strengthening of Theorem \ref{th. convergence ising géométrique + fk perco lehmkueller} (or perhaps better said, the "other direction" of the result described in the remark following that theorem): 

\begin{OP}\label{op la loi du cle(16/3) conditionnellement au cle(3) vient du discret}
    Show that the law of $\omega$ conditionally on $\Gamma$ is the limit in law of the conditional law of the $\omega^a$ w.r.t. $\Gamma^a$, i.e. a "continuous independent percolation" inside $\Gamma$.
\end{OP}

\par Beware, an affirmative answer to this question would be a \emph{strong} improvement of Theorem \ref{th. convergence ising géométrique + fk perco lehmkueller}! The paper \cite{miller2017cle} gives {insights} about the coupling $(\omega,\Gamma)$ in the continuum. However, it does not describe at all the law ${\mathcal L}(\omega \md \Gamma)$: for example, they define a \emph{CPI} (see Section 2.3 in \cite{miller2017cle} for a definition of the CPI) to be the marginal of a coupling and ask that the \emph{coupling} is conformally invariant, and show that the CPI's \emph{annealed} law is {a kind} of SLE (see Proposition 4.1 in their paper). But they do not show (nor claim) that the conditional law possesses such a property. 
\par See also section 1.3.1 of \cite{holden2023convergence} for a similar discussion of the previous difficulties. However, the techniques that they develop to prove their (deep and very difficult) main theorem (somehow an analogue of the conjecture above) cannot be applied in our context. 
\par To prove this conjecture, a major (and unavoidable) step seems to be to properly define the random law ${\calL}(\omega \md \Gamma)$. This is by no means clear, at least to us. This is why we followed in Sections \ref{se.IMFwrtCurve} and \ref{se.Proof-Main-thm} a different route. 

{That said, \emph{once we have proved the measurability (Theorem \ref{th.main}),} interestingly we can justify a posteriori that this other route --- i.e. cutting $\Gamma\sim\cle_3^\nes$ into $\omega\sim\cle_{16/3}^\nes$ and constructing the IMF out of it --- indeed gives the desired random field. In particular, it shows that the field produced this way does not depend on the additional randomness of $\omega$.} Let us state it here as a corollary:

\begin{corollary}\label{co. the road through cle percolation gives a measurable cle_3}
    Let $\Gamma\sim\cle_3$. Conditionally on $\Gamma$, let $\underline\omega$ follow the (abstract) conditional law of the coupling $(\Gamma,\underline\omega)$ described in \cite{miller2017cle}. Finally, let $\Phi$ be the magnetization field constructed out of $\underline\omega$, as described in \cite{Camia_2015}. In the coupling $(\Gamma,\Phi)$, $\Phi$ is $\Gamma$-measurable; thus it does not depend on the additional randomness of $\underline\omega$. 
\end{corollary}

%
%
%
%

\appendix

\section{Topologies and properties of the CLE.}
\label{ap. Topologies and properties of (nested) CLE}

\subsection{The Hausdorff topology on the closed subsets of $\overline D$.}
\label{ss.Hausdorff Topology}

Let $K(D)$ be the set of compact subsets of $\overline{D}$. We define on $K(D)$ the \emph{Hausdorff distance} $\d_\haus$ to be \begin{equation}
\d_\haus(F_1,F_2) = \inf\{\eps>0 \ :\ F_1\subset(F_2)_\eps\} \vee \inf\{\eps>0\ :\ F_2\subset(F_1)_\eps\} 
\end{equation} where $F_\eps=\{x\in \overline{D}, \dist(x,F)<\eps\}$ for $F\in K(D)$. For the topology induced by $\d_\haus$, it is well known that $K(D)$ is a compact space (therefore Polish).

\subsection{Convergence in the Caratheodory sense.}
\label{ss.Caratheodory Topology}

There is another notion of convergence for open subsets of $D$, which is a bit weaker than the more classical Hausdorff convergence for their closure. Actually, as we shall see the Carathéodory topology concerns convergence of \emph{pointed domains}, that is connected open subsets of $\bar \C$ with a distinguished point.
\par If $(D_n,z_n)$ is a sequence of pointed domains, and $(D,z)$ is another pointed domain, we say that $(D_n,z_n)\to(D,z)$ in the Carathéodory sense if the following conditions are satisfied:
\bnum[i)]
\item $z_n\to z$ in the topology of $\overline\C$.
\item For any compact $K\subset D$, $K$ is included in all but finitely many $D_n$.
\item For any connected open set $U$ containing $z$, if $U\subset D_n$ for infinitely many $n$ then $U\subset D$.
\enum 

There is an equivalent definition of this notion of convergence that we will need in Lemma \ref{l.OpenDomainsConvergeCaratheodory} (see, for example, \cite{comerford2013caratheodory}, Theorem 1.1).

\begin{theorem}\label{th.CaratheodoryConv}
Let $(D_n,z_n)$ be a sequence of pointed domains, and $(D,z)$ be another pointed domain. The following are equivalent: \bi 
\item $(D_n,z_n)\to(D,z)$ in the Carathéodory sense.
\item $z_n\to z$ in the topology of $\overline\C$; furthermore, for any sublimit $F$ of $F_n=\overline\C\setminus D_n$ in the Hausdorff topology (with respect to any metric of $\overline\C$), $D$ is the connected component of $z$ in $\overline\C\setminus F$.
\ei 
\end{theorem}

One should not get worried by the use of the topology of $\overline\C$: indeed, we will only apply what precedes in the context of bounded domains where we can think of the Euclidean topology on $\C$, and the classical Hausdorff topology as well.

\subsection{Some properties of loop sets}\label{ss.PropertiesLoopSets}

In this Subsection, we will define some observables out of a loop set $\Gamma\in\caX_D$, and we will prove properties about them. Note that our statements are deterministic. 

\begin{lemma}\label{l. intersection with closed set is measurable}
Let $D$ be a bounded domain and let $\eta\subset\overline D$ be a non-empty compact set. For $\Gamma\in\mathcal X_D$, define $\Gamma(\eta) :=\big\{\ell\in\Gamma \md\im(\ell)\cap\eta\neq\varnothing\big\}$
and $\Expl(\eta):=\Expl(\Gamma,\eta) =\eta\cup\bigcup_{\ell\in\Gamma(\eta)}\im(\ell)$. Then the maps
\begin{align*}
\Gamma\longmapsto\Gamma(\eta)\qquad\text{and}\qquad\Gamma\longmapsto\Expl(\eta)
\end{align*}
are Borel measurable, respectively as maps from $\mathcal X_D$ to $\mathcal X_D$ and from $\mathcal X_D$ to $K(\overline D)$.

Moreover, on the closed subspace $\mathcal X_{D,\eta} :=\big\{\Lambda\in\mathcal X_D \md\im(\ell)\cap\eta\neq\varnothing\text{ for every } \ell\in\Lambda\big\}$,
the map
\begin{align*}
\Lambda\longmapsto \eta\cup\bigcup_{\ell\in\Lambda}\im(\ell)
\end{align*}
is continuous.
\end{lemma}

\begin{proof}
The set
\begin{align*}
A_\eta := \big\{\ell\in\mathcal L_D\md \im(\ell)\cap\eta\neq\varnothing\big\}
\end{align*}
is closed in $\mathcal L_D$, hence Borel, since $\ell\mapsto\dist(\im(\ell),\eta)$ is continuous. Hence $\Gamma(\eta)$ is the restriction of the locally finite loop configuration $\Gamma$ to the Borel set $A_\eta$, and is therefore a measurable function of $\Gamma$. 

It remains to consider
\begin{align*}
U_\eta(\Lambda) := \eta\cup\bigcup_{\ell\in\Lambda}\im(\ell), \qquad
\Lambda\in\mathcal X_{D,\eta}.
\end{align*}
This set is compact: if points belonging to distinct loops converge,
then the diameters of these loops tend to zero by local finiteness; since
every loop intersects $\eta$, the limiting point belongs to $\eta$.

Finally, the definition of $\d_{\mathcal X}$ immediately gives
\begin{align*}
\d_{\mathrm{Haus}}
\bigl(U_\eta(\Lambda),U_\eta(\Lambda')\bigr)
\leq
\d_{\mathcal X}(\Lambda,\Lambda').
\end{align*}
Indeed, matched loops have close images, while every unmatched loop has
small diameter and intersects the common set $\eta$. Thus $U_\eta$
is continuous, and
\begin{align*}
\Expl(\eta)=U_\eta\bigl(\Gamma(\eta)\bigr)
\end{align*}
is measurable.
\end{proof}

\begin{lemma}\label{l.OpenDomainsConvergeCaratheodory}
With the same notation as above, and $D^\eta:= D\setminus\Expl(\eta)$, assume that the sequence $\big(\Gamma^k(\eta)\big)_{k\ge0}$ converges to $\Gamma(\eta)$ in $\caX_D$. Let $z\not\in\Expl(\eta)$. For any $k$ large enough, $z\not\in\Expl^k(\eta)$. Furthermore, if we denote by $D_z$ (resp. $D_z^k$) the connected component of $z$ in $D^\eta$ (resp. $D^{k,\eta}$), then $(D_z^k,z)$ converges to $(D_z,z)$ in the Carathéodory sense as $k$ goes to infinity.
\end{lemma}

As we have seen in the previous lemma, the convergence $\Gamma^k(\eta)\to\Gamma(\eta)$ implies the Hausdorff convergence $\Expl^k(\eta)\to\Expl(\eta)$. However, if $\dist(z,\Expl(\eta))>\eps$, then $z$ cannot be in $\Expl^k(\eta)$ as soon as $\d_\haus\big(\Expl(\eta),\Expl^k(\eta)\big)<\eps$. Therefore, $z$ is eventually not in $\Expl^k(\eta)$. Let us prove the Carathéodory convergence.

\begin{proof}
We could prove this fact with the usual definition of Carathéodory convergence; however, it would be a bit tedious and technical. Instead, let us rely on the equivalent criteria for Carathéodory convergence given in Theorem \ref{th.CaratheodoryConv}.
\par So, according to this result, we only need to show that $D_z$ is indeed the connected component $U$ of $z$ in $\overline\C\setminus F$ where $F$ is any Hausdorff sublimit of the closed subsets $F^k=\overline\C\setminus D_z^k$ (the convergence of the distinguished point being trivial). Since $F^k$ contains $\Expl^k(\eta)$ by definition, it follows that any sublimit $F$ must contain the Hausdorff limit $\Expl^k(\eta)\to\Expl(\eta)$. In particular, $U\subset D_z$.
\par The reciprocal fact that $U$ contains $D_z$ can be seen as follows: if $w\in D_z$, there is a continuous path $\gamma$ contained in $D_z$ from $z$ to $w$ by path-connectivity of $D_z$. Furthermore, $\dist(\gamma,\partial D_z)>0$ so there is a neighbourhood of $\im\gamma$ such that for large enough $k$, this neighbourhood doesn't intersect the closed set $F^k$; in particular, a neighbourhood of $\im\gamma$ is contained in $\overline\C\setminus F$ and connects $z,w$: therefore $w\in U$, and since $w$ was arbitrary, $D_z\subset U$.
\par We have shown that $D_z=U$, whence by Theorem \ref{th.CaratheodoryConv} we can conclude: $(D^k_z,z)\to(D_z,z)$ in the Carathéodory sense.
\end{proof}

\begin{remark}
Proving the convergence in the Hausdorff sense would be more technical and would need extra assumptions on $\Gamma$ (that are at any rate almost surely satisfied for $\Gamma\sim\cle_3^\nes$). Since Theorem \ref{th.1pointFunction} on the convergence of the spin correlation only requires Carathéodory convergence, we chose to only prove this.
\end{remark}

\begin{lemma}\label{l.NumberOfLoopsSurroundingForall z at once}
For $\Gamma\in\caX_D$ and $z\in D$, denote by $\nes_\Gamma(z)\in\N\cup\{\infty\}$ the cardinality of $\Gamma(z)$, the collection of loops in $\Gamma$ surrounding $z$ (the loops which pass through $z$ are in $\Gamma(z)$).
\par Let $(\Gamma^k)_{k\ge0}$ be a sequence converging to $\Gamma$ in $\caX_D$, and let $z$ be a point which does not lie on any loop of $\Gamma$. Under the assumption that $\inf_{k}\inf_{\ell\in\Gamma^k(z)}\diam(\ell)>0$, the sequence $\nes_{\Gamma^k}(z)$ converges to $\nes_{\Gamma}(z)$ as $k$ goes to infinity.
\end{lemma}

With the notations of Lemma \ref{l. intersection with closed set is measurable}, we will use this property in the case that $\Gamma^a(\eta)\xrightarrow[a\to0]{}\Gamma(\eta)$, for which the conditions on $z$ hold when $z\not\in \Expl(\eta)$: indeed, $z$ doesn't lie on a loop of $\Gamma(\eta)$, and any loop surrounding $z$ in some $\Gamma^a(\eta)$ must be of diameter at least $\dist(z,\eta)>0$. 

\begin{proof}
Let $\ell_1,\ldots,\ell_\nes$ be the loops surrounding $z$ in $\Gamma$ (with $\nes=\nes_\Gamma(z)$ possibly infinite), and for each finite $i$, let $\eps_i=\dist(z,\ell_i)$, which is nonzero since $z$ is not on $\ell_i$. Let $N\le\nes$ be an integer, and let $\eps=\inf_{i\le N}\eps_i$. Of course, $\diam(\ell_i)>\eps_i$ for any $i$, so that if $\d_{\caX}(\Gamma^k,\Gamma)<\eps$, by definition of $\d_\caX$ there must be $N$ loops $\ell_{1}^k,\ldots,\ell^k_{N}\in\Gamma^k$ at loop distance at most $\eps$ from $\ell_1,\ldots,\ell_N$ respectively. Each of them has to surround $z$, thus $\nes_{\Gamma^k}(z)\ge N$. Letting $N\to\nes\wedge\infty$ shows that $\liminf_k\nes_{\Gamma^k}(z)\ge\nes$.
\par It remains to show that $\limsup_k\nes_{\Gamma^k}(z)\le\nes$ in the case $\nes<+\infty$. Let $\eps$ be one-half of the minimum of $\inf_{k}\inf_{\ell\in\Gamma^k(z)}\diam(\ell)$ and $\inf_{\ell\in\Gamma(z)}\diam(\ell)$, which is positive by hypothesis. Let $\Gamma'$ be the finite subcollection of $\Gamma$ consisting of those loops of diameter at least $\eps$, and let $\delta=\inf_{\ell\in\Gamma'}\dist(z,\ell)\wedge\eps$. For any $k$ large enough so that $\d_\caX(\Gamma^k,\Gamma)<\delta$, we will see that $\nes_{\Gamma^k}(z)\le\nes$, which will conclude the proof. To see that, notice that, if $\ell^k\in\Gamma^k(z)$ \bi 
\item any loop in $\Gamma'\setminus\Gamma(z)$ is at loop distance at least $\inf_{\ell\in\Gamma'}\dist(z,\ell)>\delta$ from $\ell^k$ (since they do not enclose $z$, and $\ell^k$ does).
\item any loop in $\Gamma\setminus\Gamma'$ is of diameter less than $\eps<\diam(\ell^k)/2$, so at loop distance at least $\diam(\ell^k)-\eps>\delta$ from $\ell^k$.
\item $\diam(\ell^k)>\delta$.
\ei
Whence, by definition of $\d_\caX$, the bijection $\pi^k$ between finite sub-collections of $\Gamma^k,\Gamma$ must send $\Gamma^k(z)$ into $\Gamma(z)$, which shows the desired $\nes_{\Gamma^k}(z)\le\nes$.
\end{proof}

\begin{lemma}\label{l.Domains-to-be-explored-are-Jordan}
If $D$ is a Jordan domain, $\eta$ a smooth grid in $D$ and $\Gamma\in\caX_D$ a locally finite collection made of simple loops such that any two loops do not hit each other and that there are no loops tangent to $\eta$, then every connected component of $D\setminus\Expl(\Gamma,\eta)$ is a Jordan domain.
\end{lemma}

\begin{proof}We work in the topology of the sphere $\overline\C$.
Set $K:=\big(\overline{\C}\setminus D\big)\cup\Expl(\eta)$.
We prove that $K$ is a connected compact space which is locally connected and without cut points. (Recall that $x$ is a \emph{cut point} of $M$ if $M\setminus{x}$ is disconnected). Theorem 4 in \cite{kuratowski1968topology2} (see Chapter X, p.~512) then implies that the connected components of $\overline\C\setminus K=D\setminus\Expl(\Gamma,\eta)$ are indeed Jordan domains.

Compactness follows from the fact that $\Expl(\eta)$ is closed in $\overline D$ (see Lemma \ref{l. intersection with closed set is measurable}) and connectedness from the fact that $\eta$, and each loop $\ell\in\Gamma$ are obviously connected.

We next prove that $K$ is locally connected. We use the standard characterization by property~$S$: a compact metric space is locally connected if, for every $\eps>0$, it admits a finite cover by connected sets of diameter less than $\eps$.

Fix $\eps>0$. There are only finitely many loops of diameter at least $\eps/3$, and each of them is locally connected, hence covered by a finite number of connected sets of diameter at most $\eps/3$. Furthermore, one can cover $\eta$ by a finite number of connected sets of diameter at most $\eps/3$ as well. We therefore have a finite collection $(A_i)_i$ of connected spaces, each of which of diameter at most $\eps/3$, which together cover $\eta\cup\bigcup_{\ell}\im(\ell)$ where the union is over $\ell\in\Gamma(\eta)$ that are of diameter at least $\eps/3$. Now let $B_i$ be the union of $A_i$ with all the loops in $\Gamma(\eta)$ of diameter less than $\eps/3$ that it intersects: $\diam(B_i)<\eps$ and $B_i$ is still connected. Furthermore, $\bigcup_iB_i=\Expl(\eta)$. Hence, since $\overline\C\setminus D$ is obviously locally connected, $K$ is so as well.

It remains to show that $K$ has no cut point. Fix $x\in K$. If $x\not\in D$, it is clear that $K\setminus x$ is connected. If $x\in\eta$, since all the endpoints of $\eta$ are outside $D$, $\big(\eta\setminus x\big)\cup \big(\overline \C\setminus D\big)$ is connected, thus $K$ as well. Finally, if $x\in\im(\ell)$ for some $\ell\in\Gamma(\eta)$, $\im(\ell)\setminus x$, and therefore $K$, is still connected. Thus $K$ has no cut points. This concludes the proof.
\end{proof}

\subsection{A property specific to a loop set following the law of a nested $\cle_3$.}

Here we prove a lemma about random $\Gamma\in\caX_D$ that follows the law of a $\cle_3^\nes$, which states that almost surely the $\cle_3$ is not tangent to a given smooth grid. This is crucial in our paper.


\begin{lemma}\label{l. curve not tangent cle_3}
Let $\eta$ be a smooth grid in a Jordan domain $D$, and let $\Gamma\sim\cle_3^\nes$. Almost surely, $\eta$ is not tangent to any of the loops in $\Gamma$ in the sense that, for every loop $\ell\in\Gamma(\eta)$, the two connected components of $D\setminus\im(\ell)$ intersect $\eta$.
\end{lemma}

\begin{proof}

There are several possible ways to address this property. One may, for example, rely on the Brownian loop soup construction from \cite{sheffield2012conformal}. Another possible route would be to rely on $\sle_\kappa(\rho)$-type constructions of the $\cle_3^\nes$ but the fact that $\kappa=3$ and $\rho=\kappa-6=-3<-2$ would induce some additional technical difficulties here.

Note that this is the approach used  in
Lemma C.3 in \cite{lopez2026ising} for a similar statement in the case of a $\cle_4^\nes$ ($\kappa=4$).

To prove this property when $\kappa=3$, we will take advantage of the following three facts:
\bnum
\item  $\cle_3^\nes$ is the scaling limit of the exploration of discrete $+/-$ interfaces of the spin Ising model.
\item A straightforward monotonicity argument between a $+$ boundary condition and a ``plane'' boundary condition (as pictured on the right-hand side of Figure \ref{f.StochDom}).
\item The computation of the relevant half-plane critical exponent in $+$ boundary conditions by Hao Wu in \cite{wu2018alternating} and the fact that this exponent $\alpha_1^+=\frac 5 3$ is strictly larger than 1. 
\enum 

Let us assume that $\eta$ is a line $\ell$, the case of a smooth line being similar, and the result for a smooth grid is a consequence of the result for smooth lines. Let $Q_\eps$ be an $\eps$ by $\eps$ square along that line in the bulk of the domain $D$. We will prove below that the probability that a $\cle_3^\nes$ loop of diameter greater than $t>0$ intersects the square $Q_\eps$ but does not cross the line $\ell$ is bounded from above by $O(1) (\eps/t)^{5/3}$. By fixing the magnitude $t$ of the loops we consider, letting $\eps\to0$ and using a covering of $\ell$ in the bulk by at most $\eps^{-1}$ squares, this proves the desired result. 

\begin{figure}[!htp]
\begin{center}
\includegraphics[width=\textwidth]{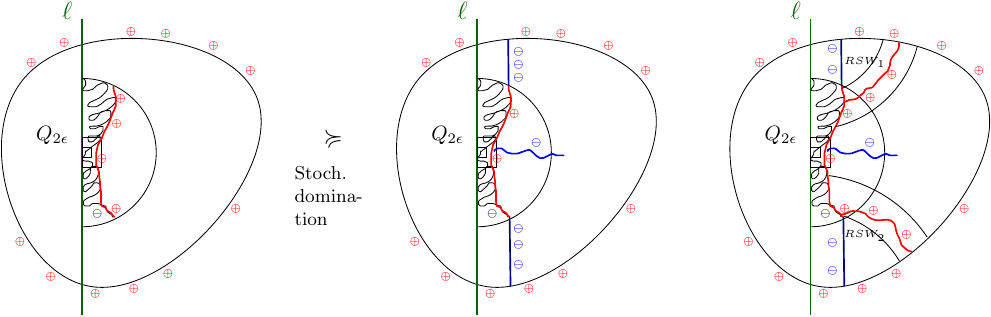}
\end{center}
\caption{Idea of the proof of Lemma \ref{l. curve not tangent cle_3}}\label{f.StochDom}
\end{figure}

\begin{figure}[!htp]
\begin{center}
\includegraphics[width=0.8\textwidth]{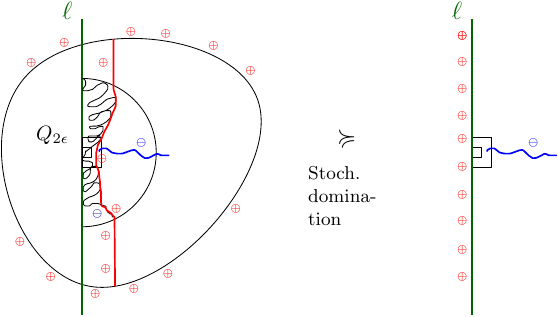}
\end{center}
\caption{}\label{f.StochDom2}
\end{figure}

By the convergence of the discrete $\cle$ exploration process towards its continuum limit (see Subsection \ref{ss.CLE3}), 
the probability we are looking for is less than the limsup as the mesh goes to zero of having a three-arm event $+/-/+$ from a square of side-length $2\eps$ to distance $t/2$. (N.B. We are considering a larger square than $Q_\eps$ to allow some margin with respect to the topology of convergence). 
The exponent driving this event did not --- to our knowledge --- appear in the literature (it should be equal to 2 by applying a celebrated argument by Michael Aizenman which is based on RSW in the setting of critical percolation). We only need the above upper bound which is good enough to conclude since $\frac 5 3 > 1$. 

We, now, argue as follows: in the bulk of the domain $D$, start exploring the spin field from the line $\ell$ to the right of the line $\ell$ and inside the ball centred at $Q_\eps$ and with radius $t/2$ (we also assume without loss of generality that  $t < \dist(Q_\eps, \p D)$). Stop the exploration the first time one sees a line of $+$ crossing the box $Q_{2\eps}$.
See the left picture on  Figure \ref{f.StochDom}.
This exploration process leads us to a complicated mixed boundary conditions. It is easy to see that it stochastically dominates the domain with $-$ arcs as pictured on the RHS of Figure \ref{f.StochDom}. As such the probability of having a minus arm in the explored left region is less than or equal to the one in the right region. Also a $-$ arm to distance $t/2$ is obviously less than having a $-$ arm up to $t/4$ (this small change is important in order to have enough room to apply RSW).

Now we wish to justify that the probability to have a $-$ one-arm up to distance $t/4$ in these mixed $+/-$ boundary conditions is less than a constant times the probability to have such a $-$ arm up to $t/4$ in all $+$ boundary conditions as pictured in the middle of Figure \ref{f.StochDom}.
If we denote by $\xi^{+/-}$ and $\xi^+$ these two boundary conditions and by $A^-(t/4)$ the $-$ one-arm event, we are thus claiming that there exists $C>1$ such that 
\begin{align}\label{e.nice}
\FK{}{\xi^{+/-}}{A^-(t/4)} \leq C \, \FK{}{\xi^{+}}{A^-(t/4)}\,.
\end{align}
To see why this holds, we follow the same argument as in  \cite{duminil2011connection} (or rather the updated version available on H. Duminil-Copin's website). Thanks to RSW for FK-Ising, it is easy to notice that under $\xi^{+/-}$ boundary conditions, there is a positive probability to have a $+$ crossing path in rectangles $RSW_1$ and $RSW_2$ as pictured in the RHS of Figure \ref{f.StochDom}. Let $B^+$ be this RSW event, we thus have 
\begin{align*}\label{}
\FK{}{\xi^{+/-}}{B^+} \geq c >0\,.
\end{align*}
As in  \cite{duminil2011connection}, let $P$ be the natural monotone Markovian coupling of $(\sigma^{\xi^{+/-}}, \sigma^{\xi^+})$ so that $\sigma^{\xi^{+/-}} \leq \sigma^{\xi^+}$ and also so that if one finds a $+$ path for $\sigma^{\xi^{+/-}}$ which realizes $B^+$, this must also be $+$ path for $\sigma^{\xi^{+}}$ and configurations inside do match. Following exactly \cite{duminil2011connection} we thus have 
\begin{align*}\label{}
\FK{}{\xi^{+}}{A^-(t/4)} & = P( \sigma^{\xi^{+}} \in A^-(t/4) )  \\
& \geq P( \{\sigma^{\xi^{+}} \in A^-(t/4) \} \cap  \{ \sigma^{\xi^{+/-}} \in B^+ \}) \\
&   = P( \{\sigma^{\xi^{+/-}} \in A^-(t/4)  \}\cap  \{ \sigma^{\xi^{+/-}} \in B^+ \})  \\
& =  \FK{}{\xi^{+/-}}{B^+  | A^-(t/4)}  \FK{}{\xi^{+/-}}{A^-(t/4)} \\
& \geq \tilde c  \,  \FK{}{\xi^{+/-}}{A^-(t/4)} 
\end{align*}
since by FK-Ising RSW, not only $\FK{}{\xi^{+/-}}{B^+} \geq c >0$ but also $\FK{}{\xi^{+/-}}{B^+  | A^-(t/4)} \geq \tilde c > 0$, this ends the proof of the bound \eqref{e.nice}. (N.B. Note that this proof would not work in the other direction for a generic event $A$ measurable w.r.t $t/4$ ball, namely $\FK{}{\xi^{+/-}}{A} \leq C \, \FK{}{\xi^{+}}{A}$ would work for any such event, but $\FK{}{\xi^{+}}{A} \leq C \, \FK{}{\xi^{+/-}}{A}$ would be much harder to prove as we cannot bring a $-$ path to touch a $+$ path with an FK-Ising RSW. Fortunately, the above mixing argument goes in the right direction for us).

We are thus left with upper bounding $\FK{}{\xi^{+}}{A^-(t/4)}$, i.e. with finding a $-$ cluster inside the unexplored region up to distance $\frac t 4$. Notice that the field of spins in the unexplored region surrounded by plus spins stochastically dominates the exploration of spins in the infinite half-plane with plus boundary conditions. See Figure \ref{f.StochDom2}. As such, the probability of finding a minus path in our unexplored region is smaller than in the infinite half-plane with plus boundary conditions. The latter one is one of the many critical exponents computed by Hao Wu in \cite{wu2018alternating}, namely this is the case $j=1$ of the $\alpha_{2j-1}^+, j\geq 1$ critical exponents with $\oplus\oplus$ boundary conditions (notations from \cite{wu2018alternating}) and is equal to $\frac 5 3 $.  This ends the proof.

\end{proof}

\section{An upper bound on the Green function}
\label{ap.Green-bounds-proof-finiteness}

In this appendix, we give a --- possibly classical --- upper bound on $G_\Omega$, the Green function on $\Omega$ with Dirichlet boundary conditions. We then apply this bound to prove the finiteness of two integrals that are needed to show the tightness of our sequence of random fields $(\Phi^a)$, see Lemmas \ref{l.int-finite-smoothBdry} and \ref{l.int-finite-Green-Dirichlet} below. Furthermore, we prove another integrability result which is used in many occurences in the core of the paper.

\begin{lemma}
Let $\Omega$ be a Jordan domain of $\R^2$. There exists a constant $C>0$ such that \begin{align}\label{eq.bounds-Green}
G_\Omega(x,y)\le C\min\Big(\frac{d(x)}{|x-y|},1\Big)^{1/2}\cdot\min\Big(\frac{d(y)}{|x-y|},1\Big)^{1/2}\cdot\Big(\log(\frac1{|x-y|})+C\Big)
\end{align}
where $d:\Omega\to\mathbb R_{>0}$ is a shorthand for $\dist(\cdot,\p\Omega)$.
\end{lemma}

\begin{proof}
Let us first prove this with the two minima replaced by 1. For $x,y\in\Omega$, $G_\Omega(x,y)\le G_B(x,y)$ where $B$ is the ball centered at $x$ and of radius $\diam(\Omega)$: indeed, $\Omega\subset B$ and the Green function is increasing in the domain. Furthermore, $G_B(x,y)=\log(\diam(\Omega)/|x-y|)$ by conformal invariance and the classical formula $G_\D(0,z)=-(2\pi)^{-1}\log|z|$. Whence the inequality holds with the minima replaced by 1.
\par Now, let $x,y$ be points and $r=|x-y|/3$; define $U_x=\{z\in\Omega\md |x-z|<r \}$ and $U_y=\{z\in\Omega\md |y-z|<r \}$. By harmonicity of the Green function in both coordinates (and the fact that it is well defined on $U_x\times U_y$ since both sets are disjoint) \begin{align*}
G_\Omega(x,y) = \iint_{\p U_x \times\p U_y} G_\Omega(u,v)\nu_{x,\p U_x}(du)\nu_{y,\p U_y}(dv)
\end{align*}
where $\nu_{x,\p U_x}$ is the harmonic measure on $\p U_x$ viewed from $x$ (and analogously for $\nu_{y,\p U_y}$).
\par We know $G_\Omega(u,v)\le 1_{u,v\in\Omega}(\log(1/r)+O(1))$ by what we proved earlier and by the fact that $G_\Omega$ is zero on the boundary $\p \Omega$. (Recall that $\Omega$ being a Jordan domain, it is in particular regular so that $G_\Omega$ is indeed harmonic outside of the diagonal and 0 on the boundary.) 

\par Therefore, \begin{align*}
G_\Omega(x,y)\le (\log(1/r)+O(1))\int_{\p U_x}1_{u\in\Omega}\nu_{x,\p U_x}(du)\int_{\p U_y}1_{v\in\Omega}\nu_{y,\p U_y}(dv),
\end{align*} and since Beurling's estimate states precisely that (see e.g. Theorem V.4.1 in \cite{bass1994probabilistic}) \begin{align*}
\int_{\p U_x}1_{u\in\Omega}\nu_{x,\p U_x}(du) = \P^x(B \text{ goes to distance $r$ before exiting }\Omega) \le 2\sqrt{d(x)/r},
\end{align*}
the bound on $G_\Omega$ follows.
\end{proof}

We will use Whitney's decomposition lemma to prove our integrability results, which we recall below. See the book \cite{grafakos2008classical} (Appendix J.1) for a reference and a short proof of this classical fact.

\begin{lemma}[Whitney's decomposition lemma]\label{l.Whitney}
Let $D$ be any domain in $\R^2$. We can write $D=\bigcup Q_i$ where the $Q_i$ are dyadic closed squares such that \bnum
\item any two of them have disjoint interiors,
\item $\dist(Q_i,\p D)\asymp \ell_i$ uniformly on $i$, where $\ell_i$ is the side-length of the square $Q_i$,
\item any two squares $Q_i,Q_j$ touching each other satisfy $\ell_j/4\le\ell_i\le4\ell_j$. Moreover, any square has at most $12^2$ squares touching it.
\enum
\end{lemma}

Let us show a useful and simple application of this lemma:

\begin{lemma}\label{l.dist^-1/8 int}
Let $D$ be a bounded domain whose boundary is of upper Minkowski dimension $\dimM(\p D)<15/8$. Then the map $z\mapsto\dist(z,\p D)^{-1/8}$ is integrable over $D$, i.e. $\int_D\dist(z,\p D)^{-1/8} dz<\infty$.
\end{lemma}

\begin{proof}
Let $(Q_i)$ be the Whitney decomposition for $D$. One can see that \begin{align*}
\qquad\qquad\qquad\dist(z,\p D)^{-1/8}\asymp \ell_i^{-1/8}\quad\quad\text{ over }Q_i
\end{align*}uniformly over the squares $Q_i$, by $2)$ in Lemma \ref{l.Whitney}. Thus, $\int_{Q_i}\dist(z,\p D)^{-1/8}dz\asymp\ell_i^{2-1/8}$ from which we derive $\int_D\dist(z,\p D)^{-1/8} dz=O\big(\sum_i\ell_i^{2-1/8}\big)$.

By assumption on $\p D$, there is $\alpha<15/8$ such that $\#\{i\md \ell_i=2^{-n}\}=O(2^{n\alpha})$. Therefore the sum $\sum_i\ell_i^{2-1/8}$ can be upper-bounded by a finite constant time $\sum_n2^{n\alpha}2^{-n(2-1/8)}=\sum_n2^{-n(15/8-\alpha)}<\infty$. This concludes the proof: the integral $\int_D\dist(z,\p D)^{-1/8} dz$ is indeed finite.
\end{proof}

\begin{remark}\label{rk.bdry-smooth on a subset}
Note that if we are integrating $z\mapsto f(z)\dist(z,\p D)^{-1/8}$ where $f$ is a continuous function on $\overline D$ vanishing on a neighbourhood of a subset $U$ of $\p D$, we may only assume that $\dimM(\p D\setminus U)<15/8$ since on the many squares $Q_i$ close to $U\subset\p D$, $\int_{Q_i}f(z)\dist(z,\p D)^{-1/8}dz=0$.
\end{remark}

Before stating our two crucial lemmas for the tightness of the discrete magnetization fields, let us recall the notations we use. If $D\subset\Omega$ are two bounded domains, let $I_\Omega(D)$ be given by \begin{align*}
I_\Omega(D):=\iint_{D\times D}G_\Omega(u,v)\dist(u,\p D)^{-1/8}\dist(v,\p D)^{-1/8}dudv \,.
\end{align*}
(This is the integral of Equation \eqref{e.I_Om} appearing in Section \ref{s.RoughBoundary}.) Now, we have the following lemmas:

\begin{lemma}\label{l.int-finite-Green-Dirichlet}
Let $D$ be any Jordan domain. Then the integral $I_D(D)$ defined above is finite. 
\end{lemma}

\begin{lemma}\label{l.int-finite-smoothBdry}
Let $D$ be a Jordan domain whose closure is contained in a larger Jordan domain $\Omega$. Under the assumption that $\overline{\dim}_M(\p D)<15/8$, the integral $I_\Omega(D)$  is finite. ($\overline{\dim}_M(S)$ denotes the upper Minkowski dimension of a set $S$).
\end{lemma}

Let us prove first Lemma \ref{l.int-finite-smoothBdry} since it is less technical.

\begin{proof}[Proof of Lemma \ref{l.int-finite-smoothBdry}: $I_\Omega(D)<\infty$ under the assumption $\overline{\dim}_M(\p D)<15/8$]
Replacing the minima in \eqref{eq.bounds-Green} by $1$, what we need to show is that \begin{align*}
\iint_{D\times D}\max\big(1,\log(1/|x-y|)\big)d(x)^{-1/8}d(y)^{-1/8}<+\infty
\end{align*} where $d$ is the function $\dist(\cdot,\p D)$. 
\par Let $(Q_i)$ be the Whitney decomposition for $D$. One can see that \begin{align*}\iint_{Q_i\times Q_j}\big(1\vee\log(1/|x-y|)\big)d(x)^{-1/8}d(y)^{-1/8}dxdy \lesssim \log\big(1/\ell_i\vee 1/\ell_j\big)\ell_i^{2-1/8}\ell_j^{2-1/8}
\end{align*}
where $\vee$ stands for the maximum of two terms, and we want to show that the sum over pairs $(i,j)$ of these integrals converges.
\par The first term $\log\big(1/\ell_i\vee 1/\ell_j\big)$ can be upper bounded by $\log(1/\ell_i)\log(1/\ell_j)$ as soon as $\ell_i,\ell_j< 1$, which is true for all but finitely many $Q_i$ (since our domain is bounded). Therefore, we can write \begin{align*}
\sum_{i,j}\log\big(1/\ell_i\vee 1/\ell_j\big)\ell_i^{2-1/8}\ell_j^{2-1/8} \lesssim \Big(\sum_i\ell_i^{2-1/8}\log(1/\ell_i)\Big)^2
\end{align*}
and we only need to show that $\sum_i\ell_i^{15/8}\log(1/\ell_i)$ is finite. 
\par On the other hand, by the hypothesis on $\overline{\dim}_M(\p D)$, there is $\alpha<15/8$ such that \begin{align*}
\# \{i\md \ell_i = 2^{-k}\} = o(2^{\alpha k})\text{    as }k\to\infty.
\end{align*}
Therefore, $\sum_i\ell_i^{15/8}\log(1/\ell_i)\lesssim\sum_{k\ge0}2^{\alpha k}2^{-15k/8}k$ which is obviously a convergent series. Whence the integral $I_\Omega(D)$ must be finite.
\end{proof}

\begin{proof}[Proof of Lemma \ref{l.int-finite-Green-Dirichlet}: $I_D(D)<\infty$ for any Jordan domain]
This time, we want to use the fact that $G_D(x,y)$ goes to zero polynomially in the distance of $x$ or $y$ to the boundary of $D$, and more specifically Inequality \eqref{eq.bounds-Green} with the terms $\sqrt{\frac{d(x)}{|x-y|}}$ and $\sqrt{\frac{d(y)}{|x-y|}}$. (In what follows, $d$ is the distance to the boundary of $D=\Omega$.) This additional argument allows us to deal with the possible roughness of $\p D$.

\par Again, let $(Q_i)$ be the Whitney decomposition for $D$. We aim to prove that the sum over $(i,j)$ of the integrals \begin{align*}
I(i,j) := \iint_{Q_i\times Q_j} G_D(x,y)d(x)^{-1/8}d(y)^{-1/8}dxdy
\end{align*} is finite. For indices $i,j$ such that $i=j$ or $Q_i$ touches $Q_j$, a relation that we denote by $i\leftrightarrow j$, we can bound $I(i,j)$ using $G_D(x,y)=O\big(\log(1/|x-y|)\big)$ which gives us, \begin{align*}
I(i,j) \lesssim \ell_i^{2-1/8}\ell_j^{2-1/8}\log\big(1/\ell_i\vee 1/\ell_j\big),
\end{align*}
so that the sum over $i,j$ with $i\leftrightarrow j$ can be rewritten as
\begin{align*}
\sum_{\substack{i,j \\ i\leftrightarrow j}}I(i,j) \lesssim \sum_{\substack{i,j \\ i\leftrightarrow j}} \ell_i^{2-1/8}\ell_j^{2-1/8}\log\big(1/\ell_i\vee 1/\ell_j\big)\lesssim \sum_i\ell_i^{4-1/4}\log(1/\ell_i).
\end{align*}
(The second inequality relies on the fact that $\ell_i\asymp\ell_j$ for $i\leftrightarrow j$, and that $\#\{j: i\leftrightarrow j\}$ is uniformly bounded in $i$ by $12^2+1$.) Finally, since by trivial area considerations one has $\#\{i: \ell_i=2^{-k}\}=O(2^{2k})$, the last sum above is dominated by $\sum_k2^{2k}2^{-k(4-1/4)}k=\sum_k k2^{-k(2-1/4)}<+\infty$. Therefore, one may restrict the sum $\sum_{i,j}I(i,j)$ to those pairs $i,j$ for which $i\not\leftrightarrow j$.
\par Now, we use the refined bound \begin{align*}
G_D(x,y)= O\Big(d(x)^{1/2}d(y)^{1/2}\log(1/{|x-y|})|x-y|^{-1}\Big)
\end{align*} to get (denoting $\alpha=1/2-1/8>0$ and $d_{ij}=\dist(Q_i,Q_j)\ge \max(\ell_i,\ell_j)/4$) \begin{align*}
I(i,j) \lesssim \ell_i^{2+\alpha}\ell_j^{2+\alpha}\log(1/d_{ij})/d_{ij}.
\end{align*}
For a fixed index $i$ and fixed integers $m,n$, remark that \begin{align*}
J_{n,m}^i:=\big\{j \md 1/2\le 2^md_{ij}<1, \ell_j=2^{-n}, j\not\leftrightarrow i\big\}
\end{align*}
satisfies $\# J_{n,m}^i= O\big(2^{2(n-m)}1_{n+1\ge m}\big)$, again by area considerations. Therefore, one can combine the previous estimates to compute for a fixed index $i$ \begin{align*}
\sum_{\substack{j \\ j\not\leftrightarrow i}} I(i,j) \ell_i^{-(2+\alpha)}\lesssim \sum_{\substack{n,m \\ n+1\ge m}}\sum_{j\in J_{n,m}^i}\ell_j^{2+\alpha}\log(1/d_{ij})/d_{ij}\lesssim\sum_{\substack{n,m \\ n\ge m}}2^{2(n-m)}2^{-n(2+\alpha)}m2^m.
\end{align*}
The last sum above can be seen to be finite: $\sum_{n\ge m}2^{2(n-m)}2^{-n(2+\alpha)}m2^m\le2\sum_m2^{-m}2^{-m\alpha}m<\infty$, so that we can finally write \begin{align*}
\sum_{\substack{i,j \\ i\not\leftrightarrow j}} I(i,j) \lesssim \sum_i \ell_i^{2+\alpha}\lesssim\sum_n2^{2n}2^{-(2+\alpha)n}=\sum_n2^{-\alpha n}<+\infty
\end{align*} once again since $\#\{i\md \ell_i=2^{-n}\}=O(2^{2n})$. This concludes the proof: $I_D(D)=\sum_{i,j}I(i,j)<\infty$, which is what we wanted to show.
\end{proof}

\section{Useful results on joint convergence in law}
Now, we extend Appendix B of \cite{PS} regarding the continuity of conditional expectations.
\begin{lemma}\label{l. conditional expectation of the limit}
    Let $(X_n,Y_n,Z_n)$ be a tight sequence of random variables living in Polish spaces, such that the sequence $(X_n)_{n\ge1}$ is real-valued and uniformly integrable and almost surely $Z_n=\E[X_n\md Y_n]$.
    \par Then for any sublimit $(X,Y,Z)$, $Z$ is integrable and one has $\E[Z\md Y]=\E[X\md Y]$ almost surely; in particular, if $Z$ is $Y$-measurable, $Z=\E[X\md Y]$ almost surely.
\end{lemma}

\begin{proof} Since $(X_n)$ is uniformly integrable (u.i.), $(Z_n)$ is u.i. as well. Hence, $Z$ must be integrable.
\par Let $f$ be a continuous bounded function defined on $\caY$, the Polish space where $Y$ lives. One has $\E[f(Y_n)Z_n]=\E[f(Y_n)X_n]$ a.s. for any $n$, so, by u.i., $\E[f(Y)Z]=\E[f(Y)X]$. This being true for any continuous bounded $f$, one must have $\E[Z\md Y]=\E[X\md Y]$.
\end{proof}

\begin{lemma}\label{l.RN-derivative-L1}
Let $(X_n,Y_n)$ converge in law to $(X,Y)$, where the $X_n$ and $X$ live in a Polish space, and $(Y_n)$ is a uniformly integrable family of positive random variables, $Y_n$ being $X_n$-measurable and $\E[Y]>0$. Let $\mu_n$ be the law of $X_n$, which converges to $\mu$, where $X\sim\mu$, and let $\mu_n'$ be the probability distribution given by $\d\mu'_n(x_n)\propto Y_n(x_n)\d\mu_n(x_n)$. Then $\mu'_n$ converges weakly to $\mu'$, a probability distribution absolutely continuous with respect to $\mu$, the Radon-Nikodym derivative being \begin{equation}
\frac{\d\mu'}{\d\mu}(x) \propto \E[Y\md X](x)
\end{equation} In particular, if $Y$ is $X$-measurable, the Radon-Nikodym derivative is $Y$.
\end{lemma}

\begin{proof} Of course, we can assume that $\E[Y_n]=1$ for all $n$. For convenience, let $Y_X=\E[Y\md X]$.
\par For any continuous bounded function $f$ defined on the state space of $X$, we can compute $\mu'_n(f)= \mu_n(fY_n)=\E[f(X_n)Y_n]\xrightarrow[n\to+\infty]{}\E[f(X)Y]=\E[f(X)Y_X]=\mu'(f)$, whence the result.
\end{proof}

\begin{lemma}\label{l. if two limits are measurable with respect to a third, then they are uniquely coupled}
    Let $(X_n,Y_n,Z_n)$ be a sequence of triples of random variables such that $(X_n,Y_n)$ converges in law to $(X,Y)$, $(Y_n,Z_n)$ converges in law to $(Y,Z)$, and $X$ and $Z$ are $Y$-measurable. Then $(X_n,Y_n,Z_n)$ converges in law to $(X,Y,Z)$.
\end{lemma}

\begin{proof}
Let $\phi,\psi$ be the measurable functions such that $X=\phi(Y)$ and $Z=\psi(Y)$. The sequence $(X_n,Y_n,Z_n)_n$ is tight and for any sublimit $(X',Y',Z')$, we must have $X'=\phi(Y'), Z'=\psi(Y')$ and $Y'\sim Y$ so that $(X',Y',Z')$ is uniquely determined. Hence the sequence converges.
\end{proof}

\bibliographystyle{alpha}
\bibliography{references-HL}

@misc{feng2026nearcriticalisinginterfacesthermal,
      title={{Near-critical Ising interfaces in the thermal direction: Absolute continuity with respect to SLE$_3$}}, 
      author={Yu Feng},
      year={2026},
      eprint={2610.32860},
      archivePrefix={arXiv},
      primaryClass={math.PR},
      url={https://arxiv.org/abs/2610.32860}, 
}

@book{kuratowski1968topology2,
  title={Topology},
  volume={II},
  author={Kuratowski, Kazimierz},
  year={1968},
  publisher={Academic Press \& PWN-Polish Scientific Publishers},
  address={New York--London--Warsaw}
}

@article{duminil2011connection,
  title={{Connection probabilities and RSW-type bounds for the two-dimensional FK Ising model}},
  author={Duminil-Copin, Hugo and Hongler, Cl{\'e}ment and Nolin, Pierre},
  journal={Communications on pure and applied mathematics},
  volume={64},
  number={9},
  pages={1165--1198},
  year={2011},
  publisher={Wiley Online Library}
}

@article{gubinelli2024fbsde,
  title={{The FBSDE approach to sine-Gordon up to $6\pi$}},
  author={Gubinelli, Massimiliano and Meyer, Sarah-Jean},
  journal={arXiv preprint arXiv:2401.13648},
  year={2024}
}

@article{garban2023statistical,
  title={{Statistical reconstruction of the GFF and KT transition}},
  author={Garban, Christophe and Sep{\'u}lveda, Avelio},
  journal={Journal of the European Mathematical Society},
  volume={26},
  number={2},
  pages={639--694},
  year={2023}
}

@article{aru2021reconstructing,
  title={Reconstructing the base field from imaginary multiplicative chaos},
  author={Aru, Juhan and Junnila, Janne},
  journal={Bulletin of the London Mathematical Society},
  volume={53},
  number={3},
  pages={861--870},
  year={2021},
  publisher={Wiley Online Library}
}

@article{ALS2,
	title={{The First Passage Sets of the 2D Gaussian Free Field: Convergence and Isomorphisms}},
	author={Aru, Juhan and Lupu, Titus and Sep{\'u}lveda, Avelio},
	journal={Communications in Mathematical Physics},
	volume={375},
	number={3},
	pages={1885--1929},
	year={2020},
	publisher={Springer}
}

@article{ALS4,
	title={Excursion decomposition of the 2D continuum GFF},
	author={Aru, Juhan and Lupu, Titus and Sep{\'u}lveda, Avelio},
	journal={arXiv preprint arXiv:2304.03150},
	year={2023}
}

@article{bauerschmidt2023coleman,
  title={{The Coleman correspondence at the free fermion point}},
  author={Bauerschmidt, Roland and Webb, Christian},
  journal={Journal of the European Mathematical Society},
  volume={26},
  number={9},
  pages={3137--3241},
  year={2023}
}

@article{bauerschmidt2025twisted,
  title={{Twisted Dirac operators and fractional correlations of the massless sine-Gordon model at the free fermion point}},
  author={Bauerschmidt, Roland and Mason, Scott and Webb, Christian},
  journal={arXiv preprint arXiv:2508.14806},
  year={2025}
}

@inproceedings{duminil2022planar,
  title={Planar random-cluster model: scaling relations},
  author={Duminil-Copin, Hugo and Manolescu, Ioan},
  booktitle={Forum of Mathematics, Pi},
  volume={10},
  pages={e23},
  year={2022},
  organization={Cambridge University Press}
}

@article{chen2026critical,
  title={{Critical exponents for planar random-cluster model with cluster-weight $q= 4$}},
  author={Chen, Hong-Bin and Duminil-Copin, Hugo and He, Tiancheng and Jacopin, Fran{\c{c}}ois and Krachun, Dmitry and Manolescu, Ioan and Xia, Jiaming},
  journal={arXiv preprint arXiv:2605.30030},
  year={2026}
}

@article{wu2018alternating,
  title={Alternating arm exponents for the critical planar Ising model},
  author={Wu, Hao},
  journal={The Annals of Probability},
  volume={46},
  number={5},
  pages={2863--2907},
  year={2018},
  publisher={JSTOR}
}

@article{boutillier2010critical,
  title={{The critical $Z$-invariant Ising model via dimers: the periodic case}},
  author={Boutillier, C{\'e}dric and De Tili{\`e}re, B{\'e}atrice},
  journal={Probability Theory and related fields},
  volume={147},
  number={3},
  pages={379--413},
  year={2010},
  publisher={Springer}
}

@article{boutillier2011critical,
  title={{The critical $Z$-invariant Ising model via dimers: locality property}},
  author={Boutillier, C{\'e}dric and De Tili{\`e}re, B{\'e}atrice},
  journal={Communications in mathematical physics},
  volume={301},
  number={2},
  pages={473--516},
  year={2011},
  publisher={Springer}
}

@article{giuliani2012scaling,
  title={{The scaling limit of the energy correlations in non-integrable Ising models}},
  author={Giuliani, Alessandro and Greenblatt, Rafael L and Mastropietro, Vieri},
  journal={Journal of mathematical physics},
  volume={53},
  number={9},
  year={2012},
  publisher={AIP Publishing}
}

@article{hongler2013energy,
  title={{The energy density in the planar Ising model}},
  author={Hongler, Cl{\'e}ment and Smirnov, Stanislav},
  year={2013}
}

@article{chelkak2021correlations,
  title={{Correlations of primary fields in the critical Ising model}},
  author={Chelkak, Dmitry and Hongler, Cl{\'e}ment and Izyurov, Konstantin},
  journal={arXiv preprint arXiv:2103.10263},
  year={2021}
}

@article{izyurov2024energy,
  title={{Energy correlations in the critical Ising model on a torus}},
  author={Izyurov, Konstantin and Kemppainen, Antti and Tuisku, Petri},
  journal={The Annals of Applied Probability},
  volume={34},
  number={2},
  pages={1699--1729},
  year={2024},
  publisher={Institute of mathematical statistics}
}

@book{friedli2017statistical,
  title={{Statistical mechanics of lattice systems: a concrete mathematical introduction}},
  author={Friedli, Sacha and Velenik, Yvan},
  year={2017},
  publisher={Cambridge University Press}
}

@article{bass1994probabilistic,
  title={Probabilistic Techniques in Analysis},
  author={Bass, Richard F},
  year={1994},
  publisher={Springer}
}

@book{grafakos2008classical,
  title={{Classical Fourier analysis}},
  author={Grafakos, Loukas and others},
  volume={2},
  year={2008},
  publisher={Springer}
}

@misc{park2019massivescalinglimitising,
      title={{Massive Scaling Limit of the Ising Model: Subcritical Analysis and Isomonodromy}}, 
      author={S. C. Park},
      year={2019},
      eprint={1811.06636},
      archivePrefix={arXiv},
      primaryClass={math.PR},
      url={https://arxiv.org/abs/1811.06636}, 
}

@article{benjamini1999noise,
  title={{Noise sensitivity of Boolean functions and applications to percolation}},
  author={Benjamini, Itai and Kalai, Gil and Schramm, Oded},
  journal={Publications Math{\'e}matiques de l'Institut des Hautes {\'E}tudes Scientifiques},
  volume={90},
  number={1},
  pages={5--43},
  year={1999},
  publisher={Springer}
}

@article{park2025near,
  title={{Near-critical Ising, sine-Gordon at the free fermion point, and bosonization}},
  author={Park, SC and Virtanen, Tuomas and Webb, Christian},
  journal={arXiv preprint arXiv:2512.11304},
  year={2025}
}

@article{hairer2013solving,
  title={{Solving the KPZ equation}},
  author={Hairer, Martin},
  journal={Annals of mathematics},
  pages={559--664},
  year={2013},
  publisher={JSTOR}
}

@article{matetski2021kpz,
  title={{The KPZ fixed point}},
  author={Matetski, Konstantin and Quastel, Jeremy and Remenik, Daniel},
  journal={Acta Mathematica},
  volume={227},
  number={1},
  pages={115--203},
  year={2021},
  publisher={Lehigh University Bethlehem, Penn., USA}
}

@article{dauvergne2022directed,
  title={{The directed landscape}},
  author={Dauvergne, Duncan and Ortmann, Janosch and Vir{\'a}g, B{\'a}lint},
  journal={Acta Mathematica},
  volume={229},
  number={2},
  pages={201--285},
  year={2022},
  publisher={Lehigh University Bethlehem, Penn., USA}
}

@article{ganguly2024stability,
  title={{Stability and chaos in dynamical last passage percolation}},
  author={Ganguly, Shirshendu and Hammond, Alan},
  journal={Communications of the American Mathematical Society},
  volume={4},
  number={09},
  pages={387--479},
  year={2024}
}

@article{himwich2024directed,
  title={The directed landscape is a black noise},
  author={Himwich, Zoe and Parekh, Shalin},
  journal={arXiv preprint arXiv:2404.16801},
  year={2024}
}

@article{ahlberg2026noise,
  title={Noise sensitivity in last-passage percolation},
  author={Ahlberg, Daniel and Hillairet, Malo and Toropova, Ekaterina},
  journal={arXiv preprint arXiv:2601.05361},
  year={2026}
}

@article{gubinelli2017kpz,
  title={{KPZ reloaded}},
  author={Gubinelli, Massimiliano and Perkowski, Nicolas},
  journal={Communications in Mathematical Physics},
  volume={349},
  number={1},
  pages={165--269},
  year={2017},
  publisher={Springer}
}

@article{lupu2016loop,
  title={FROM LOOP CLUSTERS AND RANDOM INTERLACEMENTS TO THE FREE FIELD},
  author={Lupu, Titus},
  journal={The Annals of Probability},
  volume={44},
  number={3},
  pages={2117--2146},
  year={2016}
}

@book{garban2014noise,
  title={Noise sensitivity of Boolean functions and percolation},
  author={Garban, Christophe and Steif, Jeffrey E},
  volume={5},
  year={2014},
  publisher={Cambridge University Press}
}

@incollection{garban2011oded,
  title={{Oded Schramm's contributions to noise sensitivity}},
  author={Garban, Christophe},
  booktitle={Selected Works of Oded Schramm},
  pages={287--350},
  year={2011},
  publisher={Springer}
}

@article{lupu2025relation,
  title={Relation between {W}ick powers and excursion clusters of the {2D GFF}},
  author={Lupu, Titus},
  journal={arXiv preprint arXiv:2509.01797},
  year={2025}
}

@incollection{duminil2017lectures,
  title={{Lectures on the Ising and Potts models on the hypercubic lattice}},
  author={Duminil-Copin, Hugo},
  booktitle={PIMS-CRM Summer School in Probability},
  pages={35--161},
  year={2017},
  publisher={Springer}
}

@article{dubedat2009sle,
  title={{SLE and the free field: partition functions and couplings}},
  author={Dub{\'e}dat, Julien},
  journal={Journal of the American Mathematical Society},
  volume={22},
  number={4},
  pages={995--1054},
  year={2009}
}

@article{schramm2013contour,
  title={{A contour line of the continuum Gaussian free field}},
  author={Schramm, Oded and Sheffield, Scott},
  journal={Probability Theory and Related Fields},
  volume={157},
  number={1},
  pages={47--80},
  year={2013},
  publisher={Springer}
}

@article{lupu2018convergence,
  title={{Convergence of the two-dimensional random walk loop-soup clusters to CLE}},
  author={Lupu, Titus},
  journal={Journal of the European Mathematical Society},
  volume={21},
  number={4},
  pages={1201--1227},
  year={2018}
}

@article{aru2019bounded,
  title={{On bounded-type thin local sets of the two-dimensional Gaussian free field}},
  author={Aru, Juhan and Sep{\'u}lveda, Avelio and Werner, Wendelin},
  journal={Journal of the Institute of Mathematics of Jussieu},
  volume={18},
  number={3},
  pages={591--618},
  year={2019},
  publisher={Cambridge University Press}
}

@article{garban2025energy,
  title={{Energy field of critical Ising model and examples of singular fields in QFT}},
  author={Garban, Christophe and Kupiainen, Antti},
  journal={arXiv preprint arXiv:2502.02554},
  year={2025}
}

@article{chelkak2012universality,
  title={{Universality in the 2D Ising model and conformal invariance of fermionic observables}},
  author={Chelkak, Dmitry and Smirnov, Stanislav},
  journal={Inventiones mathematicae},
  volume={189},
  number={3},
  pages={515--580},
  year={2012},
  publisher={Springer}
}

@article{chelkak2014convergence,
  title={Convergence of {I}sing interfaces to {S}chramm's $\mathrm{SLE}$ curves},
  author={Chelkak, Dmitry and Duminil-Copin, Hugo and Hongler, Cl{\'e}ment and Kemppainen, Antti and Smirnov, Stanislav},
  journal={Comptes Rendus. Math{\'e}matique},
  volume={352},
  number={2},
  pages={157--161},
  year={2014}
}

@article{chelkak2023universality,
  title={Universality of spin correlations in the Ising model on isoradial graphs},
  author={Chelkak, Dmitry and Izyurov, Konstantin and Mahfouf, R{\'e}my},
  journal={The Annals of Probability},
  volume={51},
  number={3},
  pages={840--898},
  year={2023},
  publisher={Institute of mathematical statistics}
}

@article{Camia_2015,
   title={Planar {I}sing magnetization field {I}. {U}niqueness of the critical scaling limit},
   volume={43},
   ISSN={0091-1798},
   url={http://dx.doi.org/10.1214/13-AOP881},
   DOI={10.1214/13-aop881},
   number={2},
   journal={The Annals of Probability},
   publisher={Institute of Mathematical Statistics},
   author={Camia, Federico and Garban, Christophe and Newman, Charles M.},
   year={2015}}

@misc{kemp/smirnov_ii,
      title={{Conformal invariance in random cluster models. II. Full scaling limit as a branching SLE}}, 
      author={Antti Kemppainen and Stanislav Smirnov},
      year={2019},
      eprint={1609.08527},
      archivePrefix={arXiv},
      primaryClass={math-ph},
      url={https://arxiv.org/abs/1609.08527}, 
}

@article{Furlan_2017,
   title={A tightness criterion for random fields, with application to the Ising model},
   volume={22},
   ISSN={1083-6489},
   url={http://dx.doi.org/10.1214/17-EJP121},
   DOI={10.1214/17-ejp121},
   number={none},
   journal={Electronic Journal of Probability},
   publisher={Institute of Mathematical Statistics},
   author={Furlan, Marco and Mourrat, Jean-Christophe},
   year={2017},
   month=jan }

@article{holden2023convergence,
  title={{Convergence of uniform triangulations under the Cardy embedding}},
  author={Holden, Nina and Sun, Xin},
  journal={Acta Mathematica},
  volume={230},
  number={1},
  pages={93--203},
  year={2023},
  publisher={International Press of Boston}
}

@article{garban2013pivotal,
  title={Pivotal, cluster, and interface measures for critical planar percolation},
  author={Garban, Christophe and Pete, G{\'a}bor and Schramm, Oded},
  journal={Journal of the American Mathematical Society},
  volume={26},
  number={4},
  pages={939--1024},
  year={2013}
}

@article{chelkak2015conformal,
  title={Conformal invariance of spin correlations in the planar {I}sing model},
  author={Chelkak, Dmitry and Hongler, Cl{\'e}ment and Izyurov, Konstantin},
  journal={Annals of mathematics},
  pages={1087--1138},
  year={2015},
  publisher={JSTOR}
}

@article{benoist2019scaling,
author = {St{\'e}phane Benoist and Cl{\'e}ment Hongler},
title = {{The scaling limit of critical Ising interfaces is $\mathrm{CLE}_{3}$}},
volume = {47},
journal = {The Annals of Probability},
number = {4},
publisher = {Institute of Mathematical Statistics},
pages = {2049 -- 2086},
year = {2019},
doi = {10.1214/18-AOP1301},
URL = {https://doi.org/10.1214/18-AOP1301}
}

@article{griffiths1970concavity,
  title={Concavity of magnetization of an Ising ferromagnet in a positive external field},
  author={Griffiths, Robert B and Hurst, Charles A and Sherman, Seymour},
  journal={Journal of Mathematical Physics},
  volume={11},
  number={3},
  pages={790--795},
  year={1970},
  publisher={American Institute of Physics}
}

@inproceedings{miller2017cle,
  title={{CLE} percolations},
  author={Miller, Jason and Sheffield, Scott and Werner, Wendelin},
  booktitle={Forum of Mathematics, Pi},
  volume={5},
  pages={e4},
  year={2017},
  organization={Cambridge University Press}
}

@article{kohler2024fuzzy,
  title={{The fuzzy Potts model in the plane: scaling limits and arm exponents}},
  author={K{\"o}hler-Schindler, Laurin and Lehmkuehler, Matthis},
  journal={Probability Theory and Related Fields},
  pages={1--73},
  year={2024},
  publisher={Springer}
}

@article{Schramm_2011,
   title={On the scaling limits of planar percolation},
   volume={39},
   ISSN={0091-1798},
   url={http://dx.doi.org/10.1214/11-AOP659},
   DOI={10.1214/11-aop659},
   number={5},
   journal={The Annals of Probability},
   publisher={Institute of Mathematical Statistics},
   author={Schramm, Oded and Smirnov, Stanislav and Garban, Christophe},
   year={2011},
   month=sep }

@article{papon2024interfacescalinglimitcritical,
  title={Interface scaling limit for the critical planar Ising model perturbed by a magnetic field},
  author={Papon, L{\'e}onie},
  journal={arXiv preprint arXiv:2411.16452},
  year={2024}
}

@article{camia2016planar,
  title={Planar {I}sing magnetization field {II}. {P}roperties of the critical and near-critical scaling limits},
  author={Camia, Federico and Garban, Christophe and Newman, Charles M},
  year={2016}
}

@article{lopez2026ising,
  title={{The Ising magnetisation field and the Gaussian free field}},
  author={L{\'o}pez, Tom{\'a}s Alcalde and Heeney, Lorca and Lis, Marcin},
  journal={arXiv preprint arXiv:2602.05886},
  year={2026}
}

@article{AlS,
	title={{Excursion decomposition of the XOR-Ising model}},
	author={L{\'o}pez, Tom{\'a}s Alcalde and Sep{\'u}lveda, Avelio},
	journal={arXiv preprint arXiv:2602.06011},
	year={2026}
}

@article{DCLQ1,
	title={{Conformal invariance of double random currents I: Identification of the limit}},
	author={Duminil-Copin, Hugo and Lis, Marcin and Qian, Wei},
	journal={Proceedings of the London Mathematical Society},
	volume={130},
	number={1},
	pages={e70022},
	year={2025},
	publisher={Wiley Online Library}
}

@article{DCLQ2,
	title={{Conformal invariance of double random currents II: tightness and properties in the discrete}},
	author={Duminil-Copin, Hugo and Lis, Marcin and Qian, Wei},
	journal={arXiv preprint arXiv:2107.12880},
	year={2021}
}

@article{JSW,
	title={Imaginary multiplicative chaos},
	author={Junnila, Janne and Saksman, Eero and Webb, Christian},
	journal={The Annals of Applied Probability},
	volume={30},
	number={5},
	pages={2099--2164},
	year={2020},
	publisher={JSTOR}
}

@article{MM,
	author = {Jean-Fran{\c{c}}ois Marckert and Abdelkader Mokkadem},
	title = {{Limit of normalized quadrangulations: The Brownian map}},
	volume = {34},
	journal = {The Annals of Probability},
	number = {6},
	publisher = {Institute of Mathematical Statistics},
	pages = {2144 -- 2202},
	year = {2006},
	doi = {10.1214/009117906000000557},
	URL = {https://doi.org/10.1214/009117906000000557}
}

@article{LG,
	author = {Jean-Fran{\c{c}}ois Le Gall},
	title = {{Uniqueness and universality of the Brownian map}},
	volume = {41},
	journal = {The Annals of Probability},
	number = {4},
	publisher = {Institute of Mathematical Statistics},
	pages = {2880 -- 2960},
	year = {2013},
	doi = {10.1214/12-AOP792},
	URL = {https://doi.org/10.1214/12-AOP792}
}

@article{Mier,
	author = {Gr{\'e}gory Miermont},
	title = {{The Brownian map is the scaling limit of uniform random plane quadrangulations}},
	volume = {210},
	journal = {Acta Mathematica},
	number = {2},
	publisher = {Institut Mittag-Leffler},
	pages = {319 -- 401},
	year = {2013},
	doi = {10.1007/s11511-013-0096-8},
	URL = {https://doi.org/10.1007/s11511-013-0096-8}
}

@article{Sheffield,
	author = {Scott Sheffield},
	title = {{Quantum gravity and inventory accumulation}},
	volume = {44},
	journal = {The Annals of Probability},
	number = {6},
	publisher = {Institute of Mathematical Statistics},
	pages = {3804 -- 3848},
	year = {2016},
	doi = {10.1214/15-AOP1061},
	URL = {https://doi.org/10.1214/15-AOP1061}
}

@article{DMS,
	title={Liouville quantum gravity as a mating of trees},
	author={Duplantier, Bertrand and Miller, Jason R and Sheffield, Scott},
	journal={Ast{\'e}risque},
	volume={427},
	year={2021}
}

@article{MSBrownianMap,
	title={{An axiomatic characterization of the Brownian map}},
	author={Miller, Jason and Sheffield, Scott},
	journal={Journal de l’{\'E}cole polytechnique—Math{\'e}matiques},
	volume={8},
	pages={609--731},
	year={2021}
}

@article{sheffield2012conformal,
  title={Conformal loop ensembles: the Markovian characterization and the loop-soup construction},
  author={Sheffield, Scott and Werner, Wendelin},
  journal={Annals of Mathematics},
  pages={1827--1917},
  year={2012},
  publisher={JSTOR}
}

@article{sheffield2009exploration,
  title={Exploration trees and conformal loop ensembles},
  author={Sheffield, Scott},
  journal={Duke Math. J.},
  volume={146},
  number={1},
  pages={79--129},
  year={2009}
}

@inproceedings{chelkak2018planar,
  title={{Planar Ising model at criticality: state-of-the-art and perspectives}},
  author={Chelkak, Dmitry},
  booktitle={Proceedings of the International Congress of Mathematicians: Rio de Janeiro 2018},
  pages={2801--2828},
  year={2018},
  organization={World Scientific}
}

@article{comerford2013caratheodory,
  title={{The Carath{\'e}odory topology for multiply connected domains I}},
  author={Comerford, Mark},
  journal={Central European Journal of Mathematics},
  volume={11},
  number={2},
  pages={322--340},
  year={2013},
  publisher={Springer}
}

@misc{duminilcopin2026gaussianfreefieldconvergence,
      title={{Gaussian free field convergence of the six-vertex model with $-1\leq\Delta\leq-\frac12$}}, 
      author={Hugo Duminil-Copin and Karol Kajetan Kozlowski and Piet Lammers and Ioan Manolescu},
      year={2026},
      eprint={2603.06268},
      archivePrefix={arXiv},
      primaryClass={math-ph},
      url={https://arxiv.org/abs/2603.06268}, 
}

@unpublished{alpha1,
    author = {Chen, Hong-Bin and Duminil-Copin, Hugo and He, Tiancheng and Jacopin, Fran{\c{c}}ois and Krachun, Dmitry and Manolescu, Ioan and Xia, Jiaming},
    title = {One-arm exponent for the critical planar random-cluster model with cluster weight $q\in[1,4]$},
    note = {Manuscript in preparation},
}

@article{QW,
	title={Decomposition of Brownian loop-soup clusters.},
	author={Qian, Wei and Werner, Wendelin},
	journal={Journal of the European Mathematical Society (EMS Publishing)},
	volume={21},
	number={10},
	pages={3225},
	year={2019}
}

@article{PS,
	title={{An elementary approach to quantum length of SLE}},
	author={Powell, Ellen and Sep{\'u}lveda, Avelio},
	journal={Probability and Mathematical Physics},
	volume={7},
	number={1},
	pages={175--215},
	year={2025},
	publisher={Mathematical Sciences Publishers}
}

\end{document}